\documentclass[11pt]{article}
\usepackage{fullpage}
\usepackage{amsmath,amssymb,amsthm,graphicx} 
\usepackage{epsfig}
\usepackage{psfrag}
\usepackage{hyperref}

 \hypersetup{
  colorlinks,
  citecolor=blue,
  linkcolor=violet
  }
\usepackage{enumerate} 
\usepackage{float} 
\usepackage{color}
\usepackage{array}
\usepackage{pgf,tikz}
\usepackage{mathrsfs}
\usepackage{subcaption}

\theoremstyle{plain}

\newtheorem{theo}{Theorem}[section]

\newtheorem{lem}{Lemma}[section]
\newtheorem{prop}{Proposition}[section]
\newtheorem{cor}{Corollary}[section]

\theoremstyle{definition} 

\newtheorem{nota}{Notation}[section]
\newtheorem{de}{Definition}[section]
\newtheorem{exa}{Example}[section]
\newtheorem{as}{Assumption}[section]
\newtheorem{alg}{Algorithm}[section]

\newcommand{\btheo}{\begin{theo}}
\newcommand{\bde}{\begin{de}}
\newcommand{\ble}{\begin{lem}}
\newcommand{\bpr}{\begin{prop}}
\newcommand{\bno}{\begin{nota}}
\newcommand{\bex}{\begin{exa}}
\newcommand{\bcor}{\begin{cor}}
\newcommand{\spro}{\begin{proof}}
\newcommand{\bas}{\begin{as}}
\newcommand{\balg}{\begin{alg}}

\newcommand{\etheo}{\end{theo}}
\newcommand{\ede}{\end{de}}
\newcommand{\ele}{\end{lem}}
\newcommand{\epr}{\end{prop}}
\newcommand{\eno}{\end{nota}}
\newcommand{\eex}{\end{exa}}
\newcommand{\ecor}{\end{cor}}
\newcommand{\fpro}{\end{proof}}
\newcommand{\eas}{\end{as}}
\newcommand{\ealg}{\end{alg}}

\theoremstyle{plain}

\newtheorem{theos}{Theorem}
\newtheorem{props}{Proposition}
\newtheorem{lems}{Lemma}
\newtheorem{cors}{Corollary}

\theoremstyle{definition}
\newtheorem{exas}{Example}
\newtheorem{algs}{Algorithm}
\newtheorem{asss}{Assumption}
\newtheorem{defns}{Definition}

\newcommand{\btheos}{\begin{theos}}
\newcommand{\etheos}{\end{theos}}
\newcommand{\bprops}{\begin{props}}
\newcommand{\eprops}{\end{props}}
\newcommand{\bdes}{\begin{defns}}
\newcommand{\edes}{\end{defns}}
\newcommand{\blems}{\begin{lems}}
\newcommand{\elems}{\end{lems}}
\newcommand{\bcors}{\begin{cors}}
\newcommand{\ecors}{\end{cors}}
\newcommand{\bexs}{\begin{exas}}
\newcommand{\eexs}{\end{exas}}
\newcommand{\balgs}{\begin{algs}}
\newcommand{\ealgs}{\end{algs}}
\newcommand{\bass}{\begin{asss}}
\newcommand{\eass}{\end{asss}}

\newcommand{\numobs}{\ensuremath{n}}
\newcommand{\usedim}{\ensuremath{d}}

\newcommand{\mprob}{\ensuremath{\mathbb{P}}}

\newcommand{\ConeSet}{\ensuremath{\Kcone}}

\newcommand{\Width}{\ensuremath{\mathbb{W}}}

 \newcommand{\yvec}{\ensuremath{y}}

\newcommand{\real}{\ensuremath{\mathbb{R}}}
\newcommand{\defn}{\ensuremath{: \, =}}

\newcommand{\inprod}[2]{\ensuremath{\langle #1 , \, #2 \rangle}}

\newcommand{\Sphere}[1]{\SPHERE{#1}}

\newcommand{\Exs}{\ensuremath{\mathbb{E}}}

\newcommand{\NORMAL}{\ensuremath{N}}

\long\def\comment#1{}

\newcommand{\SPHERE}[1]{\ensuremath{S^{#1}}}

\newcommand{\ltwo}[1]{\ensuremath{\|#1\|_2}}
\newcommand{\LinSpace}{\ensuremath{L}}

\newcommand{\HACKPROOF}{\begin{proof}}
\newcommand{\HACKENDPROOF}{\end{proof}}

\newcommand{\var}{\ensuremath{\operatorname{var}}}

\newcommand{\cov}{\ensuremath{\operatorname{cov}}}

\newcommand{\Ball}{\ensuremath{\mathbb{B}}}

\newcommand{\qprob}{\ensuremath{\mathbb{Q}}}

\newcommand{\Ind}{\ensuremath{\mathbb{I}}}

\newcommand{\widgraph}[2]{\includegraphics[keepaspectratio,width=#1]{#2}}

\newlength{\widebarargwidth}
\newlength{\widebarargheight}
\newlength{\widebarargdepth}

\makeatletter
\long\def\@makecaption#1#2{
        \vskip 0.8ex
        \setbox\@tempboxa\hbox{\small {\bf #1:} #2}
        \parindent 1.5em  %% How can we use the global value of this???
        \dimen0=\hsize
        \advance\dimen0 by -3em
        \ifdim \wd\@tempboxa >\dimen0
                \hbox to \hsize{
                        \parindent 0em
                        \hfil 
                        \parbox{\dimen0}{\def\baselinestretch{0.96}\small
                                {\bf #1.} #2
                                } 
                        \hfil}
        \else \hbox to \hsize{\hfil \box\@tempboxa \hfil}
        \fi
        }
\makeatother

\newcommand{\Prob}{\mathbb{P}}
\newcommand{\tvnorm}[1]{\ensuremath{\|#1\|_{\mbox{\tiny{TV}}}}}

\newcommand{\Cone}{C}
\newcommand{\Conestar}{{\Cone^*}}
\newcommand{\PROJ}{\Pi_{\Cone}}

\newcommand{\ProjKcone}{\Pi_{\ConeSet}}
\newcommand{\ProjKconeStar}{\Pi_{\ConeSet^*}}
\newcommand{\Kpos}{\ConeSet_{+}}

\newcommand{\projLperp}{\Pi_{L^\perp}}

\newcommand{\epsopt}{\epsilon_{\text{{\tiny OPT}}}}
\newcommand{\epsglrt}{\epsilon_{\text{{\tiny GLR}}}}
\newcommand{\testglrt}{\phi}
\newcommand{\Diff}{Z(\theta)}
\newcommand{\DiffPlus}{Z^+(\theta)}
\newcommand{\DiffNeg}{Z^-(\theta)}
\newcommand{\starK}{\ConeSet^*}

\newcommand{\Mon}{\ensuremath{M}}

\newcommand{\upconst}{B_\rho}
\newcommand{\loconst}{b_\rho}
\newcommand{\lowconst}{\loconst}
\newcommand{\Cir}[1]{\text{Circ}_{#1}}
\newcommand{\Prc}{\ConeSet_{\times}}
\newcommand{\projPrc}{\Pi_{\ConeSet_{\times}}}
\newcommand{\Proj}{\ensuremath{\Pi}}
\newcommand{\ProjC}{\ensuremath{\Pi_C}}
\newcommand{\ProjCstar}{\ensuremath{\Pi_{C^*}}}
\newcommand{\Ccone}{\ensuremath{C}}

\newcommand{\Kcone}{\ensuremath{K}}
\newcommand{\KconeStar}{\ensuremath{\Kcone^*}}
\newcommand{\Aevent}{\ensuremath{\mathcal{A}}}
\newcommand{\Bevent}{\ensuremath{\mathcal{B}}}
\newcommand{\Mevent}{\ensuremath{\mathcal{C}}}
\newcommand{\Sevent}{\ensuremath{\mathcal{D}}}

\newcommand{\TERMONE}{T_1} \newcommand{\TERMTWO}{T_2}
\newcommand{\TERMTHREE}{T_3} 

\newcommand{\EYE}[1]{\ensuremath{I_{#1}}}
\newcommand{\LinSpacePerp}{{\LinSpace^\perp}}

\newcommand{\SPEC}{\ensuremath{A}}
\newcommand{\OldLmat}{\ensuremath{F}}
\newcommand{\OldAmat}{\ensuremath{G}}
\newcommand{\bprim}{\ensuremath{b'}}

\newcommand{\TEST}[3]{\ensuremath{\mathcal{T}(#1, #2; #3)}}

\newcommand{\UNIERR}{\ensuremath{\mathcal{E}}}
\newcommand{\EPSCRITSQ}{\ensuremath{\delta^2_{\mbox{\tiny{LR}}}
    (\ConeSmall, \ConeBig)}}

\newcommand{\EPSCRIT}{\ensuremath{\delta_{\mbox{\tiny{LR}}}
    (\ConeSmall, \ConeBig)}}

\newcommand{\NEWEPSCRITSQ}{\ensuremath{\delta^2_{\mbox{\tiny{LR}}}
    (\{0\}, \Kcone)}}

\newcommand{\EPSGEN}{\ensuremath{\delta_{\mbox{\tiny{OPT}}} (\ConeSmall,
    \ConeBig)}}
\newcommand{\EPSGENSQ}{\ensuremath{\delta^2_{\mbox{\tiny{OPT}}}(\ConeSmall,
    \ConeBig)}}
\newcommand{\EPSGENSQK}{\ensuremath{\delta^2_{\mbox{\tiny{OPT}}}(\{0\},
    \Kcone)}}

\newcommand{\psiglrt}{\ensuremath{\testglrt_\beta}}

\newcommand{\NewBall}{\ensuremath{B}}

\newcommand{\Hyp}{\ensuremath{\mathcal{H}}}
\newcommand{\HypTil}{\ensuremath{\widetilde{\Hyp}}}

\newcommand{\ProjK}{\ensuremath{\Pi_\Kcone}}

\newcommand{\newthresh}{\ensuremath{\tau}}
\newcommand{\psinew}{\ensuremath{\phi_\tau}}

\newcommand{\infolowconst}{\ensuremath{\kappa_\rho}}
\newcommand{\dprob}{\mathbb{D}}
\newcommand{\Tan}[2]{\mathcal{T}_{#1}(#2)}

\newcommand{\KposSet}{\mathcal{S}}

\newcommand{\SubSpace}[1]{\ensuremath{S_{#1}}}

\newcommand{\Myset}{\ensuremath{\mathcal{S}}}

\newcommand{\ONES}{\ensuremath{\mathbf{1}}}

\newcommand{\myspan}{\ensuremath{\mbox{span}}}

\newcommand{\ConeSmall}{\Cone_1}
\newcommand{\ConeBig}{\Cone_2}
\newcommand{\MACPROJ}[1]{\ensuremath{\Pi_{#1}}}
\newcommand{\PROJSMALL}{\MACPROJ{\ConeSmall}}
\newcommand{\PROJBIG}{\MACPROJ{\ConeBig}}

\newcommand{\myones}{\ensuremath{\mathbf{1}}}
\newcommand{\RayCone}{\ensuremath{R}}

\newcommand{\pdim}{\ensuremath{p}}
\newcommand{\Poly}{\ensuremath{P}}
\newcommand{\range}{\ensuremath{\operatorname{range}}}

\newcommand{\Convex}{\ensuremath{V}}

\newenvironment{carlist}
 {\begin{list}{$\bullet$}
 {\setlength{\topsep}{0in} \setlength{\partopsep}{0in}
  \setlength{\parsep}{0in} \setlength{\itemsep}{\parskip}
  \setlength{\leftmargin}{0.07in} \setlength{\rightmargin}{0.08in}
  \setlength{\listparindent}{0in} \setlength{\labelwidth}{0.08in}
  \setlength{\labelsep}{0.1in} \setlength{\itemindent}{0in}}}
 {\end{list}}

\newcommand{\bcar}{\begin{carlist}}
\newcommand{\ecar}{\end{carlist}}

\newcommand{\ConeSmallPolar}{\ensuremath{\ConeSmall^*}}

\begin{document}

%%%%%%% Title PAGE %%%%%%%%%%%%%%%%%%%%%%%%%%%%%%%%%%%%%%%%%%%%%%%%%%%

% \begin{frontmatter}

% \title{The geometry of hypothesis testing over convex cones:
%     Generalized likelihood ratio tests and minimax radii
%     }
% \runtitle{The geometry of hypothesis testing over convex cones
% % :Generalized likelihood tests and minimax radii
%     }

% \begin{aug}
% \author{\fnms{Yuting} \snm{Wei} \ead[label=e1]{ytwei@berkeley.edu}}  
% \and 
% \author{\fnms{Martin} \snm{J. Wainwright} \thanksref{t1} \ead[label=e2]{wainwrig@berkeley.edu}} 
% \and
% \author{\fnms{Adityanand} \snm{Guntuboyina} \thanksref{t2} \ead[label=e3]{aditya@stat.berkeley.edu}} 

% \thankstext{t1}{Supported in part by Office of Naval Research MURI
%   grant DOD-002888, Air Force Office of Scientific Research Grant
%   AFOSR-FA9550-14-1-001, and National Science Foundation Grants
%   CIF-31712-23800} 
% \thankstext{t2}{Supported by NSF Grant DMS-1309356, NSF CAREER Grant DMS-16-54589}

% \runauthor{Wei, Y., Wainwright, M. and Guntuboyina, A.}

% \affiliation{University of California at Berkeley}

% \address{Department of Statistics \\ University of California,
%   Berkeley \\ Berkeley, California 94720
%   \\ \printead{e1}\\ \phantom{E-mail:\ }\printead*{e2}\\ \phantom{E-mail:\ }\printead*{e3}
% }

% % \address{451 Evans Hall \\
% % Berkeley, CA 94720\\

% % }

% % \address{421 Evans Hall \\ 
% % Berkeley, CA 94720 \\
% % \printead{e2}
% % }

% % \address{423 Evans Hall\\
% % Berkeley, CA 94720 \\
% % \printead{e3}
% % }

% \end{aug}

\begin{center}

{\bf \LARGE{The geometry of hypothesis testing over convex cones:
    \\ Generalized likelihood tests and minimax radii}}

\vspace*{.2in}

{\large{
\begin{tabular}{ccccc}
Yuting Wei$^\dagger$ &&  Martin J. Wainwright$^{\dagger, \star}$ &&
Adityanand Guntuboyina$^\dagger$
\end{tabular}
}}

\vspace*{.2in} \today

 \begin{tabular}{c}
 Department of Statistics$^\dagger$, and \\ Department of Electrical
 Engineering and Computer Sciences$^\star$ \\ UC Berkeley, Berkeley,
 CA 94720
 \end{tabular}
 \vspace*{.2in}

\begin{abstract}
  We consider a compound testing problem within the Gaussian sequence
  model in which the null and alternative are specified by a pair of
  closed, convex cones.  Such cone testing problem arises in various
  applications, including detection of treatment effects, trend
  detection in econometrics, signal detection in radar processing, and
  shape-constrained inference in non-parametric statistics.  We
  provide a sharp characterization of the GLRT testing radius up to a
  universal multiplicative constant in terms of the geometric
  structure of the underlying convex cones.  When applied to concrete
  examples, this result reveals some interesting phenomena that do not
  arise in the analogous problems of estimation under convex
  constraints.  In particular, in contrast to estimation error, the
  testing error no longer depends purely on the problem complexity via
  a volume-based measure (such as metric entropy or Gaussian
  complexity); other geometric properties of the cones also play an
  important role. In order to address the issue of optimality, we
  prove information-theoretic lower bounds for the minimax testing
  radius again in terms of geometric quantities. Our general theorems
  are illustrated by examples including the cases of monotone and
  orthant cones, and involve some results of independent interest.
\end{abstract}

\end{center}

% \begin{keyword}[class=AMS]
% \kwd[Primary]{62F03}
% \kwd[; secondary ]{52A05}
% \end{keyword}
% \begin{keyword}
% hypothesis testing, closed convex cone, likelihood ratio test, minimax rate, Gaussian complexity
% \end{keyword}
% \end{frontmatter}

%%%%%%%%%%%%%%%%%%%%%%%%%%%%%%%%%%%%%%%%%%%%%%%%%%%%%%%%%%%%%%%%%%%%%%%%%%%%%%%%%%%%%%

%\vspace*{2cm}

\section{Introduction}
\label{SecIntro}

Composite testing problems arise in a wide variety of applications and
the generalized likelihood ratio test (GLRT) is a general purpose
approach to such problems.  The basic idea of the likelihood ratiotest dates back to the early works of Fisher, Neyman and Pearson; it
attracted further attention following the work of
Edwards~\cite{edwards1972wf}, who emphasized likelihood as a general
principle of inference.  Recent years have witnessed a great amount of
work on the GLRT in various contexts, including the
papers~\cite{lehmann2006testing,perlman1999emperor,lehmann2012likelihood,fan2001generalized,fan2007nonparametric}.
However, despite the wide-spread use of the GLRT, its optimality
properties have yet to be fully understood.  For suitably regular
problems, there is a great deal of asymptotic theory on the GLRT, and
in particular when its distribution under the null is independent of
nuisance parameters
(e.g.,~\cite{barlow1972statistical,robertson1978likelihood,raubertas1986hypothesis}).
On the other hand, there are some isolated cases in which the GLRT can
be shown to dominated by other tests
(e.g.,~\cite{warrack1984likelihood,MenSal91,menendez1992dominance,lehmann2012likelihood}).

In this paper, we undertake an in-depth study of the GLRT in
application to a particular class of composite testing problems of a
geometric flavor.  In this class of testing problems, the null and
alternative hypotheses are specified by a pair of closed convex cones
$\ConeSmall$ and $\ConeBig$, taken to be nested as $\ConeSmall \subset
\ConeBig$.  Suppose that we are given an observation of the form $y =
\theta + w$, where $w$ is a zero-mean Gaussian noise vector.  Based on
observing $y$, our goal is to test whether a given parameter $\theta$
belongs to the smaller cone $\ConeSmall$---corresponding to the null
hypothesis---or belongs to the larger cone $\ConeBig$.  Cone testing
problems of this type arise in many different settings, and there is a
fairly substantial literature on the behavior of the GLRT in
application to such problems (e.g., see the papers and
books~\cite{Bes06,kudo1963multivariate,robertson1978testing,raubertas1986hypothesis,
  robertson1986testing,robertson1982testing,meyer2003test,MenSal91,menendnez1992testing,
  dykstra1983testing,shapiro1988towards,warrack1984likelihood}, as
well as references therein).

%%%%%%%%%%%%%%%%%%%%%%%%%%%%%%%%%%%%%%%%%%%%%%%%%%%%%%%%%%%%%%%%%%%%%%%%%%%%%%%%%%

\subsection{Some motivating examples}\label{motex}

Before proceeding, let us consider some concrete examples so as to
motivate our study.

\bexs[Testing non-negativity and monotonicity in treatment effects]
\label{ExaMonotoneTreatment}
Suppose that we have a collection of $\usedim$ treatments, say
different drugs for a particular medical condition.  Letting $\theta_j
\in \real$ denote the mean of treatment $j$, one null hypothesis could
be that none of treatments has any effect---that is, $\theta_j = 0$
for all $j = 1, \ldots, \usedim$.  Assuming that none of the
treatments are directly harmful, a reasonable alternative would be
that $\theta$ belongs to the \emph{non-negative orthant cone}
\begin{align}
  \label{EqnOrthantConeIntro}
\Kpos & \defn \big \{ \theta \in \real^\usedim \mid \theta_j \geq 0
\quad \mbox{for all $j = 1, \ldots, \usedim$} \big \}.
\end{align}
This set-up leads to a particular instance of our general set-up with
$\ConeSmall = \{0 \}$ and $\ConeBig = \Kpos$.  Such orthant testing
problems have been studied by Kudo~\cite{kudo1963multivariate} and
Raubertas et al.~\cite{raubertas1986hypothesis}, among other people.

In other applications, our treatments might consist of an ordered set
of dosages of the same drug.  In this case, we might have reason to
believe that if the drug has any effect, then the treatment means
would obey a monotonicity constraint---that is, with higher dosages
leading to greater treatment effects.  One would then want to detect
the presence or absence of such a dose response effect. Monotonicity
constraints also arise in various types of econometric models, in
which the effects of strategic interventions should be monotone with
respect to parameters such as market size (e.g.,\cite{Che12}).  For
applications of this flavor, a reasonable alternative would be
specified by the \emph{monotone cone}
\begin{align}
  \label{EqnMonotoneConeIntro}
\Mon & \defn \big \{ \theta \in \real^\usedim \, \mid \, \theta_1 \leq
\theta_2 \leq \cdots \leq \theta_\usedim \big \}.
\end{align}
This set-up leads to another instance of our general problem with
$\ConeSmall = \{0\}$ and $\ConeBig = \Mon$.  The behavior of the GLRT
for this particular testing problem has also been studied in past
works, including papers by Barlow et al.~\cite{barlow1972statistical},
and Raubertas et al.~\cite{raubertas1986hypothesis}.

As a third instance of the treatment effects problem, we might like to
include in our null hypothesis the possibility that the treatments
have some (potentially) non-zero effect but one that remains constant
across levels---i.e., $\theta_1 = \theta_2 = \cdots = \theta_\usedim$.
% and is non-negative.  
In this case, our null hypothesis is specified by the \emph{ray cone}
\begin{align}
  \label{EqnRayConeIntro}
  \RayCone & \defn \big \{ \theta \in \real^\usedim \, \mid \, \theta =
  c \myones \quad \mbox{for some $c \in \real$} \big \}.
\end{align}
Supposing that we are interested in testing the alternative that the
treatments lead to a monotone effect, we arrive at another instance of
our general set-up with $\ConeSmall = \RayCone$ and $\ConeBig = \Mon$.
This testing problem has also been studied by
Bartholomew~\cite{bartholomew1959test,bartholomew1959testII} and
Robertson et al.~\cite{robertson1988order} among other researchers.

In the preceding three examples, the cone $C_1$ was linear
subspace. Let us now consider two more examples, adapted from
Menendnez et al.~\cite{menendnez1992testing}, in which $C_1$ is not a
subspace.  As before, suppose that component $\theta_i$ of the vector
$\theta \in \real^d$ denotes the expected response of treatment $i$.
In many applications, it is of interest to test equality of the
expected responses of a subset $S$ of the full treatment set $[d] =
\{1, \ldots, d\}$.  More precisely, for a given subset $S$ containing
the index $1$, let us consider the problem of testing the the null
hypothesis
\begin{align}
\label{EqnConeEgNonLin1}
\Cone_1 \equiv E(S) & \defn \Big \{ \theta \in \real^d \mid \,
\theta_i = \theta_1 \mbox{ $\forall \; i \in S$, and } \theta_j \geq
\theta_1 \mbox{ $\forall \; j \notin S$ } \Big \}
\end{align}
versus the alternative $\Cone_2 \equiv G(S) = \{ \theta \in \real^d
\mid \theta_j \geq \theta_1 \mbox{ $\forall \; j \in [\usedim]$} \}$.
Note that $C_1$ here is not a linear subspace.

As a final example, suppose that we have a factorial design consisting
of two treatments, each of which can be applied at two different
dosages (high and level).  Let $(\theta_1, \theta_2)$ denote the
expected responses of the first treat at the low and high dosages,
respectively, with the pair $(\theta_3, \theta_4)$ defined similarly
for the second treatment.  Suppose that we are interesting in testing
whether the first treatment at the lowest level is more effective than
the second treatment at the highest level.  This problem can be
formulated as testing the null cone
\begin{align}
\label{EqnConeEgNonLin2}
C_1 & \defn \{\theta \in \mathbb{R}^4 \mid \theta_1 \leq \theta_2 \leq
\theta_3\leq \theta_4\} \quad \mbox{versus the alternative} \notag \\
C_2 & \defn \{\theta \in \mathbb{R}^4 \mid \theta_1 \leq \theta_2,
\mbox{ and } ~\theta_3\leq \theta_4\}.
\end{align}
As before, the null cone $C_1$ is not a linear subspace.  

\eexs

\bexs[Robust matched filtering in signal processing]
\label{ExaRobustSignal}
In radar detection problems~\cite{Sch91}, a standard goal is to detect
the presence of a known signal of unknown amplitude in the presence of
noise.  After a matched filtering step, this problem can be reduced to
a vector testing problem, where the known signal direction is defined
by a vector $\gamma \in \real^d$, whereas the unknown amplitude
corresponds to a scalar pre-factor $c \geq 0$.  We thus arrive at a
ray cone testing problem: the null hypothesis (corresponding to the
absence of signal) is given $\ConeSmall = \{0 \}$, whereas the
alternative is given by the positive ray cone $\RayCone_{+} = \big \{\theta \in
\real^\usedim \mid \theta = c \gamma \; \mbox{for some $c \geq 0$}
\big \}$.

In many cases, there may be uncertainty about the target signal, or
jamming by adversaries, who introduce additional signals that can be
potentially confused with the target signal $\gamma$.  Signal
uncertainties of this type are often modeled by various forms of
cones, with the most classical choice being a subspace
cone~\cite{Sch91}.  In more recent work
(e.g.,~\cite{Bes06,GreGinFar08}), signal uncertainty has been modeled
using the \emph{circular cone} defined by the target signal direction,
namely
\begin{align}
  \label{EqnCircularConeIntro}
C(\gamma; \alpha) & \defn \big \{ \theta \in \real^\usedim \, \mid \,
\inprod{\gamma}{\theta} \geq \cos(\alpha) \, \|\gamma\|_2 \|\theta\|_2
\big \},
\end{align}
corresponding to the set of all vectors $\theta$ that have angle at
least $\alpha$ with the target signal.  Thus, we are led to another
instance of a cone testing problem involving a circular cone.

\eexs

\bexs[Cone-constrained testing in linear regression]
\label{ExaLinearRegression}
Consider the
standard linear regression model
  \begin{align}
    y & = X \beta + \sigma Z, \qquad \mbox{where $Z \sim N(0,
        I_{\numobs})$,} 
  \end{align}
  where $X \in \real^{\numobs \times \pdim}$ is a fixed and known
  design matrix.  In many applications, we are interested in testing
  certain properties of the unknown regression vector $\beta$, and
  these can often be encoded in terms of cone-constraints on the
  vector $\theta \defn X \beta$.  As a very simple example, the
  problem of testing whether or not $\beta = 0$ corresponds to testing
  whether $\theta \in \ConeSmall \defn \{0 \}$ versus the alternative
  that $\theta \in \ConeBig \defn \range(X)$.  Thus, we arrive at a
  \emph{subspace testing problem}.  We note this problem is known as
  testing the global null in the linear regression literature
  (e.g.,~\cite{buhlmann2013statistical}).  If instead we consider the
  case when the $\pdim$-dimensional vector $\beta$ lies in the
  non-negative orthant cone~\eqref{EqnOrthantConeIntro}, then our
  alternative for the $\numobs$-dimensional vector $\theta$ becomes
  the \emph{polyhedral cone}
  \begin{align}
    \label{EqnPolyhedralConeIntro}
  \Poly & \defn \big \{ \theta \in \real^\numobs \mid \, \theta = X
  \beta \quad \mbox{for some $\beta \geq 0$} \big \}.
  \end{align}
  The corresponding estimation problem with non-negative constraints
  on the coefficient vector $\beta$ has been studied by Slawski et
  al.~\cite{slawski2013non} and
  Meinshausen~\cite{meinshausen2013sign}; see also Chen et
  al. \cite{chen2009nonnegativity} for a survey of this line of work.
  In addition to these preceding two cases, we can also test various
  other types of cone alternatives for $\beta$, and these are
  transformed via the design matrix $X$ into other types of cones for
  the parameter $\theta \in \real^\numobs$.

  \eexs

\bexs[Testing shape-constrained departures from parametric models]
\label{ExaNonparametric}
Our third example is non-parametric in flavor.  Consider the class of
functions $f$ that can be decomposed as
\begin{align}
  \label{EqnDecomposition}
  f & = \sum_{j=1}^k a_j \phi_j + \psi.
\end{align}
Here the known functions $\{\phi_j \}_{j=1}^k$ define a linear space,
parameterized by the coefficient vector $a \in \real^k$, whereas the
unknown function $\psi$ models a structured departure from this linear
parametric class.  For instance, we might assume that $\psi$ belongs
to the class of monotone functions, or the class of convex functions.
Given a fixed collection of design points $\{t_i\}_{i=1}^\numobs$,
suppose that we make observations of the form \mbox{$y_i = f(t_i) +
  \sigma g_i$} for $i = 1, \ldots, \numobs$, where each $g_i$ is a
standard normal variable.  Defining the shorthand notation $\theta
\defn \big(f(t_1), \ldots, f(t_\numobs) \big)$ and $g = (g_1, \ldots,
g_\numobs)$, our observations can be expressed in the standard form
\mbox{$y = \theta + \sigma g$.}  If, under the null hypothesis, the
function $f$ satisfies the decomposition~\eqref{EqnDecomposition} with
$\psi = 0$, then the vector $\theta$ must belong to the subspace $\{
\Phi a \mid \, a \in \real^k \}$, where the matrix $\Phi \in
\real^{\numobs \times k}$ has entries $\Phi_{ij} = \phi_j(x_i)$.

Now suppose that the alternative is that $f$ satisfies the
decomposition~\eqref{EqnDecomposition} with some $\psi$ that is
convex.  A convexity constraint on $\psi$ implies that we can write
$\theta = \Phi a + \gamma$, for some coefficients $a \in \real^k$ and
a vector $\gamma \in \real^\numobs$ belonging to the \emph{convex cone}
\begin{align}
\Convex(\{t_i\}_{i=1}^\numobs) & \defn \Big \{ \gamma \in
\real^\numobs \mid \frac{\gamma_2 - \gamma_1}{t_2 - t_1} \leq
\frac{\gamma_3 - \gamma_2}{t_3 - t_2} \leq \cdots \leq
\frac{\gamma_\numobs - \gamma_{\numobs-1}}{t_\numobs - t_{\numobs-1}}
\Big \}.
\end{align}
This particular cone testing problem and other forms of shape
constraints have been studied by Meyer~\cite{meyer2003test}, as well
as by Sen and Meyer~\cite{sen2016testing}.
\eexs

%%%%%%%%%%%%%%%%%%%%%%%%%%%%%%%%%%%%%%%%%%%%%%%%%%%%%%%%%%%%%%%%%%%%%%%%%%%%%%%

\subsection{Problem formulation}

Having understood the range of motivations for our problem, let us now
set up the problem more precisely.  Suppose that we are given
observations of the form $y = \theta + \sigma g$, where $\theta \in
\real^\usedim$ is a fixed but unknown vector, whereas $g \sim N(0,
I_{\usedim})$ is a $\usedim$-dimensional vector of i.i.d. Gaussian entries and $\sigma^2$ is a known noise level. 
Our goal is to distinguish the null hypothesis that $\theta \in
\ConeSmall$ versus the alternative that $\theta \in \ConeBig
\backslash \ConeSmall$, where $\ConeSmall \subset \ConeBig$ are a
nested pair of closed, convex cones in $\real^\usedim$.

In this paper, we study both the fundamental limits of solving this
composite testing problem, as well as the performance of a specific
procedure, namely the \emph{generalized likelihood ratio test}, or
GLRT for short.  By definition, the GLRT for the problem of
distinguishing between cones $\ConeSmall$ and $\ConeBig$ is based on
the statistic
\begin{subequations}
\begin{align} 
\label{EqnGLRTstat}
T(\yvec) & \defn - 2 \log \left( \frac{\sup_{\theta \in \ConeSmall}
  \mprob_\theta(\yvec)}{\sup_{\theta \in \ConeBig}
  \mprob_{\theta}(\yvec)} \right).
\end{align}
It defines a family of tests, parameterized by a threshold parameter
\mbox{$\beta \in [0,\infty)$,} of the form
\begin{align} 
\label{EqnGLRT}
\psiglrt(\yvec) & \defn \Ind(T(\yvec) \geq \beta) \; =
\; \begin{cases} 1 & \mbox{if $T(\yvec) \geq \beta$} \\ 0 &
  \mbox{otherwise.}
\end{cases}
\end{align}
\end{subequations}

Thus far, our formulation of the testing problem allows for the
possibility that $\theta$ lies in the set $\ConeBig \backslash
\ConeSmall$, but is arbitrarily close to some element of $\ConeSmall$.
Thus, under this formulation, it is not possible to make any
non-trivial assertions about the power of the GLRT nor any other test
in a uniform sense.  Accordingly, so as to be able to make
quantitative statements about the performance of different statements,
we exclude a certain $\epsilon$-ball from the alternative.  This
procedure leads to the notion of the \emph{minimax testing radius}
associated this composite decision problem.  This minimax formulation
was introduced in the seminal work of Ingster and
co-authors~\cite{ingster1987minimax,ingster2012nonparametric}; since
then, it has been studied by many authors
(e.g.,~\cite{ermakov1991minimax,Spokoiny1998testing,lepski1999minimax,lepski2000asymptotically,baraud2002non}).

For a given $\epsilon > 0$, we define the \emph{$\epsilon$-fattening}
of the cone $\ConeSmall$ as
\begin{align}
\NewBall_2(\ConeSmall; \epsilon) & \defn \big \{ \theta \in \real^\usedim
\, \mid \, \min_{u \in \ConeSmall} \ltwo{\theta-u} \leq \epsilon \big \},
\end{align}
corresponding to the set of vectors in $\real^\usedim$ that are at
most Euclidean distance $\epsilon$ from some element of $\ConeSmall$.
We then consider the testing problem of distinguishing between the
two hypotheses
\begin{align}
  \label{EqnEpsTest}
\Hyp_0: \theta \in \ConeSmall \quad \mbox{and} \quad \Hyp_1: \theta
\in \ConeBig \backslash \NewBall_2(\ConeSmall; \epsilon).
\end{align}
To be clear, the parameter $\epsilon > 0$ is a quantity that is used
during the course of our analysis in order to titrate the difficulty of
the testing problem.  All of the tests that we consider, including the
GLRT, are not given knowledge of $\epsilon$.
Let us introduce shorthand $\TEST{\ConeSmall}{\ConeBig}{\epsilon}$ to denote this testing problem~\eqref{EqnEpsTest}.

Obviously, the testing problem~\eqref{EqnEpsTest} becomes more
difficult as $\epsilon$ approaches zero, and so it is natural to study
this increase in quantitative terms.  Letting $\psi: \real^\usedim 
\rightarrow \{0,1\}$ be any (measurable) test function, we measure its
performance in terms of its \emph{uniform error}
\begin{align}
  \label{EqnUniErr}
  \UNIERR(\psi; \ConeSmall, \ConeBig, \epsilon) & \defn \sup_{\theta \in
    \ConeSmall} \Exs_\theta[\psi(y)] + \sup_{\theta \in
    \ConeBig \backslash \NewBall_2(\epsilon; \ConeSmall)}
    \Exs_\theta[1 - \psi(y)],
\end{align}
which controls the worst-case error over both null and alternative.

For a given error level $\rho \in (0,1)$, we are interested in the
smallest setting of $\epsilon$ for which either the GLRT, or some
other test $\psi$ has uniform error at most $\rho$.  More precisely,
we define
\begin{subequations}
\begin{align}
\label{EqnDefnEpsMinMax}
\epsopt(\ConeSmall, \ConeBig; \rho) & \defn \inf \Big \{ \epsilon \,
\mid \, \inf_{\psi} \UNIERR(\psi; \ConeSmall, \ConeBig, \epsilon) \leq
\rho \Big \}, \quad \text{ and } \\
\label{EqnDefnEpsGLRT}
\epsglrt(\ConeSmall, \ConeBig; \rho) & \defn \inf \Big \{ \epsilon \,
\mid \, \inf_{\beta \in \real} \UNIERR(\psiglrt; \ConeSmall, \ConeBig,
\epsilon) \leq \rho \Big \}.
\end{align}
\end{subequations}
When the subspace-cone pair $(\ConeSmall, \ConeBig)$ are clear from
the context, we occasionally omit this dependence, and write
$\epsopt(\rho)$ and $\epsglrt(\rho)$ instead.  We refer to these two
quantities as the \emph{minimax testing radius} and the \emph{GLRT
  testing radius} respectively.

By definition, the minimax testing radius $\epsopt$ corresponds to the
smallest separation $\epsilon$ at which there exists \emph{some test}
that distinguishes between the hypotheses $\Hyp_0$ and $\Hyp_1$ in
equation~\eqref{EqnEpsTest} with uniform error at most $\rho$.  Thus,
it provides a fundamental characterization of the statistical
difficulty of the hypothesis testing.  On the other hand, the GLRT
testing radius $\epsglrt(\rho)$ provides us with the smallest radius
$\epsilon$ for which there exists \emph{some threshold}---say
$\beta^*$--- for which the associated generalized likelihood ratio
test $\testglrt_{\beta^*}$ distinguishes between the hypotheses with
error at most $\rho$.  Thus, it characterizes the performance limits
of the GLRT when an optimal threshold $\beta^*$ is chosen.  Of course,
by definition, we always have $\epsopt(\rho) \leq \epsglrt(\rho)$.  We
write $\epsopt(\rho) \asymp \epsglrt(\rho)$ to mean that---in addition
to the previous upper bound---there is also a lower bound
$\epsopt(\rho) \geq c_\rho \epsglrt(\rho)$ that matches up to a
constant $c_\rho > 0$ depending only on $\rho$.

%%%%%%%%%%%%%%%%%%%%%%%%%%%%%%%%%%%%%%%%%%%%%%%%%%%%%%%%%%%%%%%%%%%%%%%%%%%%%%

\subsection{Overview of our results}

Having set up the problem, let us now provide a high-level overview of
the main results of this paper.

\bcar
  \item Our first main result, stated as Theorem~\ref{ThmGLRT} in
    Section~\ref{SecGLRTResults}, gives a sharp
    characterization---meaning upper and lower bounds that match up to
    universal constants---of the GLRT testing radius $\epsglrt$ for
    cone pairs $(\ConeSmall, \ConeBig)$ that are non-oblique (we
    discuss the non-obliqueness property and its significance at
    length in Section~\ref{SecNonOblique}).  We illustrate the
    consequences of this theorem for a number of concrete cones,
    include the subspace cone, orthant cone, monotone cone, circular
    cone and a Cartesian product cone.

\item In our second main result, stated as Theorem~\ref{ThmLBGen} in
  Section~\ref{SecGeneralLower}, we derive a lower bound that applies
  to any testing function.  It leads to a corollary that provides
  sufficient conditions for the GLRT to be an optimal test, and we use
  it to establish optimality for the subspace cone and
  circular cone, among other examples.  We then revisit the Cartesian
  product cone, first analyzed in the context of
  Theorem~\ref{ThmGLRT}, and use Theorem~\ref{ThmLBGen} to show that
  the GLRT is sub-optimal for this particular cone, even though it is
  in no sense a pathological example.

\item For the monotone and orthant cones, we find that the lower bound
  established in Theorem~\ref{ThmLBGen} is not sharp, but that the
  GLRT turns out to be an optimal test.  Thus, Section~\ref{SecDetail}
  is devoted to a detailed analysis of these two cases, in particular
  using a more refined argument to obtain sharp lower bounds.
\ecar

The remainder of this paper is organized as follows:
Section~\ref{SecBackground} provides background on conic geometry,
including conic projections, the Moreau decomposition, and the notion
of Gaussian width.  It also introduces the notion of a non-oblique
pair of cones, which have been studied in the context of the GLRT.  In
Section~\ref{SecMain}, we state our main results and illustrate their
consequences via a series of examples.  Sections~\ref{SecGLRTResults}
and~\ref{SecGeneralLower} are devoted, respectively, to our sharp
characterization of the GLRT and a general lower bound on the minimax
testing radius.  Section~\ref{SecDetail} explores the monotone and
orthant cones in more detail.  In Section~\ref{SecProofs}, we provide
the proofs of our main results, with certain more technical aspects
deferred to the appendix sections.

\paragraph{Notation}  Here we summarize some notation used throughout
the remainder of this paper.  For functions $f(\sigma, \usedim)$ and
$g(\sigma, \usedim)$, we write $f(\sigma, \usedim) \lesssim g(\sigma,
\usedim)$ to indicate that $f(\sigma, \usedim) \leq cg(\sigma,
\usedim)$ for some constant $c \in (0, \infty)$ that may only depend on $\rho$ but 
independent of $(\sigma, \usedim)$, and similarly for $f(\sigma,
\usedim) \gtrsim g(\sigma, \usedim)$.  We write $f(\sigma, \usedim)
\asymp g(\sigma, \usedim)$ if both $f(\sigma, \usedim) \lesssim
g(\sigma, \usedim)$ and $f(\sigma, \usedim) \gtrsim g(\sigma,
\usedim)$ are satisfied.

%%%%%%%%%%%%%%%%%%%%%%%%%%%%%%%%%%%%%%%%%%%%%%%%%%%%%%%%%%%%%%%%%%%%%%%%%%%%%%%%%%%%%%

\section{Background on conic geometry and the GLRT}
\label{SecBackground}

In this section, we provide some necessary background on cones and
their geometry, including the notion of a polar cone and the Moreau
decomposition.  We also define the notion of a non-oblique pair of
cones, and summarize some known results about properties of the GLRT
for such cone testing problems.

%%%%%%%%%%%%%%%%%%%%%%%%%%%%%%%%%%%%%%%%%%%%%%%%%%%%%%%%%%%%%%%%%%%%%%%%%%%%%%%%%

\subsection{Convex cones and Gaussian widths}

For a given closed convex cone $\Cone \subset \real^\usedim$, we
define the Euclidean projection operator $\PROJ: \real^\usedim
\rightarrow \Cone$ via
\begin{align}
\PROJ(v) \defn \arg \min_{u \in \Cone} \ltwo{v - u }.
\end{align}
By standard properties of projection onto closed convex sets, we are
guaranteed that this mapping is well-defined.  We also define the
polar cone
\begin{align}
\label{EqnDefnPolar}
\Conestar & \defn \big \{v \in \real^\usedim \mid \inprod{v}{u} \leq 0
\quad \mbox{for all $u \in \Cone$} \big\}.
\end{align}
Figure~\ref{FigCone}(b) provides an illustration of a cone in
comparison to its polar cone.  Using $\ProjCstar$ to denote the
projection operator onto this cone, Moreau's
theorem~\cite{moreau1962decomposition} ensures that every vector $v
\in \real^\usedim$ can be decomposed as
\begin{align}
\label{EqnMoreau}
v & = \ProjC(v) + \ProjCstar(v), \quad \mbox{and such that
  $\inprod{\ProjC(v)}{\ProjCstar(v)} = 0$.}
\end{align}
We make frequent use of this decomposition in our analysis.

Let $\Sphere{} \defn \{ u \in \real^\usedim \mid \ltwo{u} = 1 \}$
denotes the Euclidean sphere of unit radius.  For every set $A \subseteq
\Sphere{}$, we define its \emph{Gaussian width} as 
\begin{align}
  \Width(A) & \defn \Exs \big[\sup_{u\in A} \inprod{u}{g} \big] \qquad
  \mbox{where $g \sim \NORMAL(0, \EYE{\usedim})$.}
\end{align}
This quantity provides a measure of the size of the set $A$; indeed,
it can be related to the volume of $A$ viewed as a subset of the
Euclidean sphere.  The notion of Gaussian width arises in many
different areas, notably in early work on probabilistic methods in
Banach spaces~\cite{pisier1986probabilistic}; the Gaussian complexity,
along with its close relative the Rademacher complexity, plays a
central role in empirical process
theory~\cite{van2000applications,koltchinskii2001rademacher,bartlett2005local}.

Of interest in this paper are the Gaussian widths of sets of the form
$A = \Ccone \cap \Sphere{}$, where $\Ccone$ is a closed convex
cone.  For a set of this form, using the Moreau
decomposition~\eqref{EqnMoreau}, we have the useful equivalence
\begin{align}
\label{EqnWidthEquivalence}
\Width(\Ccone \cap \Sphere{}) \; = \; \Exs \big[\sup_{u\in
    \Ccone \cap \Sphere{}} \inprod{u}{\ProjC(g) +
    \ProjCstar(g)} \big] \; = \; \Exs \ltwo{\ProjC(g)},
\end{align}
where the final equality uses the fact ƒthat $\inprod{u}{\ProjCstar(g)}
\leq 0$ for all vectors $u \in \Ccone$, with equality holding when $u$
is a non-negative scalar multiple of $\ProjC(g)$.

For future reference, let us derive a lower bound on $\Exs
\ltwo{\ProjC g}$ that holds for every cone $\Cone$ strictly larger
than $\{0\}$.  Take some non-zero vector $u \in \Cone$ and let $R_{+} =
\{ c u \mid c \geq 0 \}$ be the ray that it defines.  Since $R_{+}
\subseteq \Cone$, we have $\ltwo{\ProjC g} \geq \ltwo{\Pi_{R_{+}} g}$. But since $R_{+}$ is just a ray,
the projection $\Pi_{R_{+}}(g)$ is a standard normal variable truncated to
be positive, and hence
\begin{align}
\label{EqnConstLower}
\Exs \ltwo{\ProjC g} \ge \Exs \ltwo{\Pi_{R_{+}} g} =
\sqrt{\frac{1}{2\pi}}.
\end{align}
This lower bound is useful in parts of our development.

%%%%%%%%%%%%%%%%%%%%%%%%%%%%%%%%%%%%%%%%%%%%%%%%%%%%%%%%%%%%%%%%%%%%%%%%%%%%%%%

\subsection{Cone-based GLRTs and non-oblique pairs}
\label{SecNonOblique}

In this section, we provide some background on the notion of
non-oblique pairs of cones, and their significance for the GLRT.
First, let us exploit some properties of closed convex cones in order
to derive a simpler expression for the GLRT test
statistic~\eqref{EqnGLRTstat}.  Using the form of the multivariate
Gaussian density, we have
\begin{align}
  \label{EqnMulti}
T(\yvec) \; = \; \min_{\theta \in \ConeSmall} \|y - \theta\|_2^2 -
\min_{\theta \in \ConeBig} \|y - \theta\|_2^2 & = \|y -
\Proj_{\ConeSmall}(y)\|_2^2 - \|y - \Proj_{\ConeBig}(y)\|_2^2\\
& = \|\Proj_{\ConeBig}(y)\|_2^2 - \|\Proj_{\ConeSmall}(y)\|_2^2,
\end{align}
where we have made use of the Moreau decomposition to assert that
\begin{align*}
  \ltwo{y - \Proj_{\ConeSmall}(y)}^2 = \ltwo{y}^2 -
  \ltwo{\Proj_{\ConeSmall}(y)}^2, \quad \mbox{and} \quad \ltwo{y -
    \Proj_{\ConeBig}(y)}^2 = \ltwo{y}^2 -
  \ltwo{\Proj_{\ConeBig}(y)}^2.
\end{align*}
Thus, we see that a cone-based GLRT has a natural interpretation: it
compares the squared amplitude of the projection of $\yvec$ onto the
two different cones.

When $\ConeSmall = \{0 \}$, then it can be shown that under the null
hypothesis (i.e., $y \sim N(0, \sigma^2 I_{\usedim})$), the statistic
$T(\yvec)$ (after rescaling by $\sigma^2$) is a mixture of
$\chi^2$-distributions (see e.g.,~\cite{raubertas1986hypothesis}). On
the other hand, for a general cone pair $(\ConeSmall, \ConeBig)$, it
is not straightforward to characterize the distribution of $T(\yvec)$
under the null hypothesis.  Thus, past work has studied conditions on
the cone pair under which the null distribution has a simple
characterization.  One such condition is a certain non-obliqueness
property that is common to much past work on the GLRT
(e.g.,~\cite{warrack1984likelihood,MenSal91,menendnez1992testing,HuWr94}).
The non-obliqueness condition, first introduced by Warrack et
al.~\cite{warrack1984likelihood}, is also motivated by the fact that
are many instances of oblique cone pairs for which the GLRT is known
to dominated by other tests.  Menendez et
al.~\cite{menendez1992dominance} provide an explanation for this
dominance in a very general context; see also the
papers~\cite{menendnez1992testing, HuWr94} for further studies of
non-oblique cone pairs.

A nested pair of closed convex cones $\ConeSmall \subset \ConeBig$ is
said to be \emph{non-oblique} if we have the successive projection
property
\begin{align}
  \label{EqnNonOblique}
  \PROJSMALL(x) & = \PROJSMALL(\PROJBIG(x)) \qquad \mbox{for all $x
    \in \real^\usedim$.}
\end{align}
For instance, this condition holds whenever one of the two cones is a
subspace, or more generally, whenever there is a subspace $\LinSpace$
such that $\ConeSmall \subseteq \LinSpace \subseteq \ConeBig$; see Hu
and Wright~\cite{HuWr94} for details of this latter property.  To be
clear, these conditions are sufficient---but not necessary---for
non-obliqueness to hold.  There are many non-oblique cone pairs in
which neither cone is a subspace; the cone
pairs~\eqref{EqnConeEgNonLin1} and~\eqref{EqnConeEgNonLin2}, as
discussed in Example~\ref{ExaMonotoneTreatment} on treatment testing,
are two such examples.  (We refer the reader to Section 5 of the
paper~\cite{menendnez1992testing} for verification of these
properties.)  More generally, there are various non-oblique cone pairs
that do not sandwich a subspace $\LinSpace$.

The significance of the non-obliqueness condition lies in the
following decomposition result.  For any nested pair of closed convex
cones $\ConeSmall \subset \ConeBig$ that are non-oblique, for all $x
\in \real^\usedim$ we have
\begin{align}
  \label{EqnZara}
  \Proj_{\ConeBig}(x) & = \Proj_{\ConeSmall}(x) + \Proj_{\ConeBig \cap
    \ConeSmallPolar}(x) \quad \mbox{and} \quad
  \inprod{\Proj_{\ConeSmall}(x)}{\Proj_{\ConeBig \cap
      \ConeSmallPolar}(x)} = 0.
\end{align}
This decomposition follows from general theory due to
Zarantonello~\cite{Zar71}, who proves that for non-oblique cones, we
have $\Proj_{\ConeBig \cap \ConeSmallPolar} = \Proj_{\ConeSmallPolar}
\Proj_{\ConeBig}$---in particular, see Theorem 5.2 in his paper.

An immediate consequence of the decomposition~\eqref{EqnZara} is that
the GLRT for any non-oblique cone pair $(\ConeSmall, \ConeBig)$ can be
written as
\begin{align*}
T(\yvec) = \|\Proj_{\ConeBig}(y)\|_2^2 - \|\Proj_{\ConeSmall}(y)\|_2^2
&\; = \; \|\Proj_{\ConeBig \cap \ConeSmallPolar}(y)\|_2^2 \\
&\; = \; \|y \|_2^2 - \min_{ \theta \in \ConeBig \cap \ConeSmallPolar} \|y - \theta\|_2^2.
\end{align*}
Consequently, we see that the GLRT for the pair $(\ConeSmall,
\ConeBig)$ is equivalent to---that is, determined by the same
statistic as---the GLRT for testing the \emph{reduced hypothesis}
\begin{align}
  \label{EqnSimpleEquiv}
\HypTil_0: \theta = 0 \quad \mbox{versus} \quad \HypTil_1: \theta \in
\big( \ConeBig \cap \ConeSmallPolar \big) \backslash
\NewBall_2(\epsilon).
\end{align}
Following the previous notation, write it as $\TEST{\{0\}}{\ConeBig \cap \ConeSmallPolar}{\epsilon}$ and we make frequent use of this convenient reduction in the sequel.

%%%%%%%%%%%%%%%%%%%%%%%%%%%%%%%%%%%%%%%%%%%%%%%%%%%%%%%%%%%%%%%%%%%%%%%%%%%%%%%%%%%%%%

\section{Main results and their consequences}
\label{SecMain}

We now turn to the statement of our main results, along with a
discussion of some of their consequences.
Section~\ref{SecGLRTResults} provides a sharp characterization of the
minimax radius for the generalized likelihood ratio test up to a
universal constant, along with a number of concrete examples.  In
Section~\ref{SecGeneralLower}, we state and prove a general lower
bound on the performance of any test, and use it to establish the
optimality of the GLRT in certain settings, as well as its
sub-optimality in other settings.  In Section~\ref{SecDetail}, we
revisit and study in details two cones of particular interest, namely
the orthant and monotone cones.

%%%%%%%%%%%%%%%%%%%%%%%%%%%%%%%%%%%%%%%%%%%%%%%%%%%%%%%%%%%%%%%%%%%%%%%%%%%%%%%%%%%%%%

\subsection{Analysis of the generalized likelihood ratio test}
\label{SecGLRTResults}

Let $(\ConeSmall, \ConeBig)$ be a nested pair of closed cones
$\ConeSmall \subseteq \ConeBig$ that are
non-oblique~\eqref{EqnNonOblique}.  Consider the polar cone
$\ConeSmallPolar$ as well as the intersection cone $\Kcone = \ConeBig
\cap \ConeSmallPolar$.  Letting $g \in \real^{\usedim}$ denote a
standard Gaussian random vector, we then define the quantity
\begin{align} 
\label{EqnTightGLRT}
\EPSCRITSQ & \defn \min \Biggr\{ \Exs \ltwo{\ProjK g}, ~~
\Big(\frac{\Exs \ltwo{\ProjK g}} {\max \{ 0, ~\inf \limits_{\eta \in
    \Kcone \cap \Sphere{}} \inprod{\eta}{\Exs\ProjK g} \}}
\Big)^2 \Biggr \}.
\end{align}
Note that $\EPSCRITSQ$ is a purely geometric object, depending on the
pair $(\ConeSmall, \ConeBig)$ via the new cone \mbox{$\Kcone =
  \ConeBig \cap \ConeSmallPolar$,} which arises due to the GLRT
equivalence~\eqref{EqnSimpleEquiv} discussed previously.

Recall that the GLRT is based on applying a threshold, at some level
\mbox{$\beta \in [0, \infty)$,} to the likelihood ratio statistic
  $T(y)$; in particular, see equations~\eqref{EqnGLRTstat}
  and~\eqref{EqnGLRT}.  In the following theorem, we study the
  performance of the GLRT in terms of the the uniform testing error
  $\UNIERR(\psiglrt; \ConeSmall, \ConeBig, \epsilon)$ from
  equation~\eqref{EqnUniErr}.  In particular, we show that the
  critical testing radius for the GLRT is governed by the geometric
  parameter $\EPSCRITSQ$.

\begin{theos} 
  \label{ThmGLRT}
  There are numbers $\{ (\lowconst, \upconst), \rho \in (0,1/2) \}$
  such that for every pair of non-oblique closed convex cones
  $(\ConeSmall,\ConeBig)$ with $\ConeSmall$ strictly contained within
  $\ConeBig$:
 \begin{enumerate}[(a)]
\item For every error probability $\rho \in (0, 0.5)$, we have
\begin{subequations}
\begin{align}
\inf_{\beta \in [0, \infty)} \UNIERR(\psiglrt; \ConeSmall, \ConeBig,
  \epsilon) & \leq \rho \qquad \mbox{for all $\epsilon^2 \geq \upconst
    \; \sigma^2 \; \EPSCRITSQ$.}
\end{align}
\item Conversely, for every error probability $\rho \in (0,0.11]$, we
have
\begin{align}
\inf_{\beta \in [0, \infty)} \UNIERR(\psiglrt; \ConeSmall, \ConeBig,
  \epsilon) & \geq \rho \qquad \mbox{for all $\epsilon^2 \leq \lowconst
    \; \sigma^2 \; \EPSCRITSQ$.}
\end{align}
\end{subequations}
\end{enumerate}
\end{theos}
\paragraph{Remarks} While our proof leads to universal values for
the constants $\upconst$ and $\loconst$, we have made no efforts to
obtain the sharpest possible ones, so do not state them here.  In any
case, our main interest is to understand the scaling of the testing
radius with respect to $\sigma$ and the geometric parameters of the
problem.  In terms of the GLRT testing radius $\epsglrt$ previously
defined~\eqref{EqnDefnEpsGLRT}, Theorem~\ref{ThmGLRT} establishes that
\begin{align}
  \label{EqnThmGLRTSummary}
\epsglrt(\ConeSmall, \ConeBig; \rho) & \asymp \sigma \; \EPSCRIT,
\end{align}
where $\asymp$ denotes equality up to constants depending on $\rho$,
but independent of all other problem parameters.  Since $\epsglrt$
always upper bounds $\epsopt$ for every fixed level $\rho$, we can
also conclude from Theorem~\ref{ThmGLRT} that
\begin{align*}
\epsopt(\ConeSmall, \ConeBig; \rho) & \lesssim \sigma \; \EPSCRIT.
\end{align*}
It is worthwhile noting that the quantity $\EPSCRITSQ$ depends on the
pair $(\ConeSmall, \ConeBig)$ only via the new cone \mbox{$\Kcone =
  \ConeBig \cap \ConeSmallPolar$.}  Indeed, as discussed in
Section~\ref{SecNonOblique}, for any pair of non-oblique closed convex
cones, the GLRT for the original testing problem~\eqref{EqnEpsTest} is
equivalent to the GLRT for the modified testing problem
$\TEST{\{0\}}{\Kcone}{\epsilon}$.

Observe that the quantity $\EPSCRITSQ$ from
equation~\eqref{EqnTightGLRT} is defined via the minima of two terms.
The first term $\Exs \ltwo{\ProjK g}$ is the (square root of the)
Gaussian width of the cone $\Kcone$, and is a familiar quantity from
past work on least-squares estimation involving convex
sets~\cite{van1996weak,chatterjee2014new}.  The Gaussian width measure
the size of the cone $\Kcone$, and it is to be expected that the
minimax testing radius should grow with this size, since $\Kcone$
characterizes the set of possible alternatives.  The second term
involving the inner product $\inprod{\eta}{\Exs\ProjK g}$ is less
immediately intuitive, partly because no such term arises in
estimation over convex sets.  The second term becomes dominant in
cones for which the expectation $v^* \defn \Exs[\ProjK g]$ is
relatively large; for such cones, we can test between the null and
alternative by performing a univariate test after projecting the data
onto the direction $v^*$.  This possibility only arises for cones that
are more complicated than subspaces, since $\Exs[\ProjK g] = 0$ for
any subspace $\Kcone$.

Finally, we note that Theorem 1 gives a sharp characterization of the
behavior of the GLRT up to a constant. It is different from the usual
minimax guarantee.  To the best of our knowledge, it is the first
result to provide tight upper and lower control on the uniform
performance of a specific test.

%%%%%%%%%%%%%%%%%%%%%%%%%%%%%%%%%%%%%%%%%%%%%%%%%%%%%%%%%%%%%%%%%%%%%%%%%%%%%%%%

\subsubsection{Consequences for convex set alternatives}

Although Theorem~\ref{ThmGLRT} applies to cone-based testing problems,
it also has some implications for a more general class of problems
based on convex set alternatives.  In particular, suppose that we are
interested in the testing problem of distinguishing between
\begin{align}
  \label{EqnConvexSetTest}
\Hyp_0: \theta = \theta_0,~~~\text{versus}~~~ \Hyp_1: \theta \in
\Myset,
\end{align}
where $\Myset$ is a not necessarily a cone, but rather an arbitrary
closed convex set, and $\theta_0$ is some vector such that $\theta_0 \in \Myset$.
Consider the tangent cone of $\Myset$ at $\theta_0$, which is given by
\begin{align}
  \Tan{\Myset}{\theta} & \defn \{u \in \real^\usedim \, \mid \,
  \mbox{there exists some $t > 0$ such that $ \theta + tu \in \Myset$}
  \big \}.
\end{align}
Note that $\Tan{\Myset}{\theta_0}$ contains the shifted set $\Myset -
\theta_0$.  Consequently, we have
\begin{align*}
  \UNIERR(\psi; \{0\}, \Myset - \theta_0, \epsilon) & \leq
  \Exs_{\theta=0} [\psi(\yvec)] + \sup_{ \theta \in
    \Tan{\Myset}{\theta_0} \backslash \NewBall_2(0 ; \epsilon) }
  \Exs_\theta[1 - \psi(\yvec)] \; = \; \UNIERR(\psi; \{0\},
  \Tan{\Myset}{\theta_0}, \epsilon),
\end{align*}
which shows that the tangent cone testing problem
\begin{align}
\label{EqnTangentConeTest}
\Hyp_0: \; \theta = 0 \quad \mbox{versus} \quad \Hyp_1: \; \theta \in
\Tan{\Myset}{\theta_0},
\end{align}
is more challenging than the original
problem~\eqref{EqnConvexSetTest}.  Thus, applying
Theorem~\ref{ThmGLRT} to this cone-testing
problem~\eqref{EqnTangentConeTest}, we obtain the following:
\bcors
\label{CorConvexSet}
For the convex set testing problem~\eqref{EqnConvexSetTest}, we have
\begin{align}
  \epsopt^2(\theta_0,\Myset;\rho) \lesssim \sigma^2 \min \Biggr\{ \Exs
  \ltwo{\Pi_{ \Tan{\Myset}{\theta_0}} g}, ~~ \Big(\frac{\Exs
    \ltwo{\Pi_{ \Tan{\Myset}{\theta_0}} g}} {\max \{ 0, ~\inf
    \limits_{\eta \in \Tan{\Myset}{\theta_0} \cap \Sphere{}}
    \inprod{\eta}{\Exs \Pi_{ \Tan{\Myset}{\theta_0}} g} \}} \Big)^2
  \Biggr \}.
\end{align}
This upper bound can be achieved by applying the GLRT to the tangent
cone testing problem~\eqref{EqnTangentConeTest}.  
\ecors

This corollary offers a general recipe of upper bounding the optimal
testing radius.  In Subsection~\ref{SecMonotoneCone}, we provide an
application of Corollary~\ref{CorConvexSet} to the problem of testing
\begin{align*}
  \Hyp_0 : \theta = \theta_0 ~~~\text{versus}~~~ \Hyp_1: \theta \in \Mon,
\end{align*}
where $\Mon$ is the monotone cone (defined in expression
\eqref{EqnMonotoneConeIntro}). When $\theta_0 \neq 0$, this is not a
cone testing problem, since the set $\{\theta_0\}$ is not a cone.
Using Corollary~\ref{CorConvexSet}, we prove an upper bound on the
optimal testing radius for this problem in terms of the number of
constant pieces of $\theta_0$.

% \vspace*{0.2in}

In the remainder of this section, we consider some special cases of
testing a cone $\Kcone$ versus $\{0\}$ in order to illustrate the
consequences of Theorem~\ref{ThmGLRT}.  In all cases, we compute the
GLRT testing radius for a constant error probability, and so ignore the
dependencies on $\rho$.  For this reason, we adopt the more
streamlined notation $\epsglrt(\Kcone)$ for the radius
$\epsglrt(\{0\}, \Kcone; \rho)$.

%%%%%%%%%%%%%%%%%%%%%%%%%%%%%%%%%%%%%%%%%%%%%%%%%%%%%%%%%%%%%%%%%%%%%%%%%%%%%

\subsubsection{Subspace of dimension $k$}
\label{SecSubspaceCone}

Let us begin with an especially simple case---namely, when $\Kcone$ is
equal to a subspace $\SubSpace{k}$ of dimension \mbox{$k \leq
  \usedim$}.  In this case, the projection $\ProjK$ is a linear
operator, which can be represented by matrix multiplication using a
rank $k$ projection matrix.  By symmetry of the Gaussian distribution,
we have $\Exs[\ProjK g] = 0$.  Moreover, by rotation invariance of the
Gaussian distribution, the random vector $\|\ProjK g\|_2^2$ follows a
$\chi^2$-distribution with $k$ degrees of freedom, whence
\begin{align*}
\frac{\sqrt{k}}{2} \; \leq \; \Exs \ltwo{\ProjK g} \; \leq \;
\sqrt{\Exs \ltwo{\ProjK g}^2} \; = \; \sqrt{k}.
\end{align*}
  Applying Theorem~\ref{ThmGLRT} then yields that the testing radius
  of the GLRT scales as
\begin{align} 
\label{EqnGLRTsubspace}
\epsglrt^2(\SubSpace{k}) & \asymp \sigma^2 \sqrt{k}.
\end{align}
Here our notation $\asymp$ denotes equality up to constants
independent of $(\sigma, k)$; we have omitted dependence on the
testing error $\rho$ so as to simplify notation, and will do so
throughout our discussion.

%%%%%%%%%%%%%%%%%%%%%%%%%%%%%%%%%%%%%%%%%%%%%%%%%%%%%%%%%%%%%%%%%%%%%%%%%%%%%%%

\subsubsection{Circular cone}
\label{SecCircularCone}

A circular cone in $\real^\usedim$ with constant angle $ \alpha \in
(0, \pi/2)$ is given by $ \Cir{\usedim}(\alpha) \defn \{ \theta \in
\real^\usedim \mid \theta_1 \geq \ltwo{\theta} \cos(\alpha) \}$.  In
geometric terms, it corresponds to the set of all vectors whose angle
with the standard basis vector $e_1 = (1, 0, \ldots, 0)$ is at most
$\alpha$ radians.  Figure~\ref{FigCone}(a) gives an illustration of a
circular cone.

\begin{figure}[H]
  \begin{center}
  \begin{tabular}{ccc}
    \widgraph{0.4\textwidth}{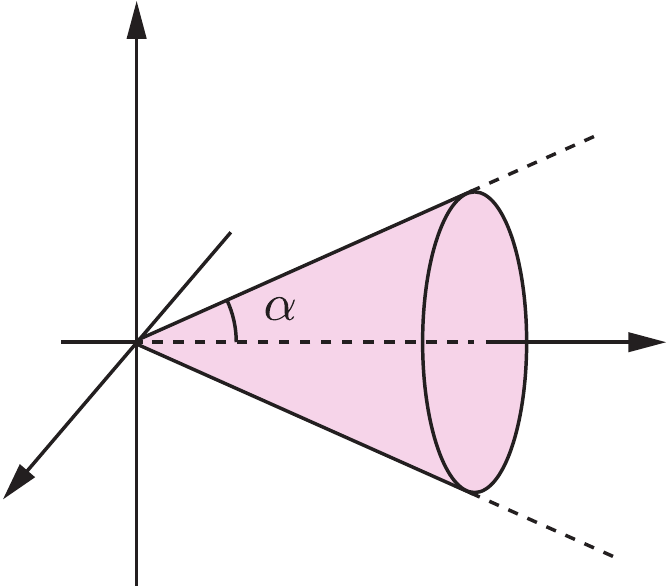} & & 
    \widgraph{0.38\textwidth}{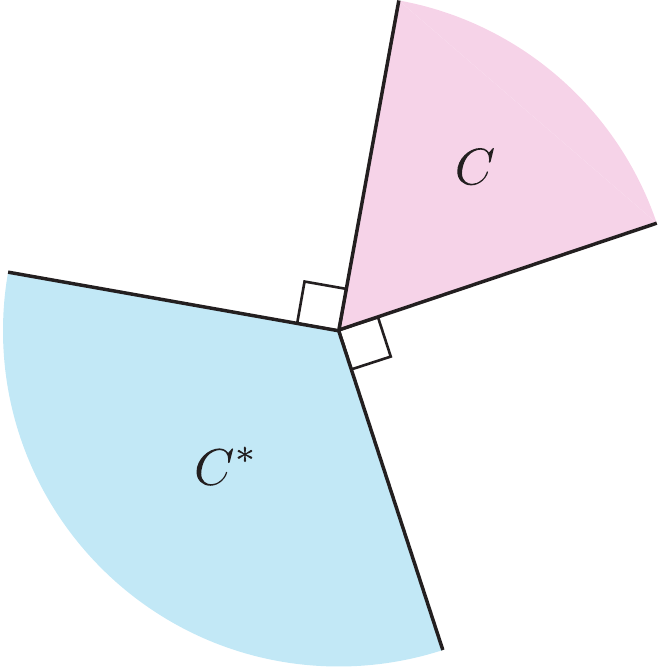} \\
    (a) & & (b)
  \end{tabular}
    \caption{(a) A $3$-dimensional circular cone with angle $\alpha$.
      (b) Illustration of a cone versus its polar cone.  }
    \label{FigCone}
  \end{center}
\end{figure}

 Suppose that we want to test the null hypothesis
$\theta = 0$ versus the cone alternative $\Kcone =
\Cir{\usedim}(\alpha)$.  We claim that, in application to this
particular cone, Theorem~\ref{ThmGLRT} implies that
\begin{align}
\label{EqnGLRTcircular} 
\epsglrt^2(\Kcone) \asymp \sigma^2 \min \big \{
\sqrt{\usedim}, 1 \big \} \; = \; \sigma^2,
\end{align}
where $\asymp$ denotes equality up to constants depending on $(\rho,
\alpha)$, but independent of all other problem parameters.

In order to apply Theorem~\ref{ThmGLRT}, we need to
evaluate both 
terms that define the geometric quantity $\EPSCRITSQ$.  On one hand,
by symmetry of the cone $\Kcone = \Cir{\usedim}(\alpha)$ in its last
$(\usedim-1)$-coordinates, we have $\Exs \ProjK g = \beta e_1$ for
some scalar $\beta > 0$ and $e_1$ denotes the standard Euclidean
basis vector with a $1$ in the first coordinate.  
Moreover, for any $\eta \in \Kcone \cap
\Sphere{}$, we have $\eta_1 \geq \cos(\alpha)$, and hence
\begin{align*}
  \inf_{\eta \in \Kcone \cap \Sphere{}} \inprod{\eta}{\Exs
    \ProjK g} = \eta_1\beta \geq \cos(\alpha) \beta \; = \; \cos(\alpha) \|\Exs
  \ProjK g\|_2.
\end{align*}
Next, we claim that $\ltwo{\Exs \ProjK g} \asymp \Exs \ltwo{\ProjK
  g}$.  In order to prove this claim, note that Jensen's inequality
yields
\begin{align}
\label{EqnCir-Sec}
  \Exs \ltwo{\ProjK g} \geq \ltwo{\Exs\ProjK g } 
  ~\stackrel{(a)}{\geq}~ (\Exs \Pi_{\Cir{\usedim}(\alpha)} g)_1
  = \Exs (\Pi_{\Cir{\usedim}(\alpha)} g)_1
  ~\stackrel{(b)}{\geq}~  \Exs \ltwo{\Pi_{\Cir{\usedim}(\alpha)} g}\cos(\alpha),
\end{align}
where in this argument, inequality (a) follows from simply fact that $\ltwo{x} \geq |x_1|$ whereas inequality (b) follows from the definition of circular cone.
Plugging into definition $\EPSCRITSQ$, the corresponding second term 
equals to a constant. 
Therefore, the second term in the definition~\eqref{EqnTightGLRT} of
$\EPSCRITSQ$ is upper bounded by a constant, independent of the
dimension $\usedim$.

On the other hand, from known results on circular cones~(see \S
6.3,~\cite{mccoy2014steiner}), there are constants $\kappa_j =
\kappa_j(\alpha)$ for $j=1,2$ such that $\kappa_1 \usedim ~\leq~ \Exs
\ltwo{\ProjK g}^2 ~\leq~ \kappa_2 \usedim$.  Moreover, we have
 \begin{align*}
      % \label{EqnCircSandwich}
  \Exs \ltwo{\ProjK g}^2 - 4 ~\stackrel{(a)}{\leq}~ (\Exs \ltwo{\ProjK
    g})^2 ~\stackrel{(b)}{\leq}~ \Exs \ltwo{\ProjK g}^2.
\end{align*}
 Here inequality (b) is an immediate consequence of Jensen's
 inequality, whereas inequality (a) follows from the fact that
 $\var(\ltwo{\ProjK g}) \leq 4$---see Lemma~\ref{LemConcentration} in
 Section~\ref{SecProofThmGLRT} and the surrounding discussion for
 details.  Putting together the pieces, we see that \mbox{$\Exs
   \ltwo{\ProjK g} \asymp \sqrt{\usedim}$} for the circular cone.
 Combining different elements of our argument leads to the stated
 claim~\eqref{EqnGLRTcircular}.

%%%%%%%%%%%%%%%%%%%%%%%%%%%%%%%%%%%%%%%%%%%%%%%%%%%%%%%%%%%%%%%%%%%%%%%%%%%%%%

\subsubsection{A Cartesian product cone}
\label{SecCartesianProductCone}

We now consider a simple extension of the previous two
examples---namely, a convex cone formed by taking the Cartesian
product of the real line $\real$ with the circular cone
$\Cir{\usedim-1}(\alpha)$---that is
\begin{align}
\label{EqnDefProdCone}
\Prc \defn \Cir{\usedim-1}(\alpha) \times \real.
\end{align}
Please refer to Figure~\ref{fig:procone} as an illustration of this
cone in three dimensions.
\begin{figure}[h]
  \centering
  \includegraphics[width=0.45\textwidth]{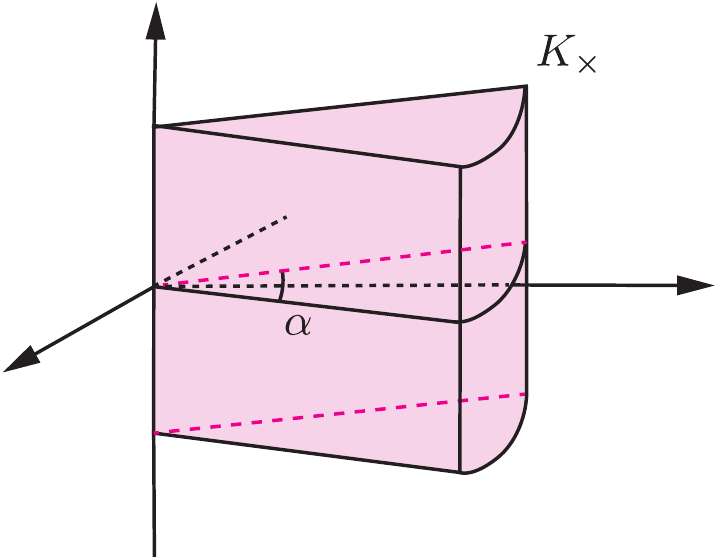}
  \caption{Illustration of the product cone defined in
    equation~\eqref{EqnDefProdCone}.}
  \label{fig:procone}
\end{figure}

This example turns out to be rather interesting because---as will be
demonstrated in Section~\ref{SecCartesianProductRevisit}---the GLRT is
sub-optimal by a factor $\sqrt{\usedim}$ for this cone.  In order to
set up this later analysis, here we use Theorem~\ref{ThmGLRT} to prove
that
\begin{align}
  \label{EqnDomGLRT}
  \epsglrt^2(\Prc) \asymp \sigma^2 \sqrt{\usedim}.
\end{align}
Note that this result is strongly suggestive of sub-optimality on the
part of the GLRT.  More concretely, the two cones that form $\Prc$ are
both ``easy'', in that the GLRT radius scales as $\sigma^2$ for each.
For this reason, one would expect that the squared radius of an
optimal test would scale as $\sigma^2$---as opposed to the $\sigma^2
\sqrt{\usedim}$ of the GLRT---and our later calculations will show
that this is indeed the case.

We now prove claim~\eqref{EqnDomGLRT} as a consequence of
Theorem~\ref{ThmGLRT}.  
First notice that projecting to the product cone $\Prc$ can be viewed as projecting the first $\usedim-1$ dimension to circular cone $\Cir{\usedim-1}(\alpha)$ and the last coordinate to $\real.$
Consequently, we have the following inequality
\begin{align*}
  \Exs \ltwo{\Pi_{\Cir{\usedim-1}(\alpha)} g} ~\leq~ \Exs
  \ltwo{\projPrc g} & \stackrel{(a)}{\leq}~ \sqrt{\Exs \ltwo{\projPrc
      g}^2} \\
& = \sqrt{\Exs \ltwo{\Pi_{\Cir{\usedim-1}(\alpha)} g}^2 + \Exs
       [g_d^2 ]}.
\end{align*}
where inequality (a) follows by Jensen's inequality.
Making use of our previous calculations for the circular cone, we have $\Exs \ltwo{\projPrc g} \asymp \sqrt{\usedim}.$
Moreover, note that the last coordinate of $\Exs [\projPrc g]$ is equal to $0$ by symmetry and the standard basis vector $e_\usedim
\in \real^\usedim$, with a single one in its last coordinate, belongs
to $\Prc \cap \Sphere{}$, we have
\begin{align*}
\inf_{\eta \in \Prc \cap \Sphere{}} \inprod{\eta}{\Exs \projPrc
  (g)} & \leq \inprod{e_\usedim}{\Exs \projPrc (g)} = 0.
\end{align*}
Plugging into definition $\EPSCRITSQ$, the corresponding second term
equals infinity. Therefore, the minimum that defines $\EPSCRITSQ$ is
achieved in the first term, and so is proportional to
$\sqrt{\usedim}$. Putting together the pieces yields the
claim~\eqref{EqnDomGLRT}.

%%%%%%%%%%%%%%%%%%%%%%%%%%%%%%%%%%%%%%%%%%%%%%%%%%%%%%%%%%%%%%%%%%%%%%%%%%%%%%%

\subsubsection{Non-negative orthant cone}
\label{SecOrthantCone}

Next let us consider the (non-negative) orthant cone given by $\Kpos
\defn \big \{ \theta \in \real^\usedim \mid \theta_j \geq 0 \quad
\mbox{for $j = 1, \ldots, \usedim$} \big \}$.  Here we use
Theorem~\ref{ThmGLRT} to show that
\begin{align}
\label{EqnGLRTkpos}
\epsglrt^2(\Kpos) \asymp \sigma^2 \sqrt{\usedim}.
\end{align}
Turning to the evaluation of the quantity $\EPSCRITSQ$, it is
straightforward to see that $[\Pi_{\Kpos}(\theta)]_j = \max \{0,
\theta_j\}$, and hence $\Exs \Pi_{\Kpos}(g) = \frac{1}{2} \Exs|g_1| \;
\ONES \; = \; \frac{1}{\sqrt{2 \pi}} \; \ONES$, where $\ONES \in
\real^d$ is a vector of all ones.  Thus, we have
\begin{align*}
\ltwo{ \Exs \Pi_{\Kpos}(g)} &= \sqrt{\frac{\usedim}{2 \pi}}\\
  ~\text{ and }~~\ltwo{\Exs \Pi_{\Kpos}(g)} ~\leq~ \Exs \ltwo{ \Pi_{\Kpos}(g)}
  &~\leq~ \sqrt{\Exs \ltwo{\Pi_{\Kpos} (g)}^2} = 
  \sqrt{\frac{d}{2}},
\end{align*}
where the second inequality follows from Jensen's inequality.  
So the first term in the definition of quantity
$\EPSCRITSQ$ is proportional to $\sqrt{\usedim}.$
As for the second term, since the standard basis vector $e_1 \in \Kpos \cap
\Sphere{}$, we have
\begin{align*}
\inf_{\eta \in \Kpos \cap \Sphere{}} \inprod{\eta}{\Exs\ProjK
  g} \leq \inprod{e_1}{\frac{1}{\sqrt{2\pi}} \; \ONES} =
\frac{1}{\sqrt{2\pi}}.
\end{align*}
Consequently, the second term in the definition of quantity
$\EPSCRITSQ$ lower bounded by a universal constant times $\usedim$.
Combining these derivations yields the stated
claim~\eqref{EqnGLRTkpos}.

%%%%%%%%%%%%%%%%%%%%%%%%%%%%%%%%%%%%%%%%%%%%%%%%%%%%%%%%%%%%%%%%%%%%%%%%%%%%%%%%%

\subsubsection{Monotone cone}
\label{SecMonotoneCone}

As our final example, consider testing in the monotone cone given by $\Mon \defn \big \{ \theta \in \real^\usedim \mid \theta_1 \leq \theta_2 \leq \cdots \leq \theta_\usedim \big \}.$
Testing with monotone cone constraint has also been studied in different settings before, where it is known in some cases that restricting to monotone cone helps reduce the hardness of the problem to be logarithmically dependent on the dimension (e.g., \cite{batu2004sublinear,wei2016sharp}).

Here we use Theorem~\ref{ThmGLRT} to show that
\begin{align} 
\label{EqnGLRTmonotone}
\epsglrt^2(\Mon) & \asymp \sigma^2 \sqrt{\log \usedim}.
\end{align}
From known results on monotone cone (see \S 3.5,~\cite{amelunxen2014living}), we know that $\Exs
\ltwo{\Pi_{\Mon} g} \asymp \sqrt{\log \usedim}$.  So the only
remaining detail is to control the second term defining
$\EPSCRITSQ$. We claim that the second term is actually infinity since 
\begin{align} \label{EqnMon-inprod}
  \max \{ 0, ~\inf \limits_{\eta \in
    \Mon \cap \Sphere{}} \inprod{\eta}{\Exs \Proj_\Mon g} \} = 0,
\end{align}
which can be seen by simply noticing vectors $\frac{1}{\sqrt{\usedim}}\ONES, -\frac{1}{\sqrt{\usedim}}\ONES \in \Mon\cap \Sphere{}$ and 
\begin{align*}
  \min \Big\{ \inprod{\frac{\ONES}{\sqrt{\usedim} }}{\Exs \Proj_\Mon g}, 
  ~~
  \inprod{ - \frac{\ONES}{\sqrt{\usedim}}}{\Exs \Proj_\Mon g} \Big\} \leq 0.
\end{align*}
Here $\ONES \in \real^\usedim$ denotes the vector of all ones. Combining the pieces yields the claim~\eqref{EqnGLRTmonotone}.

\paragraph{Testing constant versus monotone}
It is worth noting that the same GLRT bound also holds for the more general
problem of testing the monotone cone $\Mon$ versus the linear subspace
$\LinSpace = \myspan(\ONES)$ of constant vectors, namely:
\begin{align} 
\label{EqnGLRTmonotone_const}
\epsglrt^2(\LinSpace, \Mon) & \asymp \sigma^2 \sqrt{\log \usedim}.
\end{align}
In particular, the following lemma provides the control
that we need:
\begin{lems}
\label{LemMonInnProCal} 
For the monotone cone $\Mon$ and the linear space $\LinSpace =
\myspan(\ONES)$, there is a universal constant $c$ such that
\begin{align*}
\inf \limits_{\eta \in \Kcone \cap \Sphere{}}
\inprod{\eta}{\Exs \ProjK g} \leq c,
\qquad{\Kcone \defn \Mon \cap \LinSpacePerp}.
\end{align*}
\end{lems}
\noindent See Appendix~\ref{AppLemMonInnProCal} for the proof of this
lemma.

\paragraph*{Testing an arbitrary vector $\theta_0$ versus the monotone cone}

Finally, let us consider an important implication of
Corollary~\ref{CorConvexSet} in the context of testing departures in
monotone cone. More precisely, for a fixed vector $\theta_0 \in \Mon$,
consider the testing problem
\begin{align}
\label{TestKpiece-Mon}
\Hyp_0: \theta = \theta_0,~~~\text{versus}~~~ \Hyp_1: \theta \in \Mon,
\end{align}
Let us define $k(\theta_0)$ as the number of constant \emph{pieces} of
$\theta_0$, by which we mean there exist integers
$\usedim_1,\ldots,\usedim_{k(\theta_0)}$ with $\usedim_i \geq 1$ and
$\usedim_1 + \cdots + \usedim_{k(\theta_0)} = \usedim$ such that
$\theta_0$ is a constant on each set $S_i \defn \{j : \sum_{t=1}^{i-1}
\usedim_t + 1 \leq j \leq \sum_{t=1}^{i} \usedim_i\},$ for $1\leq i
\leq k(\theta_0)$.

We claim that Corollary~\ref{CorConvexSet} guarantees that the optimal
testing radius satisfies
\begin{align}
\label{EqnKpiece}
  \epsopt^2(\theta_0,\Mon;\rho) \, \lesssim \, \sigma^2
  \sqrt{k(\theta_0) \log \left(\frac{\usedim}{k(\theta_0)}\right)}.
\end{align}
Note that this upper bound depends on the structure of $\theta_0$
through how many pieces $\theta_0$ possesses, which reveals the
adaptive nature of Corollary~\ref{CorConvexSet}.

In order to prove inequality~\eqref{EqnKpiece}, let us use shorthand
$k$ to denote $k(\theta_0)$.  First notice that both
$\ONES/\sqrt{\usedim}, -\ONES/\sqrt{\usedim} \in
\Tan{\Mon}{\theta_0}$, then
\begin{align*}
  \max \{0, ~\inf \limits_{\eta \in \Tan{\Mon}{\theta_0} \cap
    \Sphere{}} \inprod{\eta}{\Exs \Pi_{ \Tan{\Mon}{\theta_0}}
    g} \} \leq 0,
\end{align*}
which implies the second term for $\EPSCRITSQ$ equals to infinity.  It
only remains to calculate $\Exs \ltwo{\Pi_{ \Tan{\Mon}{\theta_0}} g}$.
Since the tangent cone $\Tan{\Mon}{\theta_0}$ equals to the Cartesian
product of $k$ monotone cones, namely $\Tan{\Mon}{\theta_0} =
\Mon_{\usedim_1} \times \cdots \times \Mon_{\usedim_k}$, we have
\begin{align*}
  \Exs \ltwo{\Pi_{ \Tan{\Mon}{\theta_0}} g}^2 = \Exs
  \ltwo{\Pi_{\Mon_{\usedim_1}} g}^2 + \cdots + \Exs
  \ltwo{\Pi_{\Mon_{\usedim_k}} g}^2 & = \log(\usedim_1) + \cdots +
  \log (\usedim_k) \\ & \leq k\log \left(\frac{d}{k}\right),
\end{align*}
where the last step follows from convexity of the logarithm
function. Therefore Jensen's inequality guarantees that
\begin{align*}
\Exs \ltwo{\Pi_{ \Tan{\Mon}{\theta_0}} g} \leq \sqrt{\Exs
    \ltwo{\Pi_{ \Tan{\Mon}{\theta_0}} g}^2} \leq \sqrt{k \log
  \left(\frac{d}{k}\right)}.
\end{align*}
Putting the pieces together, Corollary~\ref{CorConvexSet} guarantees
that the claimed inequality~\eqref{EqnKpiece} holds for the testing
problem~\eqref{TestKpiece-Mon}.

%%%%%%%%%%%%%%%%%%%%%%%%%%%%%%%%%%%%%%%%%%%%%%%%%%%%%%%%%%%%%%%%%%%%%%%%%%%%%%%%%%

%%%%%%%%%%%%%%%%%%%%%%%%%%%%%%%%%%%%%%%%%%%%%%%%%%%%%%%%%%%%%%%%%%%%%%%%%%%%%%%%%%%%%

\subsection{Lower bounds on the testing radius}
\label{SecGeneralLower}

Thus far, we have derived sharp bounds for a particular
procedure---namely, the GLRT.  Of course, it is of interest to
understand when the GLRT is actually an optimal test, meaning that
there is no other test that can discriminate between the null and
alternative for smaller separations.  In this section, we use
information-theoretic methods to derive a lower bound on the optimal
testing radius $\epsopt$ for every pair of non-oblique and nested
closed convex cones $(\ConeSmall, \ConeBig)$. Similar to
Theorem~\ref{ThmGLRT}, this bound depends on the geometric structure of
intersection cone $\Kcone \defn \ConeBig \cap \ConeSmallPolar$, where
$\ConeSmallPolar$ is the polar cone to $\ConeSmall$.

In particular, let us define the quantity
\begin{align} 
\label{EqnRcri}
\EPSGENSQ & \defn \min \Biggr \{ \Exs \ltwo{\ProjK g }, ~~
\Big(\frac{\Exs \ltwo{\ProjK g}}{ \sup \limits_{\eta \in \Kcone \cap
    \Sphere{}} \inprod{\eta}{\Exs \ProjK g} } \Big)^2 \Biggr
\}.
\end{align}
Note that the only difference from $\EPSCRITSQ$ is the replacement of
the infimum over $\Kcone \cap \Sphere{}$ with a supremum, in
the denominator of the second term.  Moreover, since the supremum is
achieved at $\frac{\Exs \ProjK g}{\| \Exs \ProjK g\|_2}$, we have
$\sup_{\eta \in \Kcone \cap \Sphere{}} \inprod{\eta}{\Exs
  \ProjK g} = \| \Exs \ProjK g\|_2$.  Consequently, the second term on
the right-hand side of equation~\eqref{EqnRcri} can be also written in
the equivalent form $\Big(\frac{\Exs \ltwo{\ProjK g}}{ \| \Exs \ProjK
  g\|_2} \Big)^2$.

\noindent With this notation in hand, are now ready to state a general
lower bound for minimax optimal testing radius:
\begin{theos} 
  \label{ThmLBGen}
 There are numbers $\{ \infolowconst, \rho \in (0, 1/2]\}$ such that
for every nested pair of non-oblique closed convex cones $\ConeSmall
\subset \ConeBig$, we have
\begin{align}
\label{EqnConEps}
\inf_{\psi} \UNIERR(\psi; \ConeSmall, \ConeBig, \epsilon) & \geq \rho
\qquad \mbox{whenever $\epsilon^2 \leq \infolowconst \, \sigma^2 \,
  \EPSGENSQ$},
\end{align}
In particular, we can take $\kappa_\rho = 1/14$ for all $\rho \in
(0, 1/2]$.
\end{theos}

\paragraph*{Remarks} In more compact terms, Theorem~\ref{ThmLBGen} can be
understood as guaranteeing
\begin{align*}
   \epsopt(\ConeSmall, \ConeBig; \rho) \gtrsim \sigma \EPSGEN,
 \end{align*} 
where $\gtrsim$ denotes an inequality up to constants (with $\rho$
viewed as fixed).

Theorem~\ref{ThmLBGen} is proved by constructing a distribution over
the alternative $\Hyp_1$ supported only on those points in $\Hyp_1$
that are hard to distinguish from $\Hyp_0$.  Based on this
construction, the testing error can be lower bound by controlling the
total variation distance between two marginal likelihood functions.
We refer our readers to our Section~\ref{SecProofThmLBGen} for more
details on this proof.

\vspace*{.1in}

One useful consequence of Theorem~\ref{ThmLBGen} is in providing a
sufficient condition for optimality of the GLRT, which we summarize
here:

\bcors[Sufficient condition for optimality of GLRT]
\label{CorSufficientGLRT}
Given the cone $\Kcone = \ConeBig \cap \ConeSmallPolar$, suppose that
there is a numerical constant $b > 1$, independent of $\Kcone$ and all
other problem parameters, such that
\begin{align}
  \label{EqnSufficientGLRT}
\sup \limits_{\eta \in \Kcone \cap \Sphere{}}
\inprod{\eta}{\Exs \ProjK g} \; = \; \| \Exs \ProjK g\|_2 & \; \leq b
\; \inf \limits_{\eta \in \Kcone \cap
  \Sphere{}} \inprod{\eta}{\Exs \ProjK g}.
\end{align}
Then the GLRT is a minimax optimal test---that is,
$\epsglrt(\ConeSmall, \ConeBig; \rho) \asymp \epsopt(\ConeSmall,
\ConeBig; \rho)$.
\ecors
It is natural to wonder whether the
condition~\eqref{EqnSufficientGLRT} is also necessary for optimality
of the GLRT.  This turns out not to be the case.  The monotone cone,
to be revisited in Section~\ref{SecMonotoneRevisit}, provides an
instance of a cone testing problem for which the GLRT is optimal while
condition~\eqref{EqnSufficientGLRT} is violated.  Let us now return to
these concrete examples.

%%%%%%%%%%%%%%%%%%%%%%%%%%%%%%%%%%%%%%%%%%%%%%%%%%%%%%%%%%%%%%%%%%%%%%%%%%%%%%%

\subsubsection{Revisiting the $k$-dimensional subspace}
\label{SecSubspaceRevisit}

Let $S_k$ be a subspace of dimension $k \leq \usedim$.  In our earlier
discussion in Section~\ref{SecSubspaceCone}, we established that
$\epsglrt^2(S_k) \asymp \sigma^2 \sqrt{k}$.  Let us use
Corollary~\ref{CorSufficientGLRT} to verify that the GLRT is optimal
for this problem.  For a $k$-dimensional subspace $\Kcone = S_k$, we
have $\Exs \ProjK g = 0$ by symmetry; consequently,
condition~\eqref{EqnSufficientGLRT} holds in a trivial manner.  Thus,
we conclude that \mbox{$\epsopt^2(S_k) \asymp \epsglrt^2(S_k)$,}
showing that the GLRT is optimal over all tests.

%%%%%%%%%%%%%%%%%%%%%%%%%%%%%%%%%%%%%%%%%%%%%%%%%%%%%%%%%%%%%%%%%%%%%%%%%%%%%%%%c
\subsubsection{Revisiting the circular cone}
\label{SecCircularRevisit}

Recall the circular cone $\Kcone = \{ \theta \in \real^d \mid \,
\theta_1 \geq \|\theta\|_2 \cos(\alpha) \}$ for fixed $0 < \alpha < \pi/2.$
In our earlier
discussion, we proved that \mbox{$\epsglrt^2(\Kcone) \asymp
  \sigma^2$.}  Here let us verify that this scaling is optimal over
all tests.  By symmetry, we find that \mbox{$\Exs \ProjK g = \beta e_1
  \in \real^\usedim$,} where $e_1$ denotes the standard Euclidean
basis vector with a $1$ in the first coordinate, and $\beta > 0$ is
some scalar.  For any vector $\eta \in \Kcone \cap \Sphere{}$,
we have $\eta_1 \geq \cos(\alpha)$, and hence
\begin{align*}
\inf \limits_{\eta \in \Kcone \cap \Sphere{}}
\inprod{\eta}{\Exs \ProjK g} & \geq \cos(\alpha) \beta \; = \;
\cos(\alpha) \|\Exs \ProjK g\|_2.
\end{align*}
Consequently, we see that condition~\eqref{EqnSufficientGLRT} is
satisfied with $b = \frac{1}{\cos(\alpha)} > 0$, so that the GLRT is
optimal over all tests for each fixed $\alpha$.  (To be clear, in this
example, our theory does not provide a sharp bound uniformly over
varying $\alpha$.)

%%%%%%%%%%%%%%%%%%%%%%%%%%%%%%%%%%%%%%%%%%%%%%%%%%%%%%%%%%%%%%%%%%%%%%%%%%%%%%

\subsubsection{Revisiting the product cone}
\label{SecCartesianProductRevisit}

Recall from Section~\ref{SecCartesianProductCone} our discussion of
the Cartesian product cone $\Prc = \Cir{\usedim-1}(\alpha) \times
\real$.  
In this section, we establish that the GLRT, when
applied to a testing problem based on this case, is sub-optimal by a
factor of $\sqrt{\usedim}$.

Let us first prove that the sufficient
condition~\eqref{EqnSufficientGLRT} is violated, so that
Corollary~\ref{CorSufficientGLRT} does \emph{not} imply optimality of
the GLRT.  From our earlier calculations, we know that $\Exs
\ltwo{\projPrc g} \asymp \sqrt{\usedim}$.  On the other hand, we also
know that $\Exs \projPrc g$ is equal to zero in its last coordinate.
Since the standard basis vector $e_d$ belongs to the set $\Prc \cap
\Sphere{}$, we have
\begin{align*}
  \inf_{\eta \in \Prc \cap \Sphere{}} \inprod{\eta}{\Exs
    \Pi_{\Prc} g} & \leq \inprod{e_d}{\Exs \Pi_{\Prc} g} \; = \; 0,
\end{align*}
so that condition~\eqref{EqnSufficientGLRT} does not hold.

From this calculation alone, we cannot conclude that the GLRT is
sub-optimal.  So let us now compute the lower bound guaranteed by
Theorem~\ref{ThmLBGen}.  From our previous discussion, we know that
$\Exs \projPrc g = \beta e_1$ for some scalar $\beta >0$.  Moreover,
we also have $\ltwo{\Exs \projPrc g} = \beta \asymp \sqrt{\usedim}$;
this scaling follows because we have $\ltwo{\Exs \projPrc g} =
\ltwo{\Exs \Pi_{\Cir{\usedim-1}(\alpha)} g} \asymp \sqrt{\usedim-1}$,
where we have used the previous inequality~\eqref{EqnCir-Sec} for
circular cone.  Putting together the pieces, we find that
Theorem~\ref{ThmLBGen} implies that
\begin{align}
  \label{EqnProductLB}
\epsopt^2(\Prc) \gtrsim \sigma^2,
\end{align}
which differs from the GLRT scaling in a factor of $\sqrt{\usedim}$.

Does there exist a test that achieves the lower
bound~\eqref{EqnProductLB}?  It turns out that a simple truncation
test does so, and hence is optimal.  To provide intuition for the
test, observe that for any vector $\theta \in \Prc \cap \Sphere{}$, we
have $\theta_1^2 + \theta_d^2 \geq \cos^2(\alpha)$.  To verify this
claim, note that
\begin{align*}
\frac{1}{\cos^2(\alpha)} \Big (\theta_1^2 + \theta_\usedim^2 \Big) &
\geq \frac{\theta_1^2}{\cos^2(\alpha)} + \theta_\usedim^2 \; \geq \;
\sum_{j=1}^{d-1} \theta_j^2 + \theta_\usedim^2 \; = \;
1.
\end{align*}
Consequently, the two coordinates $(y_1, y_\usedim)$ provide
sufficient information for constructing a good test.  In particular,
consider the truncation test
\begin{align*}
  \varphi(y) & \defn \Ind \big[ \ltwo{(y_1, y_d)} \geq \beta \big],
\end{align*}
for some threshold $\beta > 0$ to be determined.  This can be viewed
as a GLRT for testing the standard null against the alternative
$\real^2$, and hence our general theory guarantees that it will
succeed with separation $\epsilon^2 \gtrsim \sigma^2$.  This guarantee
matches our lower bound~\eqref{EqnProductLB}, showing that the
truncation test is indeed optimal, and moreover, that the GLRT is
sub-optimal by a factor of $\sqrt{\usedim}$ for this particular problem.

We provide more intuition on why the the GLRT sub-optimal and use this intuition to construct a more general class of
problems for which a similar sub-optimality is witnessed in Appendix~\ref{AppGLRTsub}.

%%%%%%%%%%%%%%%%%%%%%%%%%%%%%%%%%%%%%%%%%%%%%%%%%%%%%%%%%%%%%%%%%%%%%%%%%%

%%%%%%%%%%%%%%%%%%%%%%%%%%%%%%%%%%%%%%%%%%%%%%%%%%%%%%%%%%%%%%%%%%%%%%%%%%%%%%%%%%%

\subsection{Detailed analysis of two cases}
\label{SecDetail}

This section is devoted to a detailed analysis of the orthant cone,
followed by the monotone cone.  Here we find that the GLRT is again
optimal for both of these cones, but establishing this optimality
requires a more delicate analysis.

%%%%%%%%%%%%%%%%%%%%%%%%%%%%%%%%%%%%%%%%%%%%%%%%%%%%%%%%%%%%%%%%%%%%%%%%%%%%%%%%%%%%%%%

\subsubsection{Revisiting the orthant cone}
\label{SecOrthantRevisit}

Recall from Section~\ref{SecOrthantCone} our discussion of the
(non-negative) orthant cone
\begin{align*}
    \Kpos \defn \{\theta \in \real^\usedim \mid \theta_j \geq 0 \quad
    \mbox{for $j = 1, \ldots, \usedim$} \},
\end{align*}
where we proved that $\epsglrt^2(\Kpos) \asymp \sigma^2
\sqrt{\usedim}$. Let us first show that the sufficient condition
\eqref{EqnSufficientGLRT} does not hold, so that
Corollary~\ref{CorSufficientGLRT} does \emph{not} imply optimality of
the GLRT.  As we have computed in our Section~\ref{SecOrthantCone}, 
quantity $\Exs\ltwo{\Pi_{\Kpos} (g) } \asymp \sqrt{\usedim}$ and 
\begin{align*}
  \inf_{\eta \in \Kpos \cap \Sphere{}} \inprod{\eta}{\Exs\ProjK
  g} \leq \inprod{e_1}{\frac{1}{\sqrt{2\pi}} \; \ONES} =
\frac{1}{\sqrt{2\pi}},
\end{align*}
where use the fact that $\Exs \Pi_{\Kpos}(g)  = \; \frac{1}{\sqrt{2 \pi}} \; \ONES$.
So that condition~\eqref{EqnSufficientGLRT} is violated.

Does this mean the GLRT is sub-optimal? It turns out that the GLRT is
actually optimal over all tests, as we can demonstrate by proving a
lower bound---tighter than the one given in Theorem~\ref{ThmLBGen}---that
matches the performance of the GLRT.  We summarize it as follows:
\begin{props} 
\label{PropKpos} 
There are numbers $\{ \infolowconst, \rho \in (0, 1/2]\}$ such that
for the (non-negative) orthant cone $\Kpos$, we have
\begin{align}
\label{EqnConEpsKpos}
\inf_{\psi} \UNIERR(\psi; \{0\}, \Kpos, \epsilon) & \geq \rho
\qquad{\text{ whenever }\epsilon^2 \leq \infolowconst \, \sigma^2
  \sqrt{\usedim}.}
\end{align}
\end{props}
\noindent See the Section~\ref{sec:proof_prop_kpos} for the proof of this
proposition.

\noindent From Proposition~\ref{PropKpos}, we see that the optimal
testing radius satisfies $\epsopt^2(\Kpos) \gtrsim \sigma^2
\sqrt{\usedim}$.  Compared to the GLRT radius $\epsglrt^2(\Kpos)$
established in expression~\eqref{EqnGLRTkpos}, it implies the optimality of the GLRT.

%%%%%%%%%%%%%%%%%%%%%%%%%%%%%%%%%%%%%%%%%%%%%%%%%%%%%%%%%%%%%%%%%%%%%%%%%%%%%%%

\subsubsection{Revisiting the monotone cone} 
\label{SecMonotoneRevisit}

Recall the monotone cone given by \mbox{$\Mon \defn \{\theta \in
  \real^\usedim \mid \theta_1 \leq \theta_2 \leq \cdots \leq
  \theta_{\usedim} \}$.}  In our previous discussion in
Section~\ref{SecMonotoneCone}, we established that $\epsglrt^2(\Mon)
\asymp \sigma^2 \sqrt{\log \usedim}$.  We also pointed out that this
scaling holds for a more general problem, namely, testing cone $\Mon$
versus linear subspace $\LinSpace = \myspan(\ONES)$.  In this section,
we show that the GLRT is also optimal for both cases.

First, observe that Corollary~\ref{CorSufficientGLRT} does not imply
optimality of the GLRT.  In particular, using symmetry of the inner
product, we have shown in expression~\eqref{EqnMon-inprod} that  
\begin{align*}
  \max \{ 0, \inf \limits_{\eta \in \Mon \cap \Sphere{}}
  \inprod{\eta}{\Exs \Proj_\Mon g} \}= 0,
\end{align*}
for cone pair $(\ConeSmall, \ConeBig) = (\{0\}, \Mon)$. 
Also note that from Lemma~\ref{LemMonInnProCal} we know that for cone pair $(\ConeSmall, \ConeBig) = (\myspan(\ONES), \Mon)$, there is a universal constant $c$ such that
\begin{align*}
\inf \limits_{\eta \in \Kcone \cap \Sphere{}}
\inprod{\eta}{\Exs \ProjK g} \leq c,
\qquad{\Kcone \defn \Mon \cap \LinSpacePerp}.
\end{align*}
In both cases, since $\Exs \ltwo{\Pi_{\Kcone} g} \asymp \sqrt{\log \usedim},$ 
so that the sufficient condition~\eqref{EqnSufficientGLRT} for GLRT optimality fails to hold.

It turns out that we can demonstrate a matching lower bound for
$\epsopt^2(\Mon)$ in a more direct way by carefully constructing a
prior distribution on the alternatives and control the testing error.
Doing so allows us to conclude that the GLRT is optimal, and we summarize
our conclusions in the following:

\begin{props} 
\label{PropMon}
There are numbers $\{ \infolowconst, \rho \in (0, 1/2]\}$ such that
  for the monotone cone $\Mon$ and subspace $\LinSpace = \{0\}$ or
  $\myspan(\ONES)$, we have
\begin{align}
\label{EqnConEpsMon}
\inf_{\psi} \UNIERR(\psi; \LinSpace, \Mon, \epsilon) & \geq \rho
\qquad{\text{ whenever }\epsilon^2 \leq \infolowconst \, \sigma^2
  \sqrt{\log (e \usedim)}}.
\end{align}
\end{props}
\noindent See Section~\ref{sec:proof_prop_mon} for the proof of this
proposition.

Proposition~\ref{PropMon}, equipped with previous achievable results
by GLRT \eqref{EqnGLRTmonotone}, gives a sharp rate characterization on the testing radius for both
problems with regard to monotone cone:
\begin{align*}
  \Hyp_0: \theta = 0 ~~\text{ versus }~~\Hyp_1: \theta \in \Mon\\
  \text{ and }~~\Hyp_0: \theta \in \myspan(\ONES) ~~\text{ versus }~~\Hyp_1: \theta \in \Mon.
\end{align*}
In both cases, the optimal testing radius satisfies
$\epsopt^2(\LinSpace, \Mon, \rho) \asymp \sigma^2
\sqrt{\log(e\usedim)}$.  As a consequence, the GLRT is optimal up to
an universal constant.  As far as we know, the problem of testing a
zero or constant vector versus the monotone cone as the alternative
has not been fully characterized in any past work.

%%%%%%%%%%%%%%%%%%%%%%%%%%%%%%%%%%%%%%%%%%%%%%%%%%%%%%%%%%%%%%%%%%%%%%%%%%%%%%%%%%%

\section{Proofs of main results}
\label{SecProofs}

We now turn to the proofs of our main results, with the proof of
Theorems~\ref{ThmGLRT} and~\ref{ThmLBGen} given in
Sections~\ref{SecProofThmGLRT} and~\ref{SecProofThmLBGen}
respectively. In all cases, we defer the
proofs of certain more technical lemmas to the appendices.

%%%%%%%%%%%%%%%%%%%%%%%%%%%%%%%%%%%%%%%%%%%%%%%%%%%%%%%%%%%%%%%%%%%%%%%%%%%

\subsection{Proof of Theorem~\ref{ThmGLRT}}
\label{SecProofThmGLRT}

Since the cones $(\ConeSmall, \ConeBig)$ are both invariant under
rescaling by positive numbers, we can first prove the result for noise
level $\sigma = 1$, and then recapture the general result by rescaling
appropriately.  Thus, we fix $\sigma = 1$ throughout the remainder of
the proof so as to simplify notation.  Moreover, let us recall that
the GLRT consists of tests of the form $\psiglrt(\yvec) \defn
\Ind(T(\yvec) \geq \beta)$, where the likelihood ratio $T(\yvec)$ is
given in equation~\eqref{EqnGLRTstat}.  Note here the cut-off $\beta
\in [0, \: \infty)$ is a constant that does not depend on the data vector
  $\yvec$.

By the previously discussed equivalence~\eqref{EqnSimpleEquiv}, we can
focus our attention on the simpler problem
$\TEST{\{0\}}{\Kcone}{\epsilon}$, where $\Kcone = \ConeBig \cap
\ConeSmallPolar$. 
By the monotonicity of the square function for positive numbers,
the GLRT is controlled by the
behavior of the statistic $\ltwo{\ProjK(\yvec)}$, and in particular
how it varies depending on whether $\yvec$ is drawn according to
$\Hyp_0$ or $\Hyp_1$.

Letting $g \in \real^\usedim$ denote a standard Gaussian random
vector, let us introduce the random variable \mbox{$Z(\theta) \defn
  \ltwo{\ProjK(\theta + g)}$} for each $\theta \in
\real^\usedim$. Observe that the statistic $\|\ProjK(y)\|_2$ is
distributed according to $Z(0)$ under the null $\Hyp_0$, and according
to $Z(\theta)$ for some $\theta \in \Kcone$ under the alternative
$\Hyp_1$.  The Lemma~\ref{LemConcentration} which is stated and proved in Appendix~\ref{AppLemConcentration} guarantees random variables of the type
$Z(\theta)$ and $\inprod{\theta}{\ProjK g}$ are sharply concentrated around their expectations.

As shown in the sequel, using the concentration
bound~\eqref{EqnLipschitz1}, the study of the GLRT can be reduced to
the problem of bounding the mean difference
\begin{align} 
\label{EqnGammaDifference}
\Gamma(\theta) & \defn \Exs \left( \ltwo{\ProjK (\theta + g)} -
\ltwo{\ProjK g} \right)
\end{align}
for each $\theta \in \Kcone$.  In particular, in order to prove the
achievability result stated in part (a) of Theorem~\ref{ThmGLRT}, we
need to lower bound $\Gamma(\theta)$ uniformly over $\theta \in
\Kcone$, whereas a uniform upper bound on $\Gamma(\theta)$ is required
in order to prove the negative result in part (b).

%%%%%%%%%%%%%%%%%%%%%%%%%%%%%%%%%%%%%%%%%%%%%%%%%%%%%%%%%%%%%%%%%%%%%%%%%%%%%%%%%%%%%%%%%%%%

\subsubsection{Proof of GLRT achievability result (Theorem~\ref{ThmGLRT}(a))}

By assumption, we can restrict our attention to alternative
distributions defined by vectors $\theta \in \Kcone$ satisfying the
lower bound \mbox{$\ltwo{\theta}^2 \geq \upconst \; \NEWEPSCRITSQ$},
where for every target level $\rho \in (0,1)$, constant $\upconst$ is
chosen such that
\begin{align*}
  \upconst \defn \max \left\{ 32\pi, ~ \inf \left(B>0 \mid
  \frac{B^{1/2}}{(2^7\pi B)^{1/4} + 16} - \frac{2}{\sqrt{e}} \geq
  \sqrt{-8\log(\rho/2)} \right) \right\}.
\end{align*}
Since function $f(x) \defn \frac{x^{1/2}}{(2^7\pi x)^{1/4} + 16} -
\frac{2}{\sqrt{e}}$ is strictly increasing and goes to infinity, so
that the constant $\upconst$ defined above is always finite.

We first claim that it suffices to show that for such vector, the
difference~\eqref{EqnGammaDifference} is lower bounded as
\begin{align}
\label{Eqnc1.eq}
\Gamma(\theta) \geq \frac{\upconst^{1/2}}{(2^7\pi
    \upconst)^{1/4} + 16} - \frac{2}{\sqrt{e}} = f(\upconst).
\end{align}
Taking inequality~\eqref{Eqnc1.eq} as given for the moment, we claim
that the test
\begin{align*}
\psinew(\yvec) = \Ind[\ltwo{\ProjK(\yvec)}^2 \geq \newthresh] \qquad
\mbox{with threshold $\newthresh \defn (\frac{1}{2} f(\upconst) +
  \Exs[\ltwo{\ProjK(g)}])^2$}
\end{align*}
has uniform error probability controlled as
\begin{align}
\label{EqnUniControl}
\UNIERR(\psinew; \{0\}, \Kcone, \epsilon) & \defn \Exs_{0}
       [\psinew(\yvec)] + \sup_{\theta\in \Kcone, \ltwo{\theta}^2 \geq \epsilon^2} \Exs_\theta[1 -
         \psinew(\yvec)] 
         \; \leq \; 2 e^{- f^2(\upconst)/8} < \rho.
\end{align}
where the last inequality follows from the definition of $\upconst$.

\paragraph{Establishing the error control~\eqref{EqnUniControl}}

Beginning with errors under the null $\Hyp_0$, we have
\begin{align*}
\Exs_{0} [\psinew(\yvec)]  = \Prob_0(\ltwo{\ProjK g} \geq \sqrt{\newthresh})
& = \Prob_0 \big[\ltwo{\ProjK g} - \Exs[\ltwo{\ProjK g}] \geq
  f(\upconst)/2 \big] \\
& \leq \exp(-f^2(\upconst)/8),
\end{align*}
where the final inequality follows from the concentration
bound~\eqref{EqnLipschitz1} in Lemma~\ref{LemConcentration}, as along
as $f(\upconst) > 0$.

On the other hand, we have
\begin{align*}
\sup_{\theta\in \Kcone, \ltwo{\theta}^2 \geq \epsilon^2} \Exs_{\theta} [1-\psinew(\yvec)] & = \Prob
\Big[\ltwo{\ProjK(\theta + g)} \leq \frac{1}{2} f(\upconst) + \Exs \ltwo{\ProjK g}
  \Big] \\
& = \Prob \Big[ \ltwo{\ProjK (\theta+ g)} - \Exs \ltwo{\ProjK(\theta+
    g)} \leq \frac{1}{2} f(\upconst) - \Gamma(\theta) \Big],
\end{align*}
where the last equality follows by substituting \mbox{$\Gamma(\theta) = \Exs
[\ltwo{\ProjK (\theta+g)}] - \Exs[\ltwo{\ProjK g}]$}.  Since the lower
bound~\eqref{Eqnc1.eq} guarantees that \mbox{$\frac{1}{2} f(\upconst) - \Gamma(\theta)
\leq -\frac{1}{2} f(\upconst)$}, we find that
\begin{align*}
\sup_{\theta\in \Kcone, \ltwo{\theta}^2 \geq \epsilon^2} \Exs_{\theta} [1-\psinew(\yvec)] & \leq \Prob
\Big[ \ltwo{\ProjK (\theta+ g)} - \Exs \ltwo{\ProjK(\theta+ g)} \leq -
  \frac{1}{2}f(\upconst)\Big] \\
& \leq \exp(-f^2(\upconst)/8),
\end{align*}
where the final inequality again uses the concentration
inequality~\eqref{EqnLipschitz1}.  Putting the pieces together yields
the claim~\eqref{EqnUniControl}.

% \paragraph{Lower bounding the difference $\Gamma(\theta)$}  
The only remaining
detail is to prove the lower bound~\eqref{Eqnc1.eq} on the
difference~\eqref{EqnGammaDifference}.  
To mainstream our proof, we leave the proof of this detail to Appendix~\ref{AppLemLSEup}.

% In order to do so, we require an 
% auxiliary lemma which is stated and proved in  

% \begin{lems} 
% \label{LemLSEup}
% For every closed convex cone $\Kcone$ and vector $\theta \in \Kcone$, we
% have the lower bounds
% \begin{subequations}
% \begin{align}
% \label{Eqngl1.eq}
% \Gamma(\theta) & \geq \frac{\ltwo{\theta}^2}{2 \ltwo{\theta} + 8\Exs
%   \ltwo{\ProjK g}} -\frac{2}{\sqrt{e}}.
% \end{align}
% Moreover, for any vector $\theta$ that also satisfies the inequality
% $\inprod{\theta}{\Exs \ProjK g} \geq \ltwo{\theta}^2 $, we have
% \begin{align}
% \label{Eqngl2.eq}
% \Gamma(\theta) & \geq \alpha^2(\theta) \frac{\inprod{\theta}{\Exs
%     \ProjK g} - \ltwo{\theta}^2} {\alpha(\theta) \ltwo{\theta} + 2
%   \Exs \ltwo{\ProjK g}} -\frac{2}{\sqrt{e}},
% \end{align}
% where $\alpha(\theta) \defn 1 - \exp \left(\frac{-
%   \inprod{\theta}{\Exs \ProjK g}^2}{8 \ltwo{\theta}^2} \right)$.
% \end{subequations}
% \end{lems}
% \noindent  \\

%%%%%%%%%%%%%%%%%%%%%%%%%%%%%%%%%%%%%%%%%%%%%%%%%%%%%%%%%%%%%%%%%%%%%%%%%%%%%%%%%

\subsubsection{Proof of GLRT lower bound  (Theorem~\ref{ThmGLRT}(b))}
\label{sec:probThmGLRTb}

We divide our proof into two scenarios, depending on whether or not $\Exs \ltwo{\ProjK g}$ is less than 128. We focus on the case when $\Exs \ltwo{\ProjK g} \geq 128$ and leave the other scenario to the Appendix~\ref{AppScene2}. 

In this case, our strategy is to exhibit some $\theta \in \Hyp_1$ for which the expected difference
$\Gamma(\theta) = \Exs \left( \ltwo{\ProjK (\theta + g)} -
\ltwo{\ProjK g} \right)$ is small, which then leads to significant
error when using the GLRT.  In order to do so, we require an auxiliary lemma (Lemma~\ref{LemLatte}) to suitable control $\Gamma(\theta)$ which is stated and proved in Appendix~\ref{AppLemLatte}.

We now proceed to prove our main claim.  Based on
Lemma~\ref{LemLatte}, we claim that if $\epsilon^2 \leq \lowconst
\NEWEPSCRITSQ$ for a suitably small constant $\lowconst$ such that 
\begin{align*}
 \lowconst \defn \sup \left\{ \lowconst>0 \mid 12\sqrt{\lowconst} + 3\sqrt{\lowconst}
  \left(\frac{2}{e}\right)^{1/4} + 24\sqrt{\frac{\lowconst}{2e}} \,\leq \, \frac{1}{16} \right\},
\end{align*}
then 
\begin{align}
\label{EqnMusgraves}
\Gamma(\theta) \leq \frac{1}{16},
\qquad
\text{ for some } \theta \in \Kcone, ~\ltwo{\theta} \geq \epsilon.
\end{align}
We take inequality~\eqref{EqnMusgraves} as given for now, returning to prove it in our appendix~\ref{AppLemMusgraves}. In summary, then, we have exhibited some $\theta \in{}
\Hyp_1$---namely, a vector $\theta \in \Kcone$ with $\ltwo{\theta}
\geq \epsilon$---such that $\Gamma(\theta) \leq 1/16$.  This special
vector $\theta$ plays a central role in our proof.

We claim that the GLRT cannot succeed with error smaller than $0.11$ no matter how the cut-off $\beta$ is chosen. We leave this calculation to Appendix~\ref{AppCalculate}.

%%%%%%%%%%%%%%%%%%%%%%%%%%%%%%%%%%%%%%%%%%%%%%%%%%%%%%%%%%%%%%%%%%%%%%%%%%%%%%%%%%%%%

\subsection{Proof of Theorem~\ref{ThmLBGen}}
\label{SecProofThmLBGen}

We now turn to the proof of Theorem~\ref{ThmLBGen}.  As in the proof of
Theorem~\ref{ThmGLRT}, we can assume without loss of generality that
$\sigma = 1$.
% since the cones $(\ConeSmall, \ConeBig)$ are both invariant under
% rescaling by positive numbers.
Since $0 \in \ConeSmall$ and $\Kcone \defn \ConeBig
\cap \ConeSmallPolar \subseteq \ConeBig$, it suffices to prove a lower
bound for the reduced problem of testing
\begin{align*}
\Hyp_0: \theta = 0,~~~\text{versus}~~~ \Hyp_1: \ltwo{\theta} \geq
\epsilon, ~\theta \in \Kcone.
\end{align*}
Let $\NewBall(1) = \{ \theta \in \real^\usedim \, \mid \, \|\theta\|_2
< 1 \}$ denotes the open Euclidean ball of radius $1$, and let
$\NewBall^c(1) \defn \real^\usedim \backslash \NewBall(1)$ denotes its
complement.

We divide our analysis into two cases, depending on whether or not $\Exs
\ltwo{\ProjK g}$ is less than $7$. In both cases, let us set $\infolowconst = 1/14$.

\paragraph{Case 1} 
Suppose that $\Exs \ltwo{\ProjK g} < 7$.  In this case,  
\begin{align*}
\epsilon^2 \leq \infolowconst \EPSGENSQK \leq \infolowconst
\Exs \ltwo{\ProjK g} < 1/2.
\end{align*}
Similar to our proof of Theorem~\ref{ThmGLRT}(b), Case 1, by
reducing to the simple verses simple testing
problem~\eqref{EqnSimpleTesting}, any test yields testing error no
smaller than $1/2$ if $\epsilon^2 < 1/2$.  So our lower bound
directly holds for the case when $\Exs \ltwo{\ProjK g} < 7$. 

\paragraph{Case 2} Otherwise, suppose we have \mbox{$\Exs \ltwo{\ProjK g} \geq 7$}. 
The following lemma provides a generic way to lower bound the testing error. 

\begin{lems} 
\label{LemChisquareBound}
For every non-trivial closed convex cone $\Kcone$ and probability measure $\qprob$
supported on $\Kcone \cap \NewBall^c(1)$, the testing error is lower
bounded as
\begin{align}
\label{EqnTestErrorLowerBound}
\inf_{\psi} \UNIERR(\psi; \{0\}, \Kcone, \epsilon) \geq 1 -
\frac{1}{2} \sqrt{\Exs_{\eta, \eta'} \,
  \exp(\epsilon^2 \inprod{\eta}{\eta'})-1},
\end{align} 
where $\Exs_{\eta, \eta'}$ denotes expectation with respect to an
i.i.d pair $\eta, \eta' \sim \qprob$.
\end{lems}
\noindent See Appendix~\ref{AppLemChisquareBound} for the proof of
this claim.

%%%%%%%%%%%%%%%%%%%%%%%%%%%%%%%%%%%%%%%%%%%%%%%%%%%%%%%%%%%%%%%%%%%%%%%%%%%%%%%%%%%%%

We apply Lemma \ref{LemChisquareBound} with the probability measure
$\qprob$ defined as 
\begin{align} 
\label{EqnPrior}
\qprob(A) \defn \Prob \left(\frac{\ProjK g}{\Exs \ltwo{\ProjK g}/2} \in A \, 
\Bigm \vert \, \ltwo{\ProjK g} \geq \Exs \ltwo{\ProjK g}/2 \right),
\end{align}
for measurable set $A \subset \real^\usedim$ where $g$ denotes a standard
$d$-dimensional Gaussian random vector i.e., \mbox{$g \sim \NORMAL(0,
I_{\usedim}).$}  It is easy to check that measure $\qprob$ is supported
on $\Kcone \cap \NewBall^c(1)$. 
We make use of Lemma~\ref{LemInfoKey} in Appendix~\ref{AppLemInfoKey} to control 
$\Exs_{\eta,\eta'} \exp(\epsilon^2 \inprod{\eta}{\eta'})$ and thus upper bounding the testing error. 

% claim.

% %
% \blems
% \label{LemInfoKey}
% %
% Letting $\eta$ and $\eta'$ denote an i.i.d pair of random variables
% drawn from the distribution $\qprob$ defined in
% equation~\eqref{EqnPrior}, we have
% \begin{align}
% \label{EqnUBspicy}
%   \Exs_{\eta,\eta'} \exp(\epsilon^2 \inprod{\eta}{\eta'}) & \leq
%   \frac{1}{a^2} \exp\left(\frac{5\epsilon^2\ltwo{\Exs \ProjK g}^2
%   }{(\Exs \ltwo{\ProjK g})^2} + \frac{40 \epsilon^4 \Exs(\ltwo{\ProjK
%       g}^2)}{(\Exs\ltwo{\ProjK g})^4} \right),
% \end{align}
% where \mbox{$a \defn \Prob(\ltwo{\ProjK g} \geq \frac{1}{2} \Exs
%   \ltwo{\ProjK g})$} and $\epsilon > 0$ satisfies the inequality
% $\epsilon^2 \leq (\Exs \|\Pi_K g\|_2)^2/32$. 
% \elems
% \noindent See supplementary file [Appendix~\ref{AppLemInfoKey}] for the proof of this
% claim.

We now lower bound the testing error when \mbox{$\epsilon^2 \leq
  \infolowconst \, \EPSGENSQK$}. By definition of $\EPSGENSQK$, the
inequality $\epsilon^2 \leq \infolowconst \, \EPSGENSQK$
implies that
\begin{align*}
\epsilon^2 \leq \infolowconst \Exs \ltwo{\ProjK g} \qquad \text{ and
}~ \epsilon^2 \leq \infolowconst \left(\frac{\Exs \ltwo{\ProjK
    g}}{\ltwo{\Exs \ProjK g}}\right)^2.
\end{align*}
The first inequality above implies, with $\infolowconst = 1/14$, that
\mbox{$\epsilon^2 \leq \Exs \ltwo{\ProjK g}/14 \leq (\Exs \ltwo{\ProjK
    g})^2/32$} (note that $\Exs \ltwo{\ProjK g} \geq 7$). Therefore
the assumption in Lemma~\ref{LemInfoKey} is satisfied so that
inequality~\eqref{EqnUBspicy} gives
\begin{align} 
\label{EqnUBmiddle}
\Exs_{\eta,\eta'} \exp(\epsilon^2 \inprod{\eta}{\eta'}) \, \leq \,
\frac{1}{a^2} \exp\left(5 \infolowconst + \frac{40 \infolowconst^2
  \Exs (\ltwo{\ProjK g}^2)}{(\Exs\ltwo{\ProjK g})^2} \right).
\end{align}
So it suffices to control the right hand side above. From the
concentration result in Lemma~\ref{LemConcentration}, we obtain 
\begin{align*}
a & = \Prob( \ltwo{\ProjK g} - \Exs \ltwo{\ProjK g} \geq - \frac{1}{2}
\Exs \ltwo{\ProjK g} ) \, \geq \, 1 - \exp(-\frac{(\Exs
  \ltwo{\ProjKcone g})^2}{8}) > 1 - \exp(-6),
\end{align*}
where the last step uses $\Exs \ltwo{\ProjK g} \geq 7$, and 
\begin{align*}
   \Exs\ltwo{\ProjK g}^2 =  (\Exs\ltwo{\ProjK g})^2 + \var(\ltwo{\ProjK g})
   \leq  (\Exs\ltwo{\ProjK g})^2 + 4.
\end{align*}
Here the last inequality follows from the fact that $\var(\ltwo{\ProjK
  g}) \leq 4$---see Lemma~\ref{LemConcentration}.  Plugging these two
inequalities into expression~\eqref{EqnUBmiddle} gives
\begin{align*}
  \Exs_{\eta,\eta'} \exp(\epsilon^2 \inprod{\eta}{\eta'}) \, \leq \,
  \left(\frac{1}{1-\exp(-6)} \right)^2 \exp\left(5 \infolowconst + 40
  \infolowconst^2 + \frac{160 \infolowconst^2}{(\Exs\ltwo{\ProjK
      g})^2} \right),
\end{align*}
where the right hand side is less than $2$ when $\infolowconst = 1/14$
and $\Exs\ltwo{\ProjK g} \geq 7.$ Combining with
inequality~\eqref{EqnTestErrorLowerBound} forces the testing error to
be lower bounded as
\begin{align*}
\forall \psi,~~ \UNIERR(\psi; \{0\}, \Kcone, \epsilon) \geq 1 -
\frac{1}{2} \sqrt{\Exs_{\eta, \eta'} \, \exp(\epsilon^2
  \inprod{\eta}{\eta'})-1} \geq \frac{1}{2} > \rho,
\end{align*}
which completes the proof of Theorem~\ref{ThmLBGen}.

%%%%%%%%%%%%%%%%%%%%%%%%%%%%%%%%%%%%%%%%%%%%%%%%%%%%%%%%%%%%%%%%%%%%%%%%%%%%%%%%%%%%%

\section{Discussion}
\label{SecDiscussion}

In this paper, we have studied the the problem of testing between two
hypotheses that are specified by a pair of non-oblique closed convex
cones.  Our first main result provided a characterization, sharp up to
universal multiplicative constants, of the testing radius achieved by
the generalized likelihood ratio test.  This characterization was
geometric in nature, depending on a combination of the Gaussian width
of an induced cone, and a second geometric parameter.  Due to the
combination of these parameters, our analysis shows that the GLRT can
have very different behavior even for cones that have the same
Gaussian width; for instance, compare our results for the circular and
orthant cone in Section~\ref{SecGLRTResults}.  It is worth noting that
this behavior is in sharp contrast to the situation for estimation
problems over convex sets, where it is understood that (localized)
Gaussian widths completely determine the estimation error associated
with the least-squares estimator~\cite{van1996weak,chatterjee2014new}.
In this way, our analysis reveals a fundamental difference between
minimax testing and estimation.

Our analysis also highlights some new settings in which the GLRT is
non-optimal.  Although past
work~\cite{warrack1984likelihood,MenSal91,perlman1999emperor} has
exhibited non-optimality of the GLRT in certain settings, in the
context of cones, all of these past examples involve oblique cones.
In Section~\ref{SecCartesianProductCone}, we gave an example of
sub-optimality which, to the best of our knowledge, is the first for a
non-oblique pair of cones---namely, the cone $\{0\}$, and a certain
type of Cartesian product cone.  

Our work leaves open various questions, and we conclude by
highlighting a few here.  First, in Section~\ref{SecGeneralLower}, we
proved a general information-theoretic lower bound for the minimax
testing radius.  This lower bound provides a sufficient condition for
the GLRT to be minimax optimal up to constants. Despite being tight in
many non-trivial situations, our information-theoretic lower bound is
not tight for all cones; proving such a sharp lower bound is an
interesting topic for future research.  Second, as with a long line of
past work on this
topic~\cite{raubertas1986hypothesis,menendnez1992testing,menendez1992dominance,
warrack1984likelihood},
our analysis is based on assuming that the noise variance $\sigma^2$
is known.  In practice, this may or may not be a realistic assumption,
and so it is interesting to consider the extension of our results to
this setting.

We note that our minimax lower bounds are proved by constructing prior
distributions on $\Hyp_0$ and $\Hyp_1$ and then control the distance
between marginal likelihood functions.  Following this idea, we can
also consider our testing problem in the Bayesian framework.  Without
any prior preference on which hypothesis to take, we will let
$\Pr(\Hyp_0) = \Pr(\Hyp_1) = 1/2$; thus the Bayesian testing procedure
makes decision based on quantity
\begin{align}
\label{EqnBF}
  B_{01} \defn \frac{m(y \mid \Hyp_0)}{m(y\mid \Hyp_1)} =
  \frac{\int_{\theta\in \ConeSmall} \Prob_\theta(y) \pi_1(\theta)
    d\theta}{\int_{\theta\in \ConeBig} \Prob_\theta(y) \pi_2(\theta)
    d\theta},
\end{align}
which is often called Bayesian factor in literature.  Analyzing the
behavior of this statistic is an interesting direction to pursue in
the future.

\subsection*{Acknowledgements}

This work was partially supported by Office of Naval Research MURI
grant DOD-002888, Air Force Office of Scientific Research Grant
AFOSR-FA9550-14-1-001, and National Science Foundation Grants
CIF-31712-23800 and DMS-1309356, NSF CAREER Grant DMS-16-54589

%%%%%%%%%%%%%%%%%%%%%%%%%%%%%%%%%%%%%%%%%%%%%%%%%%%%%%%%%%%%%%%%%%%%%%%%%%%%%%%%%

%%%%%%%%%%%%%%%%%%%%%%%%%%%%%%%%%%%%%%%%%%%%%%%%%%%%%%%%%%%%%%%%%%%%%%%%%%%%%%%%

% \vspace*{1cm}

%%%%%%%%%%% DO BIBLIOGRAPHY %%%%%%%%%%%%%%%%%%%%%%%%%%%%%%%%%%%%%%%%%
% \bibliographystyle{abbrv}
% \bibliography{../ConeBiblio}

%%%%%%%%%%%%%%%%%%%%%%%%%%%%%%%%%%%%%%%%%%%%%%%%%%%%%%%%%%%%%%%%%%%%%%%

\vspace*{2cm}

\appendix

% This supplementary material is organized as follows. 
% In Section~\ref{AppGLRTsub}, we first explain the intuition behind the example in Section \ref{SecCartesianProductRevisit} \cite{Wei17glrt} where the GLRT is shown to be sub-optimal, and construct a series of other cases where this sub-optimality is observed.
% We then provide the proofs of Propositions~\ref{PropKpos} and~\ref{PropMon} in 
% Sections~\ref{sec:proof_prop_kpos} and~\ref{sec:proof_prop_mon}, respectively. 
% It follows by some background on distance metrics and their properties in Section~\ref{SecTv-Chi}.
% The proofs of Theorem \ref{ThmGLRT} (a) and (b) are completed in Section~\ref{AppThmGLRTA}
% and \ref{AppThmGLRTB} respectively. The proofs of the lemmas for Theorem \ref{ThmLBGen} are collected  in Section \ref{AppThmLBGen}.
% Finally, the technical lemmas which were crucially used in the proofs of
% the Proposition~\ref{PropMon} and the monotone cone example are proved in Section \ref{Applemmas}.  

%%%%%%%%%%%%%%%%%%%%%%%%%%%%%%%%%%%%%%%%%%%%%%%%%%%%%%%%%%%%%%%%%%%%%%%%%%%%

\section{The GLRT sub-optimality}
\label{AppGLRTsub}
In this appendix, we first try to understand why the GLRT is sub-optimal for the Cartesian product cone $\Prc = \Cir{\usedim-1}(\alpha) \times \real$, and use this intuition to construct a more general class of
problems for which a similar sub-optimality is witnessed.  

\subsection{Why is the GLRT sub-optimal?}
\label{badex}
Let us consider tests with null $\ConeSmall = \{0\}$ and a general product
alternative of the form $\ConeBig = \Prc = K \times \real$, where $K
\subseteq \real^{d-1}$ is a base cone.  Note that $K =
\Cir{\usedim-1}$ in our previous example.

Now recall the decomposition~\eqref{EqnMulti} of the statistic $T$
that underlies the GLRT. By the product nature of the cone, we have
\begin{align*}
T(y) = \ltwo{\Pi_{\Prc} y} & = \ltwo{(\Pi_{K} (y_{-\usedim}),
  ~y_{\usedim})} = \sqrt{\ltwo{\Pi_{K} (y_{-\usedim})}^2 +
  \ltwo{y_{\usedim}}^2},
\end{align*}
where $y_{-\usedim} \defn (y_1,\ldots,y_{\usedim-1}) \in \real^{d-1}$
is formed from the first $\usedim-1$ coordinates of $y$.  Suppose that
we are interested in testing between the zero vector and a vector
$\theta^* = (0, \ldots, 0, \theta^*_\usedim)$, non-zero only in the
last coordinate, which belongs to the alternative.  With this
particular choice, under the null distribution, we have $y = \sigma g$
whereas under the alternative, we have $y = \theta^* + \sigma g$.
Letting $\Exs_0$ and $\Exs_1$ denote expectations under these two
Gaussian distributions, the performance of the GLRT in this direction
is governed by the difference
\begin{align*}
  \frac{1}{\sigma} \big \{\Exs_1[T(y)] - \Exs_0[T(y)] \big \} =
  \Exs_1
\sqrt{ \ltwo{\Pi_{K} (g_{-\usedim})}^2 +
  \ltwo{\frac{\theta^*_\usedim}{\sigma} + g_d}^2} \\
\quad \quad \quad - \Exs_0 \sqrt{\ltwo{\Pi_{K} (g_{-\usedim})}^2 +
  \ltwo{g_d}^2}.
\end{align*}
Note both terms in this difference involve a $(\usedim-1)$-dimensional
``pure noise'' component---namely, the quantity $\ltwo{\Pi_{K}
  (g_{-\usedim})}^2$ defined by the sub-vector $g_{-\usedim} \defn
(g_1, \ldots, g_{\usedim-1})$---with the only signal lying the last
coordinate.  For many choices of cone $K$, the pure noise component
acts as a strong mask for the signal component, so that the GLRT is
poor at detecting differences in the direction $\theta^*$. Since the
vector $\theta^*$ belongs to the alternative, this leads to
sub-optimality in its overall behavior. Guided by this idea, we can
construct a series of other cases where the GLRT is sub-optimal.  See
Appendix~\ref{AppNonOptGLRT} for details.

%%%%%%%%%%%%%%%%%%%%%%%%%%%%%%%%%%%%%%%%%%%%%%%%%%%%%%%%%%%%%%%%%%%%%%%%%%%%

\subsection{More examples on the GLRT sub-optimality}
\label{AppNonOptGLRT}

Now let us construct a larger class of product cones for
which the GLRT is sub-optimal.  For a given subset $S \subseteq \{1,
\ldots, \usedim \}$, define the subvectors $\theta_S = (\theta_i, i
\in S)$ and $\theta_{S^c} = (\theta_j, j \in S^c \}$, where $S^c$
denotes the complement of $S$.  For an integer $\ell \geq 1$, consider
any cone $\Kcone_\ell \subset \real^{\usedim}$ with the following two
properties:
\begin{itemize}
  \item its Gaussian width scales as $\Exs \Width(\Kcone_\ell \cap
    \Ball(1)) \asymp \sqrt{\usedim}$, and
  \item for some fixed subset $\{1,2\ldots,\usedim\}$ of cardinality
    $\ell$, there is a scalar $\gamma > 0$ such that
\begin{align*}
  \|\theta_S\|_2 \geq \gamma \|\theta_{S^c}\|_2 \quad \mbox{for all
    $\theta \in \Kcone_\ell$.}
\end{align*}
\end{itemize}
As one concrete example, it is easy to check that the circular cone is
a special example with $\ell = 1$ and $\gamma = 1/\tan(\alpha)$.  The
following result applies to the GLRT when applied to testing the null
$C_1 = \{0\}$ versus the alternative $C_2 = \Kcone^s_\times = \Kcone
\times \real$.
\begin{prop}
\label{PropSubOpt}
For the previously described cone testing problem, the GLRT testing
radius is sandwiched as
\begin{align*}
   \epsglrt^2 \asymp \sqrt{\usedim} \sigma^2,
 \end{align*} 
whereas a truncation test can succeed at radius $\epsilon^2 \asymp
\sqrt{\ell} \sigma^2$.
\end{prop}

\begin{proof}
The claimed scaling of the GLRT testing radius follows as a corollary
of Theorem~\ref{ThmGLRT} after a direct evaluation of $\EPSCRITSQ$.
In order to do so, we begin by observing that
\begin{align*}
  \inf_{\eta \in C_2 \times \Sphere{}} \inprod{\eta}{\Exs\Pi_{C_2} g}
  & \leq \inprod{e_\usedim}{\Exs \Pi_{c_2} g} \; = \; 0, \quad
  \mbox{and} \\
 \Exs \Width(C_2\cap \Ball(1)) &= \Exs \ltwo{\Pi_{C_2} g} \asymp
 \sqrt{\usedim}
\end{align*}
which implies that $\EPSCRITSQ \asymp \sqrt{\usedim}$, and hence
implies the sandwich claim on the GLRT via Theorem~\ref{ThmGLRT}.

On the other hand, for some pre-selected $\beta > 0$, consider the
truncation test
\begin{align*}
  \varphi(\yvec) & \defn \Ind \big[ \|\yvec_S \|_2 \geq \beta \big],
\end{align*}
This test can be viewed as a GLRT for testing the zero null against
the alternative $\real^\ell$, and hence it will succeed with
separation \mbox{$\epsilon^2 \asymp \sigma^2 \sqrt{\ell}$.}  Putting
these pieces together, we conclude that the GLRT is sub-optimal
whenever $\ell$ is of lower order than $\usedim$.

\end{proof}

%%%%%%%%%%%%%%%%%%%%%%%%%%%%%%%%%%%%%%%%%%%%%%%%%%%%%%%%%%%%%%%%%%%%%%%%%%%%

\section{Proofs for Proposition 1 and 2}
In this section, we complete the proofs of Propositions~\ref{PropKpos} and~\ref{PropMon} in 
Sections~\ref{sec:proof_prop_kpos} and~\ref{sec:proof_prop_mon}, respectively. 

\subsection{Proof of Proposition~\ref{PropKpos}}
\label{sec:proof_prop_kpos}

As in the proof of Theorem~\ref{ThmGLRT} and Theorem~\ref{ThmLBGen}, we can assume without loss of generality that
$\sigma = 1$ since $\Kpos$ is invariant under
rescaling by positive numbers.
We split our proof into two cases, depending on whether or not the
dimension $\usedim$ is less than $81$.

\paragraph{Case 1:}  First suppose that
$\usedim < 81$.  If the separation is upper bounded as $\epsilon^2
\leq \infolowconst \sqrt{\usedim}$, then setting
$\infolowconst = 1/18$ yields
\begin{align*}
\epsilon^2 \leq \infolowconst \sqrt{\usedim} < 1/2.
\end{align*}
Similar to our proof for Theorem~\ref{ThmGLRT}(b) Case 1, if
$\epsilon^2 < 1/2$, every test yields testing error no smaller
than $1/2$. It is seen by considering a simple verses simple testing
problem~\eqref{EqnSimpleTesting}.  So our lower bound directly holds
for the case when $\usedim < 81$ satisfies.

\paragraph{Case 2:} Let us consider the case when dimension $\usedim \geq 81$.
The idea is to make use of our Lemma~\ref{LemChisquareBound} to show that the testing error is at least $\rho$ whenever $\epsilon^2 \leq \infolowconst \, \sqrt{\usedim}.$
In order to apply Lemma~\ref{LemChisquareBound}, the key is to construct a probability measure $\qprob$ supported on set $\Kcone \cap
\NewBall^c(1)$ such that for i.i.d. pair $\eta, \eta'$ drawn from $\qprob$, quantity $\Exs e^{\lambda \inprod{\eta}{\eta'}}$ can be well controlled. 
We claim that there exists such a probability measure $\qprob$ that
\begin{align} 
\label{EqnNonNegLB}
  \Exs_{\eta, \eta'} e^{\lambda \inprod{\eta}{\eta'}} \leq
  \exp\left(\exp \left(\frac{2 + \lambda}{\sqrt{\usedim}-1}\right)- \left(1-\frac{1}{\sqrt{d}}
  \right)^2 \right) \qquad{\text {where } \lambda \defn
    \epsilon^2.}
\end{align}
Taking inequality~\eqref{EqnNonNegLB} as given for now, letting
$\infolowconst = 1/8$, we have $\lambda = \epsilon^2 \leq 
\sqrt{\usedim}/8$. So the right hand side in expression~\eqref{EqnNonNegLB} can be further upper bounded as
\begin{align*}
  \exp\left(\exp \left(\frac{2}{\sqrt{\usedim}-1} + \frac{\sqrt{\usedim}}{\sqrt{\usedim} - 1} \frac{\lambda}{\sqrt{\usedim}} \right)- \left(1-\frac{1}{\sqrt{d}}
  \right)^2 \right)
  &\leq 
  \exp\left(\exp \left(\frac{1}{4} + \frac{9}{8}\cdot \frac{1}{8} \right)- \left(1-\frac{1}{9}\right)^2 \right) \\
  &< 2,
\end{align*}
where we use the fact that $\usedim \geq 81$. 
As a consequence of Lemma~\ref{LemChisquareBound}, the testing error of
every test satisfies
\begin{align*}
\inf_{\psi} \UNIERR(\psi; \{0\}, \Kpos, \epsilon) \geq 1 - \frac{1}{2}
\sqrt{\Exs_{\eta, \eta'} \, \exp(\epsilon^2
  \inprod{\eta}{\eta'})-1} > \frac{1}{2} \geq \rho.
\end{align*}
Putting these two cases together, our lower bound holds for any
dimension thus we complete the proof of Proposition~\ref{PropKpos}.

So it only remains to construct a probability measure $\qprob$ such that the
inequality~\eqref{EqnNonNegLB} holds.  We begin by introducing some
helpful notation.  For an integer $s$ to be specified, consider a
collection of vectors $\KposSet$ containing all $\usedim$-dimensional
vectors with exactly $s$ non-zero entries and each non-zero entry
equals to $1/\sqrt{s}.$
Note that there are in total $M \defn {\usedim \choose s}$ vectors of this
type.  
Letting $\qprob$ be the uniform distribution over this set of vectors namely
\begin{align} \label{EqnKposMeasure}
  \qprob( \{\eta\} ) \defn \frac{1}{M}, 
  \qquad \eta \in \KposSet.
\end{align}
Then we can write the expectation as
\begin{align*}
  \Exs e^{\lambda \inprod{\eta}{\eta'}} = \frac{1}{M^2} \sum_{\eta,
    \eta' \in \KposSet} e^{\lambda \inprod{\eta}{\eta'}}.
\end{align*}
Note that the inner product $\inprod{\eta}{\eta'}$ takes values $i/s$, for integer $i \in \{0, 1, \ldots, s \}$ and 
given every vector $\eta$ and integer $i \in \{0, 1, \ldots, s \}$, the
number of $\eta'$ such that $\inprod{\eta}{\eta'} = i/s$ equals to
${s \choose i}{\usedim-s \choose s-i}$. 
Consequently, we obtain
\begin{align}
\label{EqnChiSquareUpper}
\Exs e^{\lambda \inprod{\eta}{\eta'}} = {\usedim \choose s}^{-1}
\sum_{i=0}^s {s \choose i} {\usedim-s \choose s-i} e^{\lambda i/s}
= \sum_{i=0}^s \frac{\SPEC_i z^i}{i !},
\end{align}
where
\begin{align*}
z \defn e^{\lambda/s} \text{ and } \SPEC_i \defn
\frac{(s!(\usedim-s)!)^2}{((s-i)!)^2 \usedim!(\usedim-2s+i)!}.
\end{align*}

Let us set integer $s \defn \lfloor \sqrt{\usedim} \rfloor$.
We claim quantity $\SPEC_i$ satisfies the following bound
\begin{align} \label{EqnNonNe-Ai}
\SPEC_i \leq \exp\Big(- (1-\frac{1}{\sqrt{d}})^2 +
\frac{2i}{\sqrt{\usedim}-1}\Big) \qquad \mbox{for all $i \in \{0, 1,
  \ldots, s\}$.}
\end{align} 
Taking expression~\eqref{EqnNonNe-Ai} as given for now and plugging into inequality \eqref{EqnChiSquareUpper}, we have
\begin{align*}
\Exs e^{\lambda \inprod{\eta}{\eta'}} &\leq \exp \Big(-
(1-\frac{1}{\sqrt{d}})^2 \Big)\sum_{i=0}^{s} \frac{(z\exp(\frac{2}{\sqrt{\usedim}-1}))^i}{i !} \\
& \, \stackrel{(a)}{\leq} \,
\exp \Big(- (1-\frac{1}{\sqrt{d}})^2 \Big) \exp \left( z \exp(\frac{2}{\sqrt{\usedim}-1})\right)\\
& \, \stackrel{(b)}{\leq} \,
\exp \left(- \left(1-\frac{1}{\sqrt{d}}\right)^2 + \exp \left(\frac{2+\lambda}{\sqrt{\usedim}-1}\right) \right),
\end{align*}
where step (a) follows from the standard power series expansion
$e^x = \sum_{i=0}^\infty \frac{x^i}{i!}$ and step (b) follows by $z = e^{\lambda/s}$ and $s = \lfloor \sqrt{\usedim} \rfloor > \sqrt{d} - 1$. 
Therefore it verifies inequality~\eqref{EqnNonNegLB} and complete our argument.

It is only left for us to check inequality~\eqref{EqnNonNe-Ai} for $A_i$.
Using the fact that $1 - x \leq e^{-x}$, it is guaranteed that 
\begin{subequations}
\begin{align}
\label{EqnHfun1}  
\SPEC_0 = \frac{((\usedim - s)!)^2}{\usedim! (\usedim - 2s)!}  &= (1 -
\frac{s}{\usedim}) (1 - \frac{s}{\usedim-1})\cdots(1 -
\frac{s}{\usedim-s+1}) \leq \exp( - s\sum_{i=1}^s
\frac{1}{\usedim-s+i}).
\end{align}
Recall that integer $s = \lfloor \sqrt{\usedim} \rfloor$, then we
can bound the sum in expression~\eqref{EqnHfun1} as
\begin{align*}
s\sum_{i=1}^s \frac{1}{\usedim-s+i} \geq s \sum_{i=1}^{s}
\frac{1}{\usedim} = \frac{s^2}{\usedim} \geq (1-\frac{1}{\sqrt{d}})^2,
\end{align*}
which, when combined with inequality~\eqref{EqnHfun1}, implies that
$\SPEC_0 \leq \exp( - (1-\frac{1}{\sqrt{d}})^2 )$.

Moreover, direct calculations yield
\begin{align} 
\label{EqnHfun2} 
\frac{\SPEC_i}{\SPEC_{i-1}} = \frac{(s-i+1)^2}{\usedim-2s+i},
\qquad 1\leq i\leq s.
\end{align}
\end{subequations}
This ratio is decreasing with index $i$ as $1\leq i\leq s$, thus is upper bounded by $A_1/A_0$, which
implies that
\begin{align*}
\frac{\SPEC_i}{\SPEC_{i-1}} \leq
\frac{\usedim}{\usedim-2\sqrt{\usedim}+1} = (1 +
\frac{1}{\sqrt{\usedim}-1})^2 \leq \exp(\frac{2}{\sqrt{\usedim}-1}),
\end{align*}
where the last inequality follows from $1+x \leq e^x.$ Putting pieces
together validates bound \eqref{EqnNonNe-Ai} thus finishing the proof
of Proposition~\ref{PropKpos}.

%%%%%%%%%%%%%%%%%%%%%%%%%%%%%%%%%%%%%%%%%%%%%%%%%%%%%%%%%%%%%%%%%%%%%%

\subsection{Proof of Proposition~\ref{PropMon}}
\label{sec:proof_prop_mon}

As in the proof of Theorem~\ref{ThmGLRT} and Theorem~\ref{ThmLBGen}, we can assume without loss of generality that
$\sigma = 1$ since $\LinSpace$ and $\Mon$ are both invariant under
rescaling by positive numbers.

We split our proof into two cases, depending on whether or not $\sqrt{\log (e\usedim)} < 14.$

\paragraph{Case 1:}  First suppose $\sqrt{\log (e\usedim)} < 14$, so that the choice $\infolowconst = 1/28$
yields the upper bound
\begin{align*}
\epsilon^2 \leq \infolowconst \sqrt{\log (e\usedim)} < 1/2.
\end{align*}
Similar to our proof of the lower bound in Theorem~\ref{ThmGLRT}, by
reducing to a simple testing problem~\eqref{EqnSimpleTesting}, any
test yields testing error no smaller than $1/2$ if $\epsilon^2 <
1/2$.  Thus, we conclude that the stated lower bound holds when $\sqrt{\log (e\usedim)} < 14$.

\paragraph{Case 2:}  Otherwise, we may assume that $\sqrt{\log (e\usedim)} \geq 14$.  In this case, we exploit
Lemma~\ref{LemChisquareBound} in order to show that the testing error
is at least $\rho$ whenever $\epsilon^2 \leq \infolowconst \, 
\sqrt{\log (e\usedim)}$. 
Doing so requires constructing a probability measure $\qprob_\LinSpace$ supported
on $\Mon \cap \LinSpacePerp \cap \NewBall^c(1)$ such that the expectation $\Exs
e^{\epsilon^2 \inprod{\eta}{\eta'}}$ can be well controlled, where
$(\eta, \eta')$ are drawn i.i.d according to $\qprob_\LinSpace$. 
Note that $\LinSpace$ can be either $\{0\}$ or $\myspan(\ONES)$.

Before doing that, let us first introduce some notation. 
Let $\delta \defn 9$ and $r \defn 1/3$ (note that $\delta = r^{-2}$). Let 
\begin{align} \label{EqnM}
m \defn \max \left \{ n ~\Big\vert~ \sum_{i=1}^n \lfloor \frac{\delta-1}{\delta^i} (\usedim + \log_\delta \usedim + 3)\rfloor \, < \, \usedim \right \}.  
\end{align}
We claim that the integer $m$ defined above satisfies:
\begin{align} \label{EqnSandwichM}
  \lceil \frac{3}{4} \log_\delta (\usedim) \rceil + 1 \leq  m \leq \lceil \log_\delta \usedim \rceil,
 \end{align} 
where $\lceil x \rceil$ denotes the smallest integer that is greater than or equal to $x$.
To see this, notice that for $t = \lceil \frac{3}{4} \log_\delta (\usedim) \rceil + 1$, we have 
\begin{align*}
  \sum_{i=1}^{t} \lfloor \frac{\delta-1}{\delta^i} (\usedim + \log_\delta \usedim + 3) \rfloor
  \leq \sum_{i=1}^{t} \frac{\delta-1}{\delta^i} (\usedim + \log_\delta \usedim + 3)
  & = (1-\frac{1}{\delta^t})(\usedim + \alpha)\\
  & \, \stackrel{(i)}{\leq} \, \usedim + \alpha - \frac{\usedim + \alpha}{\delta^2 \usedim^{3/4}} 
  \, \stackrel{(ii)}{<} \, \usedim,
\end{align*}
where we denote $\alpha \defn \log_\delta \usedim + 3.$
The step (i) follows by definition that $t = \lceil \frac{3}{4} \log_\delta (\usedim) \rceil + 1$ while step (ii) holds because as $\sqrt{\log (e\usedim)} \geq 14$, we have $\alpha = \log_\delta \usedim + 3 < \usedim^{1/4}/\delta^2$.
On the other hand, for $t  = \lceil \log_\delta \usedim \rceil$, we have 
\begin{align*}
   \sum_{i=1}^{t} \lfloor \frac{\delta-1}{\delta^i} (\usedim + 
   \log_\delta \usedim + 3) \rfloor 
   &\geq \,
   \sum_{i=1}^{t} \frac{\delta-1}{\delta^i} (\usedim + \alpha) - t\\
   &= (1-\frac{1}{\delta^t})(\usedim + \alpha) - t\\
   &> \usedim + \alpha - \frac{\usedim + \alpha}{\usedim} - (\log_\delta \usedim + 1),
\end{align*}
where the last step uses fact $t = \lceil \log_\delta \usedim \rceil$. 
Since when $\sqrt{\log (e\usedim)} \geq 14$, we have $\alpha = \log_\delta \usedim + 3 < \usedim$, therefore $(\usedim + \alpha)/\usedim + \log_\delta \usedim + 1 \leq 2+ \log_\delta \usedim + 1 = \alpha$, which guarantees that  
\begin{align*}
   \sum_{i=1}^{t} \lfloor \frac{\delta-1}{\delta^i} (\usedim + 
   \log_\delta \usedim + 3) \rfloor > \usedim.
\end{align*}
We thereby established inequality~\eqref{EqnSandwichM}.

We now claim that there exists a probability measure $\qprob_\LinSpace$ supported on $\Mon\cap\LinSpacePerp\cap \NewBall^c(1)$ such that
\begin{align}
\label{EqnMoMoLB}
\Exs_{\eta, \eta' \sim \qprob_\LinSpace} e^{\lambda \inprod{\eta}{\eta'}} \leq
\exp\left(\exp \left(\frac{9\lambda/4 + 2}{\sqrt{m}-1}\right) - \left(1-\frac{1}{\sqrt{m}} \right)^2 + \frac{27 \lambda }{32 (\sqrt{m}-1)}\right),
~~
\text{ where } \lambda \defn \epsilon^2.
\end{align}
Recall that we showed in inequality~\eqref{EqnSandwichM} that $m \geq \lceil \frac{3}{4} \log_\delta (\usedim) \rceil + 1$.
Setting $\infolowconst = 1/62$ implies that whenever $\epsilon^2 \leq \infolowconst \sqrt{\log (e\usedim)}$, we have
\begin{align}
  \epsilon^2 \leq \frac{1}{62} \sqrt{\log (e\usedim)} =  
  \frac{1}{62} \sqrt{1 + \frac{4}{3} \log \delta \cdot \frac{3}{4} \log_{\delta} \usedim}
  \, \leq \,
  \frac{1}{62} \sqrt{\frac{4}{3} \log \delta \left(1+\frac{3}{4} \log_{\delta} \usedim \right)}
  \, \leq \,
  \frac{1}{36} \sqrt{m}.
\end{align}
So the right hand side in expression~\eqref{EqnMoMoLB}
can be made less than 2 by
\begin{align*}
  & \exp \left(\frac{9\lambda/4 + 2}{\sqrt{m}-1}\right) - \left(1-\frac{1}{\sqrt{m}} \right)^2 + \frac{27 \lambda }{32 (\sqrt{m}-1)}\\
  \leq &
  \exp \left(
  \frac{9\lambda}{4\sqrt{m}} \frac{\sqrt{m}}{\sqrt{m}-1} + \frac{2}{7}\right) - \left(1-\frac{1}{8} \right)^2 + \frac{27 \lambda }{32\sqrt{m}}
  \frac{\sqrt{m}}{\sqrt{m}-1}\\
  \leq &
  \exp \left( 
  \frac{9}{4\cdot 36} \frac{8}{7} + \frac{2}{7}\right) - \left(1-\frac{1}{8}  \right)^2 + \frac{27}{32 \cdot 36} \frac{8}{7} < \log 2,
\end{align*}
where we use the fact that $\sqrt{m} \geq \sqrt{1+\frac{3}{4} \log_{\delta} \usedim} \geq 8.$
Lemma~\ref{LemChisquareBound} thus
guarantees the testing error to be no less than
\begin{align*}
\inf_{\psi} \UNIERR(\psi; \LinSpace, \Mon, \epsilon) \geq 1 - \frac{1}{2}
\sqrt{\Exs_{\eta, \eta'} \, \exp(\epsilon^2 \inprod{\eta}{\eta'})-1} > \frac{1}{2} \geq \rho, 
\end{align*}
which leads to our result in Proposition~\ref{PropMon}.

Now it only remains to construct a probability measure $\qprob_\LinSpace$ with the right
support such that inequality~\eqref{EqnMoMoLB} holds.  To do this, we
make use of a fact from the proof of Proposition~\ref{PropKpos} for the orthant cone
$\Kpos \subset \real^m$. 
Recall that to establish Proposition~\ref{PropKpos}, we constructed a probability measure $\dprob$ supported on $\Kpos \cap \Sphere{m} \subset
\real^m$ such that if $b, \bprim$ are an i.i.d pair drawn from $\dprob$, we have
\begin{align} 
\label{Eqnbjduck}
  \Exs_{b,b' \sim \dprob} e^{\lambda \inprod{b}{b'}} \leq
  \exp\left( \exp\left(\frac{2 + \lambda}{\sqrt{m}-1}\right)-
  \left(1-\frac{1}{\sqrt{m}} \right)^2 \right).
\end{align}
By construction, $\dprob$ is a uniform probability measure on the finite set $\KposSet$ which consists of all vectors in $\real^{m}$ which have $s$ non-zero entries which are all equal to $1/\sqrt{s}$ where $s = \lfloor \sqrt{m} \rfloor.$

Based on this measure $\dprob$, let us define $\qprob_\LinSpace$ as in the following lemma and establish some of its properties under the assumption that $\sqrt{\log (e\usedim)} \geq 14.$

\begin{lem} 
\label{LemSupport}
Let $\OldAmat$ be the $m \times m$ lower triangular matrix given by
\begin{subequations}
\begin{align} \label{EqnOldAmat}
  \OldAmat \defn
  \begin{pmatrix}
   1 & \\ r & 1 & \\ r^2 & r & 1\\ \vdots & \vdots & & \ddots
   \\ r^{m-1} & r^{m-2} & \cdots & & 1
  \end{pmatrix}.
\end{align} 
There exists an $\usedim \times m$ matrix $\OldLmat$ such that  
\begin{align} \label{EqnFF-ind}
  \OldLmat^T \OldLmat \, = \, \Ind_m
\end{align}  
\end{subequations}
and such that for every $b\in \KposSet$ and $\eta \defn \OldLmat \OldAmat b$, we have  
\begin{enumerate}
  \item $\eta \in \Mon \cap \LinSpacePerp \cap \NewBall^c(1)$ if $\LinSpace = \{0\}$, and 

  \item $\eta - \bar{\eta} \ONES \in \Mon\cap\LinSpacePerp \cap \NewBall^c(1)$ if $\LinSpace = \myspan(\ONES)$, where $\bar{\eta} = \sum_{i=1}^\usedim \eta_i /\usedim$ denotes the mean of the vector $\eta$.
\end{enumerate}

\end{lem}
\noindent See Appendix~\ref{sec:LemSupport} for the proof of this
claim.

% such that if we define random vector $\eta$ as
% \begin{align} 
% \label{EqnDefEta}
% \eta \defn \OldLmat \OldAmat b, \qquad{b \sim \dprob}, 
% \end{align}
% then the distributions of $\eta$, $\eta-\bar{\eta}$ are supported on  with subspace $\LinSpace = \{0\}$ or  correspondingly.

If $\LinSpace = \{0\}$, let probability measure $\qprob_\LinSpace$ be defined as the distribution of $\eta \defn \OldLmat \OldAmat b$ where $b\sim \dprob.$ 
Otherwise if $\LinSpace = \myspan(\ONES)$, let $\qprob_\LinSpace$ be
the distribution of $\eta - \bar{\eta}\ONES$ where again $\eta \defn \OldLmat \OldAmat b$ and $b\sim \dprob.$   
From Lemma~\ref{LemSupport} we know that $\qprob_\LinSpace$
is supported on $\Mon \cap \LinSpacePerp \cap \NewBall^c(1)$. It only
remains to verify the critical inequality~\eqref{EqnMoMoLB} to complete the proof of
Proposition~\ref{PropMon}. 
Let $\eta = \OldLmat \OldAmat b$ and $\eta' = \OldLmat \OldAmat \bprim$ with 
$b,\bprim$ being i.i.d having distribution $\dprob$.
Using the fact that $\OldLmat^T\OldLmat = \Ind_m$, we can write the inner product of $\eta,\eta'$ as
\begin{align*}
  \inprod{\eta}{\eta'} = b^T \OldAmat^T \OldLmat^T \OldLmat \OldAmat b' = \inprod{\OldAmat b}{\OldAmat b'}.
\end{align*}
The following lemma relates inner product $\inprod{\eta}{\eta'}$ to $\inprod{b}{\bprim}$, and thereby allows us to derive inequality~\eqref{EqnMoMoLB} based on inequality \eqref{Eqnbjduck}.
Recall that $\KposSet$ consists of all vectors in $\real^{m}$ which have $s$ non-zero entries which are all equal to $1/\sqrt{s}$ where $s = \lfloor \sqrt{m} \rfloor.$

\begin{lem}
\label{LemBeta-to-Eta}
For every $b, \bprim \in \KposSet$, we have
\begin{subequations}
\begin{align} 
\label{EqnGInprod}
\inprod{\OldAmat b}{\OldAmat b'} &~\leq~ \frac{\inprod{b}{\bprim}}{
  (1-r)^2} + \frac{r}{s (1-r)^2 (1-r^2)},\\
\label{EqnNorm}
\ltwo{\OldAmat b}^2 &~\geq~ \frac{1}{(1-r)^2} - \frac{2r + r^2}{s(1-r^2)(1-r)^2}.
\end{align}
\end{subequations}
\end{lem}

\noindent See Appendix~\ref{AppLemBeta-to-Eta} for the proof of this
claim.

We are now ready to prove inequality~\eqref{EqnMoMoLB}. We consider the two cases $\LinSpace = \{0\}$ and $\LinSpace = \myspan(\ONES)$ separately. 

For $\LinSpace = \{0\}$, recall that $r = 1/3$ and $s = \lfloor \sqrt{m} \rfloor \geq \sqrt{m} - 1$. Therefore as a direct
consequence of inequality~\eqref{EqnGInprod}, we have
\begin{align} 
  \label{EqnPlotMatrix}
\Exs_{\eta,\eta \sim \qprob} e^{\lambda \inprod{\eta}{\eta'}} \leq \Exs_{b,\bprim \sim \dprob} \exp \left(\frac{9\lambda}{4}
 \inprod{b}{\bprim} + \frac{27 \lambda }{32 (\sqrt{m}-1)}\right).
\end{align}
Combining inequality~\eqref{EqnPlotMatrix} with \eqref{Eqnbjduck}
completes the proof of inequality~\eqref{EqnMoMoLB}.

Let us now turn to the case when $\LinSpace = \myspan(\ONES)$.  The proof is essentially the same as for $\LinSpace = \{0\}$ with only some minor changes.  Again our goal is to check
inequality~\eqref{EqnMoMoLB}. For this, we write 
\begin{align*}
\Exs_{\eta, \eta'\sim \qprob_{\LinSpace}} e^{\lambda \inprod{\eta}{\eta'}} =
\Exs_{\eta, \eta'\sim \qprob_{\{0\}}} e^{\lambda \inprod{\eta -
    \bar{\eta}\ONES}{\eta' - \bar{\eta'}\ONES}} ~\leq~ \Exs_{\eta,
  \eta'\sim \qprob_{\{0\}}} e^{\lambda \inprod{\eta}{\eta'}}.
\end{align*}
Here the last step use the fact that $\inprod{\eta -
  \bar{\eta}\ONES}{\eta' - \bar{\eta'}\ONES} = \inprod{\eta}{\eta'} -
d\bar{\eta}\bar{\eta'} \leq \inprod{\eta}{\eta'}$ where the last inequality follows from the non-negativity of every entry of vectors $\eta$ and $\eta'$ (this non-negativity is a consequence of the non-negativity of $\OldLmat$ and $\OldAmat$ from Lemma~\ref{LemSupport} and non-negativity of vectors in $\KposSet$).

Thus, we have completed the proof of Proposition~\ref{PropMon}.

%%%%%%%%%%%%%%%%%%%%%%%%%%%%%%%%%%%%%%%%%%%%%%%%%%%%%%%%%%%%%%%%%%%%%%%

%%%%%%%%%%%%%%%%%%%%%%%%%%%%%%%%%%%%%%%%%%%%%%%%%%%%%%%%%%%%%%%%%%%%%%%%%%%%%%%%%%%%

\section{Distances and their properties}
\label{SecTv-Chi}

Here we collect some background on distances between probability
measures that are useful in analyzing testing error.  Suppose
$\mprob_1$ and $\mprob_2$ are two probability measures on Euclidean
space $(\real^\usedim, \mathcal{B})$ equipped with Lebesgue
measure. For the purpose of this paper, we assume $\mprob_1 \ll
\mprob_2$.  The \emph{total variation} (TV) distance between
$\mprob_1$ and $\mprob_2$ is defined as
\begin{subequations}
\begin{align}
\tvnorm{\mprob_1 - \mprob_2} & \defn \sup_{B \in \mathcal{B}}
|\mprob_1(B) - \mprob_2(B)| = \frac{1}{2} \int |d \mprob_1 - d
\mprob_2|.
\end{align}
A closely related measure of distance is the \emph{$\chi^2$ distance}
given by
\begin{align}
\chi^2(\mprob_1, \mprob_2) & \defn \int (\frac{d \mprob_1}{d \mprob_2}
- 1)^2 d \mprob_2.
\end{align}
For future reference, we note that the TV distance and $\chi^2$
distance are related via the inequality
\begin{align} \label{EqnChi-TV}
\tvnorm{\mprob_1 - \mprob_2} \leq \frac{1}{2} \sqrt{\chi^2(\mprob_1,
  \mprob_2)}.
\end{align}
\end{subequations}

%%%%%%%%%%%%%%%%%%%%%%%%%%%%%%%%%%%%%%%%%%%%%%%%%%%%%%%%%%%%%%%%%%%%%%%%%%%%%%

\section{Auxiliary proofs for Theorem 1 (a)}
\label{AppThmGLRTA}

In this appendix, we collect the proofs of lemmas involved in the
proof of Theorem~\ref{ThmGLRT}(a).

%%%%%%%%%%%%%%%%%%%%%%%%%%%%%%%%%%%%%%%%%%%%%%%%%%%%%%%%%%%%%%%%%%%%%%%%%%%%%%%%%%%%%%

\subsection{Proof of Lemma~\ref{LemConcentration}}
\label{AppLemConcentration}
Let us start with the statement with this lemma. 

\begin{lem} 
\label{LemConcentration}
For a standard Gaussian random vector $g \sim \NORMAL(0,
\EYE{\usedim})$, closed convex cone $\Kcone \in \real^\usedim$ and
vector $\theta \in \real^\usedim$, we have
\begin{subequations}
\begin{align} 
\label{EqnLipschitz1}
\Prob \Big( \pm (Z(\theta) - \Exs[Z(\theta)]) \geq t \Big) &
\leq \exp \Big(-\frac{t^2}{2} \Big), \qquad \mbox{and} \\
\label{EqnLipschitz2}
\Prob \Big( \pm (\inprod{\theta}{\ProjK g} - \Exs \inprod{\theta}{\ProjK  g}) \geq t \big) & \leq \exp \Big(-\frac{t^2}{2 \ltwo{\theta}^2}
\Big),
\end{align}
\end{subequations}
where both inequalities hold for all $t \geq 0$.
\end{lem}
\noindent 

For future reference, we also note that tail
bound~\eqref{EqnLipschitz1} implies that the variance is bounded as
\begin{align}
  \label{EqnVarianceBound}
  \var(Z(\theta)) & = \int_0^\infty \Prob \Big( \big| Z(\theta) -
  \Exs[Z(\theta)] \big| \geq \sqrt{u} \Big) du \; \leq \;
  2 \int_0^\infty e^{-u/2} du \; = \; 4.
\end{align}

To prove Lemma~\ref{LemConcentration}, given every vector $\theta$, we claim that the function $g \mapsto
\ltwo{\ProjKcone(\theta + g)}$ is $1$-Lipschitz, whereas the function
\mbox{$g \mapsto \inprod{\theta}{\ProjKcone g}$} is a
$\ltwo{\theta}$-Lipschitz function.  From these claims, the
concentration results then follow from Borell's
theorem~\cite{borell1975brunn}.

In order to establish the Lipschitz property, consider two vectors $g,
g' \in \real^\usedim$.  By the triangle inequaliuty non-expansiveness
of Euclidean projection, we have
\begin{align*} 
\Big| \ltwo{\ProjKcone(\theta + g)} - \ltwo{\ProjKcone(\theta + g')}
\Big| & \leq \ltwo{\ProjKcone(\theta + g) - \ProjKcone(\theta + g')}
\; \leq \; \ltwo{g - g'}.
\end{align*}
Combined with the Cauchy-Schwarz inequality, we conclude that
\begin{align*}
  \big|\inprod{\theta}{\ProjK g} - \inprod{\theta}{\ProjK g'}\big|
        \leq \ltwo{\theta} \; \ltwo{\ProjK g - \ProjK g' } \leq
        \ltwo{\theta} \; \ltwo{g - g'},
\end{align*}
which completes the proof of Lemma~\ref{LemConcentration}.

%%%%%%%%%%%%%%%%%%%%%%%%%%%%%%%%%%%%%%%%%%%%%%%%%%%%%%%%%%%%%%%%%%%%%%%%%%%%%%%%%%%%%%%%%%%%%

\subsection{Proof of inequality~\eqref{Eqnc1.eq}}
\label{AppLemLSEup}

To prove inequality~\eqref{Eqnc1.eq}, we make use of
the following auxiliary Lemma~\ref{LemLSEup}.

% Let us first state our Lemma~\ref{LemLSEup} and give a proof of it. 

\begin{lem} 
\label{LemLSEup}
For every closed convex cone $\Kcone$ and vector $\theta \in \Kcone$, we
have the lower bounds
\begin{subequations}
\begin{align}
\label{Eqngl1.eq}
\Gamma(\theta) & \geq \frac{\ltwo{\theta}^2}{2 \ltwo{\theta} + 8\Exs
  \ltwo{\ProjK g}} -\frac{2}{\sqrt{e}}.
\end{align}
Moreover, for any vector $\theta$ that also satisfies the inequality
$\inprod{\theta}{\Exs \ProjK g} \geq \ltwo{\theta}^2 $, we have
\begin{align}
\label{Eqngl2.eq}
\Gamma(\theta) & \geq \alpha^2(\theta) \frac{\inprod{\theta}{\Exs
    \ProjK g} - \ltwo{\theta}^2} {\alpha(\theta) \ltwo{\theta} + 2
  \Exs \ltwo{\ProjK g}} -\frac{2}{\sqrt{e}},
\end{align}
where $\alpha(\theta) \defn 1 - \exp \left(\frac{-
  \inprod{\theta}{\Exs \ProjK g}^2}{8 \ltwo{\theta}^2} \right)$.
\end{subequations}
\end{lem}

We now use Lemma \ref{LemLSEup} to prove the lower bound~\eqref{Eqnc1.eq}.
Note that the inequality $\ltwo{\theta}^2
\geq \upconst \NEWEPSCRITSQ$ implies that one of the following two
lower bounds must hold:
\begin{subequations}
\begin{align}
\label{Eqnga1}
\ltwo{\theta}^2 &\geq \upconst \Exs \ltwo{\ProjK g},\\
\label{Eqnga2}
\text{ or } ~~~\inprod{\theta}{\Exs \ProjK g} &\geq \sqrt{\upconst}
\Exs \ltwo{\ProjK g}.
\end{align} 
\end{subequations}
We will analyze these two cases separately.

\paragraph{Case 1} In order to show that the lower
bound~\eqref{Eqnga1} implies inequality~\eqref{Eqnc1.eq}, we will
prove a stronger result---namely, that the inequality $\ltwo{\theta}^2
\geq \sqrt{\upconst} \Exs\ltwo{\ProjK g}/2$ implies that
inequality~\eqref{Eqnc1.eq} holds.

From the lower bound~\eqref{Eqngl1.eq} and the fact that, for each
fixed $a > 0$, the function $x \mapsto x^2/(2x + a)$ is increasing on
the interval $[0, \infty)$, we find that
\begin{align*}
  \Gamma(\theta) \geq \frac{\sqrt{\upconst\Exs \ltwo{\ProjK g}}/2}{
    \sqrt{2}\upconst^{1/4} + 8 \sqrt{\Exs \ltwo{\ProjK g}}} -
  \frac{2}{\sqrt{e}}.
\end{align*}
Further, because of general bound \eqref{EqnConstLower} that $\Exs
\ltwo{\ProjK g} \geq 1/\sqrt{2\pi}$ and the fact that the function $x
\mapsto x/(a + x)$ is increasing in $x$, we obtain
\begin{align*}
  \Gamma(\theta) \geq \frac{\sqrt{\upconst}}{2(8\pi \upconst)^{1/4} +
    16} - \frac{2}{\sqrt{e}},
\end{align*}
which ensures inequality \eqref{Eqnc1.eq}.

\paragraph{Case 2}
We now turn to the case when
inequality~\eqref{Eqnga2} is satisfied. We may assume the inequality
$\ltwo{\theta}^2 \geq \sqrt{\upconst} \Exs\ltwo{\ProjK g}/2$ is
violated because otherwise, inequality~\eqref{Eqnc1.eq} follows
immediately.  When this inequality is violated, we have
\begin{align}
\label{Eqnub.2}
\inprod{\theta}{\Exs \ProjK g } \geq \sqrt{\upconst} \Exs \ltwo{\ProjK
  g} ~~ \text{ and } ~~ \ltwo{\theta}^2 < \sqrt{\upconst} \Exs
\ltwo{\ProjK g}/2.
\end{align}

Our strategy is to make use of inequality~\eqref{Eqngl2.eq}, and we
begin by bounding the quantity $\alpha$ appearing therein.  By
combining inequality~\eqref{Eqnub.2} and inequality \eqref{EqnConstLower}---namely, 
$\Exs \ltwo{\ProjK g} \geq 1/\sqrt{2\pi}$, we
find that
\begin{align*}
  \alpha \geq 1 - \exp \left(- \frac{\sqrt{\upconst} \Exs \ltwo{\ProjK
      g}}{4} \right) \geq 1 - \exp \left(-
  \frac{\sqrt{\upconst}}{4\sqrt{2\pi}} \right) \geq 1/2, ~~\text{ whenever
  } \upconst \geq 32 \pi.
\end{align*}
Using expression \eqref{Eqnub.2}, we deduce that
\begin{align*}
  \Gamma(\theta) \; \geq \; \frac{\alpha^2\sqrt{\upconst \Exs
      \ltwo{\ProjK g}}}{\alpha(4 \upconst)^{1/4} + 4 \sqrt{\Exs
      \ltwo{\ProjK g}}} - \sqrt{\frac{2}{e}}
  \; \geq \; \frac{\sqrt{\upconst \Exs \ltwo{\ProjK g}}}{(2^6
    \upconst)^{1/4} + 16 \sqrt{\Exs \ltwo{\ProjK g}}} -
  \sqrt{\frac{2}{e}}.
\end{align*}
where the second inequality uses the previously obtained lower bound
$\alpha > 1/2$, and the fact that the function $x \mapsto x^2/(x + b)$
is increasing in $x$.  

This completes the proof of inequality~\eqref{Eqnc1.eq}.

%%%%%%%%%%%%%%%%%%%%%%%%%%%%%%%%%%%%%%%%%%%%%%%%%%%%%%%%%%%%%%%%%%%%%%%%%%%%%%%%%%%%%%%%%%%%%

\paragraph{Proof of Lemma~\ref{LemLSEup}}
Now it is only left for us to prove Lemma~\ref{LemLSEup}.
We define the random variable $\Diff \defn \ltwo{\ProjK (\theta + g)}
- \ltwo{\ProjK g}$, as well as its positive and negative parts
$\DiffPlus = \max \{0, \Diff \}$ and $\DiffNeg = \max \{0, -\Diff \}$,
so that $\Gamma(\theta) = \Exs \Diff = \Exs \DiffPlus - \Exs
\DiffNeg$. Our strategy is to bound $\Exs \DiffNeg$ from above and
then bound $\Exs \DiffPlus$ from below.  The following auxiliary lemma
is useful for these purposes:\\
\begin{lem}
\label{LemCuteBasic}
For every closed convex cone $\Kcone \subset \real^\usedim$ and vectors $x
\in \Kcone$ and $y \in \real^\usedim$, we have:
\begin{align}
\label{EqncuteBasic1} 
\Big| \ltwo{\ProjKcone(x + y)} - \ltwo{\ProjKcone(y)} \Big| & \leq \ltwo{x},
\qquad \mbox{and} \\
\label{EqncuteBasic2}
\max \left\{ 2\inprod{x}{y} + \ltwo{x}^2, \, 2\inprod{x}{\ProjKcone y}
- \ltwo{x}^2 \right\} & \stackrel{(i)}{\leq} \ltwo{\ProjKcone(x +
  y)}^2 - \ltwo{\ProjKcone(y)}^2 \; \stackrel{(ii)}{\leq}
2\inprod{x}{\ProjKcone \yvec} + \ltwo{x}^2.
\end{align}
\end{lem}
\noindent We return to prove this claim in
Appendix~\ref{AppLemCuteBasic}.

Inequality~\eqref{EqncuteBasic1} implies that $\Diff \geq -
\ltwo{\theta}$ and thus $\Exs \DiffNeg \leq \ltwo{\theta} \mprob
\{\Diff \leq 0\}$.  The lower bound in
inequality~\eqref{EqncuteBasic2} then implies that $\mprob \{\Diff
\leq 0\} \leq \mprob \{\inprod{ \theta}{g} \leq - \ltwo{\theta}^2/2\}
\leq \exp \big(-\frac{\ltwo{\theta}^2}{8} \big)$, whence
\begin{align*}
  \Exs \DiffNeg \leq \ltwo{\theta} \exp
        \left(\frac{-\ltwo{\theta}^2}{8} \right) \leq \sup_{u > 0}
        \left(u e^{-u^2/8} \right) = \frac{2}{\sqrt{e}}.
\end{align*}
Putting together the pieces, we have established the lower bound
\begin{align}
\label{Eqnjg0}
\Exs \Diff \; = \; \Exs \DiffPlus - \Exs \DiffNeg \; \geq \; \Exs
\DiffPlus - \frac{2}{\sqrt{e}}.
\end{align}
The next task is to lower bound the expectation $\Exs \DiffPlus$. By
the triangle inequality, we have
\begin{align*}
\|\ProjK(\theta + g)\|_2 & \leq \|\ProjK(\theta + g) - \ProjK(g)\|_2 +
\|\ProjK(g)\|_2 \\
& \leq \|\theta\|_2 + \|\ProjK(g)\|_2,
\end{align*}
where the second inequality uses non-expansiveness of the projection.
Consequently, we have the lower bound
\begin{align}
\label{Eqnjg}
\Exs \DiffPlus = \Exs \frac{\left(\ltwo{\ProjKcone(\theta + g)}^2 -
  \ltwo{\ProjKcone g}^2 \right)^+}{ \ltwo{\ProjKcone(\theta + g) } +
  \ltwo{\ProjKcone g}} \geq \Exs \frac{\left(\ltwo{\ProjKcone(\theta +
    g) }^2 - \ltwo{\ProjK g}^2\right)^+}{ \ltwo{\theta} + 2
  \ltwo{\ProjK g}} .
\end{align}
Note that inequality~\eqref{EqncuteBasic2}(i) implies two lower bounds
on the difference \mbox{$\ltwo{\ProjKcone(\theta + g)}^2 -
  \ltwo{\ProjKcone g}^2$}. We treat each of these lower bounds in
turn, and show how they lead to inequalities~\eqref{Eqngl1.eq}
and~\eqref{Eqngl2.eq}.

%%%%%%%%%%%%%%%%%%%%%%%%%%%%%%%%%%%%%%%%%%%%%%%%%%%%%%%%%%%%%%%%%%%%%%%%%%%%

\paragraph{Proof of inequality~\eqref{Eqngl1.eq}:}  Inequality~\eqref{Eqnjg} 
and the first lower bound term from
inequality~\eqref{EqncuteBasic2}(i) imply that
\begin{align*}
\Exs \DiffPlus \geq \Exs \frac{\left(2 \inprod{\theta}{g} +
  \ltwo{\theta}^2\right)^+}{ \ltwo{\theta} + 2 \ltwo{\ProjK g}} \geq
\Exs \frac{\ltwo{\theta}^2}{\ltwo{\theta} + 2 \ltwo{\ProjK g}} \Ind
\{\inprod{\theta}{g} \geq 0\}.
\end{align*}
Jensen's inequality (and the fact that $\mprob \{\inprod{\theta}{g}
\geq 0\} = 1/2$) now allow us to deduce
\begin{align*}
 \Exs \DiffPlus \geq \mprob \left\{\inprod{\theta}{g} \geq 0 \right\}
 \ltwo{\theta}^2 \left(\ltwo{\theta} + \frac{2 \Exs \ltwo{\ProjK g}}{P
   \left\{\inprod{\theta}{g} \geq 0 \right\}} \right)^{-1} =
 \frac{\ltwo{\theta}^2}{2\ltwo{\theta} + 8 \Exs \ltwo{\ProjK g}}
\end{align*}
and this gives inequality \eqref{Eqngl1.eq}.

%%%%%%%%%%%%%%%%%%%%%%%%%%%%%%%%%%%%%%%%%%%%%%%%%%%%%%%%%%%%%%%%%%%%%%%%%%%%%%%%%%%%%%

\paragraph{Proof of inequality~\eqref{Eqngl2.eq}:}

Putting inequality~\eqref{Eqnjg}, the second term on the left hand side of inequality~\eqref{EqncuteBasic2}(i), along with the fact
that $\inprod{\theta}{\Exs \ProjK g } \geq \ltwo{\theta}^2$ together guarantees that
\begin{align*}
 \Exs \DiffPlus \geq \Exs \frac{\left(2 \inprod{\theta}{ \ProjK g } -
   \ltwo{\theta}^2\right)^+}{ \ltwo{\theta} + 2 \ltwo{\ProjK g}} \geq
 \Exs \frac{\inprod{\theta}{ \Exs \ProjK g } - \ltwo{\theta}^2}{
   \ltwo{\theta} + 2 \ltwo{\ProjK g}} \; \; \Ind \left\{\inprod{\theta}{
   \ProjK g } > \frac{1}{2} \inprod{\theta}{ \Exs \ProjK g }\right\}.
\end{align*}
Now introducing the event $\Sevent \defn \big \{ \inprod{ \theta}{
  \ProjK g } > \inprod{ \theta}{ \Exs \ProjK g }/2 \big \}$, Jensen's
inequality implies that
\begin{align}
\label{Eqnnm}
\Exs \DiffPlus \geq \Prob(\Sevent) \; \Exs \frac{\inprod{\theta}{\Exs
    \ProjK g } - \ltwo{\theta}^2}{ \ltwo{\theta} + 2 \frac{\Exs
    \ltwo{\ProjK g}}{\Prob(\Sevent)}}.
\end{align}
The concentration
inequality~\eqref{EqnLipschitz2} from Lemma~\ref{LemConcentration}
gives us that
\begin{align}
\label{Eqnpsb}
\Prob(\Sevent) \geq \Prob \left\{\inprod{\theta}{\ProjK g} >
\frac{1}{2} \inprod{\theta}{\Exs \ProjK g }\right\} \geq 1 - \exp
\left(- \frac{\inprod{\theta}{\Exs \ProjK g }^2}{8 \ltwo{\theta}^2}
\right).
\end{align}
Inequality~\eqref{Eqngl2.eq} now follows by combining
inequalities~\eqref{Eqnjg0}, ~\eqref{Eqnnm} and~\eqref{Eqnpsb}.

%%%%%%%%%%%%%%%%%%%%%%%%%%%%%%%%%%%%%%%%%%%%%%%%%%%%%%%%%%%%%%%%%%%%%%%%%%%%%%%%%%%%%%%%

\subsection{Proof of Lemma~\ref{LemCuteBasic}}
\label{AppLemCuteBasic}

It remains to prove Lemma~\ref{LemCuteBasic}.
Inequality~\eqref{EqncuteBasic1} is a standard Lipschitz property of
projection onto a closed convex cone. Turning to
inequality~\eqref{EqncuteBasic2}, recall the polar cone $\starK \defn
\{z \, \mid \, \inprod{z}{\theta} \leq 0, ~ \forall ~ \theta \in
\Kcone \}$, as well as the Moreau
decomposition~\eqref{EqnMoreau}---namely, $z = \ProjKcone(z) +
\ProjKconeStar(z)$. Using this notation, we have
\begin{align*}
\ltwo{\ProjKcone(x + y)}^2 - \ltwo{\ProjK \yvec}^2 & = \ltwo{x + \yvec -
  \ProjKconeStar(x + y)}^2 - \ltwo{y - \ProjKconeStar y}^2\\
& = \ltwo{x}^2 + 2 \inprod{x}{y - \ProjKconeStar(x +y)} + \ltwo{y -
  \ProjKconeStar(x + y)}^2 - \ltwo{y - \ProjKconeStar y}^2.
\end{align*}
Since $\ProjKconeStar(y)$ is the closest point in $\starK$ to $y$, we have
$\ltwo{y - \ProjKconeStar(x + y)} \geq \ltwo{y - \ProjKconeStar(y)}$, and hence
\begin{align}
\label{Eqnbll}
\ltwo{\ProjKcone(x + y)}^2 - \ltwo{\ProjK \yvec}^2 & \geq \ltwo{x}^2 + 2
\inprod{x}{ \yvec - \ProjKconeStar(x +y)}.
\end{align}
Since $x \in K$ and $\ProjKconeStar(x + y) \in \starK$, we have
$\inprod{x}{\ProjKconeStar(x + y)} \leq 0$, and hence,
inequality~\eqref{Eqnbll} leads to the bound (i) in
equation~\eqref{EqncuteBasic2}. In order to establish inequality (ii)
in equation~\eqref{EqncuteBasic2}, we begin by rewriting
expression~\eqref{Eqnbll} as
\begin{align*}
 \ltwo{\ProjKcone(x + y)}^2 - \ltwo{\ProjK \yvec}^2 \geq \ltwo{x}^2 + 2
 \inprod{x}{y - \ProjKconeStar \yvec } + 2 \inprod{x}{\ProjKconeStar \yvec -
   \ProjKconeStar(x + y)}.
\end{align*}
Applying the Cauchy-Schwarz inequality to the final term above and
using the $1$-Lipschitz property of $z \mapsto \ProjKconeStar z$, we
obtain:
\begin{align*}
\inprod{x}{\ProjKconeStar \yvec - \ProjKconeStar(x + y)} & \geq - \ltwo{x}
\ltwo{\ProjKconeStar \yvec - \ProjKconeStar (x + y)} \geq - \ltwo{x}^2,
\end{align*} 
which establishes the upper bound of inequality~\eqref{EqncuteBasic2}.

Finally, in order to prove the lower bound in
inequality~\eqref{EqncuteBasic2}, we write
\begin{align*}
& \ltwo{\ProjKcone(x + y)}^2 - \ltwo{\ProjK \yvec}^2 \\
= & \ltwo{x + \yvec
  - \ProjKconeStar(x + y)}^2 - \ltwo{x + \yvec - \ProjKconeStar \yvec
  - x}^2 \\
= &\ltwo{x + \yvec - \ProjKconeStar(x + y)}^2 - \ltwo{x + \yvec -
  \ProjKconeStar y}^2 + 2 \inprod{x}{x + \yvec - \ProjKconeStar y} -
\ltwo{x}^2.
\end{align*}
Since the vector $\ProjKconeStar(x+y)$ corresponds to the projection
of $x+y$ onto $\starK$, we have $\ltwo{x+y -\ProjKconeStar(x + y)}
\leq \ltwo{x + \yvec - \ProjKconeStar y}$ and thus
\begin{align*}
\ltwo{\ProjKcone(x + y)}^2 - \ltwo{\ProjK \yvec}^2 & \leq \ltwo{x}^2 +
2 \inprod{x}{\ProjK \yvec},
\end{align*}
which completes the proof of inequality~\eqref{EqncuteBasic2}.

%%%%%%%%%%%%%%%%%%%%%%%%%%%%%%%%%%%%%%%%%%%%%%%%%%%%%%%%%%%%%%%%%%%%%%%%%%%%%%%%%%%%%%

\section{Auxiliary proofs for Theorem 1 (b)}
\label{AppThmGLRTB}

In this appendix, we collect the proofs of lemmas involved in the
proof of Theorem~\ref{ThmGLRT}(b), corresponding to the lower bound
on the GLRT performance.

%%%%%%%%%%%%%%%%%%%%%%%%%%%%%%%%%%%%%%%%%%%%%%%%%%%%%%%%%%%%%%%%%%%%%%%%%%%%%%%%%%%%%%%%

\subsection{Proof for scenario $\Exs \ltwo{\ProjK g} < 128$}
\label{AppScene2}

When $\Exs \ltwo{\ProjK g} < 128$, we begin by
setting $\lowconst = \frac{1}{256}$.  The assumed bound $\epsilon^2
\leq \frac{1}{256} \NEWEPSCRITSQ$ then implies that
\begin{align*}
\epsilon^2 \leq \frac{1}{256} \NEWEPSCRITSQ\leq \frac{\Exs
  \ltwo{\ProjK g}}{256} < \frac{1}{2}.
\end{align*}
For every $\epsilon^2 \leq \frac{1}{2}$, we claim that
$\UNIERR(\testglrt; \{0\}, \Kcone, \epsilon) \geq 1/2$.  Note
that the uniform error $\UNIERR(\testglrt; \{0\}, \Kcone, \epsilon)$ is at least as large as the error in
the simple binary test
\begin{subequations}
\begin{align} 
\label{EqnSimpleTesting}
\Hyp_0: \yvec \sim \NORMAL(0, \EYE{\usedim}) ~~\text{ versus
}~~ \Hyp_1: \yvec \sim \NORMAL(\theta, \EYE{\usedim}),
\end{align}
where $\theta \in \Kcone$ is any vector such that $\ltwo{\theta} =
\epsilon$.  We claim that the error for the simple binary
test~\eqref{EqnSimpleTesting} is lower bounded as
\begin{align}
\label{EqnBinaryLower}
\inf_{\psi} \UNIERR(\psi; \{0\}, \{\theta\}, \epsilon) & \geq 1/2
\qquad \mbox{whenever $\epsilon^2 \leq 1/2$.}
\end{align}
\end{subequations}
The proof of this claim is straightforward: introducing the shorthand
$\Prob_\theta = \NORMAL(\theta, \EYE{\usedim})$ and $\Prob_0
= \NORMAL(0, \EYE{\usedim})$, we have
\begin{align*}
\inf_{\psi} \UNIERR(\psi; \{0\}, \{\theta\}, \epsilon) & = 1 -
\tvnorm{\mprob_\theta - \mprob_0}.
\end{align*}
Using the relation between $\chi^2$ distance and TV-distance in
expression~\eqref{EqnChi-TV} and the fact that
\mbox{$\chi^2(\Prob_\theta , \Prob_0) = \exp(\epsilon^2)-1$,} we find
that the testing error satisfies
\begin{align*}
  \inf_{\psi} \UNIERR(\psi; \{0\}, \{\theta\}, \epsilon) \geq 1 -
  \frac{1}{2} \sqrt{\exp(\epsilon^2)-1} \geq 1/2, \qquad{
    \text{whenever } \epsilon^2 \leq 1/2}.
\end{align*}
(See Section~\ref{SecTv-Chi} for more details on the relation between
the TV and \mbox{$\chi^2$-distances.)} This completes the proof under
the condition $\Exs \ltwo{\ProjK g} < 128$.

%%%%%%%%%%%%%%%%%%%%%%%%%%%%%%%%%%%%%%%%%%%%%%%%%%%%%%%%%%%%%%%%%%%%%%%%%%%%%%%%%%%%%%%%

\subsection{Proof of Lemma~\ref{LemLatte}}
\label{AppLemLatte}

Let us first state Lemma~\ref{LemLatte} and give a proof of it. 
\begin{lem}
\label{LemLatte}
For any constant $a \geq 1$ and for every closed convex cone $\Kcone
\neq \{0 \}$, we have
\begin{subequations}
\begin{align} 
\label{EqnLatte}
0 \leq \Gamma(\theta) \leq \frac{2a \ltwo{\theta}^2 + 4
  \inprod{\theta}{\Exs \ProjK g}}{\Exs \ltwo{\ProjK g}} + b\ltwo{\theta} \qquad
\mbox{for all $\theta \in \Kcone$,}
\end{align}
where
\begin{align}
b \defn 3 \exp( - \frac{(\Exs\ltwo{\ProjK g})^2}{8}) + 24 \exp( -
\frac{a^2 \ltwo{\theta}^2}{16}).
\end{align}
\end{subequations}
\end{lem}

In order to prove that $\Gamma(\theta) \geq 0$, we first introduce the
convenient shorthand notation $v_1 \defn \ProjKconeStar (\theta+g)$ and
$v_2 \defn \ProjKconeStar g$. 
Recall that $\KconeStar$ denotes the polar cone of $\Kcone$ defined in expression~\eqref{EqnDefnPolar}.
 With this notation, the the Moreau
decomposition~\eqref{EqnMoreau} then implies that
\begin{align*}
\ltwo{\ProjK (\theta+g)}^2 - \ltwo{\ProjK g}^2 &= \ltwo{\theta + g -
  v_1}^2 - \ltwo{g -v_2}^2 \\
  &= \ltwo{\theta}^2 + 2\inprod{\theta}{g-v_1} + \ltwo{g-v_1}^2 - \ltwo{g-v_2}^2.
\end{align*}
The right hand side above is greater than $\ltwo{\theta}^2 +
2\inprod{\theta}{g-v_1}$ because $\ltwo{g-v_1}^2 \geq
\min_{v\in\Kcone^*} \ltwo{g-v}^2 = \ltwo{g-v_2}^2$. 
From the fact that $\Exs \inprod{\theta}{g} = 0$ and $\inprod{\theta}{v} \leq 0$ for
all $v\in \Kcone^*$, we have $\Gamma(\theta) \geq 0$.

Now let us prove the upper bound for expected difference
$\Gamma(\theta)$.  Using the convenient shorthand notation
\mbox{$\Diff \defn \ltwo{\ProjK (\theta + g)} - \ltwo{\ProjK g }$,} we
define the event
\begin{align*}
\Bevent \defn \{ \ltwo{\ProjK g} \geq \frac{1}{2} \Exs
\ltwo{\ProjKcone g}\}, \qquad{\text{ where } g\sim
  \NORMAL(0,\EYE{\usedim})}.
\end{align*}
Our proof is then based on the decomposition $\Gamma(\theta) \; = \;
\Exs \Diff \; = \; \Exs \Diff \Ind(\Bevent^c) + \Exs \Diff
\Ind(\Bevent)$.  In particular, we upper bound each of these two terms
separately.

\paragraph{Bounding $\Exs[\Diff \Ind(\Bevent^c)]$:}
The analysis of this term is straightforward:
inequality~\eqref{EqncuteBasic1} from Lemma~\ref{LemCuteBasic}
guarantees that $\Diff \leq \ltwo{\theta}$, whence
\begin{align}
\label{EqnDiffTermOne}
\Exs \Diff \Ind(\Bevent^c) \leq \ltwo{\theta}\Prob(\Bevent^c).
\end{align}

\paragraph{Bounding $\Exs[ \Diff \Ind(\Bevent)]$:}
Turning to the second term, we have
\begin{align*}
\Exs \Diff \Ind(\Bevent) & \leq \Exs \DiffPlus \Ind(\Bevent) \\
& = \Exs \frac{\left( \ltwo{\ProjK (\theta + g)}^2 - \ltwo{\ProjKcone
    g}^2\right)^+}{ \ltwo{\ProjK (\theta + g)} + \ltwo{\ProjK g}}
\Ind(\Bevent) \leq \Exs \frac{\left( \ltwo{\ProjK (\theta + g)}^2 -
  \ltwo{\ProjK g}^2 \right)^+}{\ltwo{\ProjK g}} \Ind(\Bevent).
\end{align*}
On event $\Bevent$, we can lower bound quantity $\ltwo{\ProjK g}$ with
$\Exs \ltwo{\ProjK g}/2$ thus
\begin{align} 
\label{EqnWine}
\Exs \frac{\left( \ltwo{\ProjK (\theta + g)}^2 - \ltwo{\ProjK g}^2
  \right)^+}{ \ltwo{\ProjK g}} \Ind(\Bevent) \leq \underbrace{\Exs
  \frac{\left( \ltwo{\ProjK (\theta + g)}^2 - \ltwo{\ProjK g}^2
    \right)^+ \Ind(\Bevent)}{\Exs \ltwo{\ProjK g}/2} }_{\defn \TERMONE
}.
\end{align}

Next we use inequality~\eqref{EqncuteBasic2} to bound the numerator of
the quantity $\TERMONE$, namely
\begin{align*} 
\Exs \left(\ltwo{\ProjK (\theta + g) }^2 - \ltwo{\ProjK g}^2\right)^+
\Ind(\Bevent) 
&\leq \Exs \left(2\inprod{\theta}{\ProjK g} +
\ltwo{\theta}^2\right)^+ \Ind(\Bevent)\\
& \leq \Exs \left(2\inprod{\theta}{\ProjK g} + a
        \ltwo{\theta}^2\right)^+ \Ind(\Bevent),
\end{align*}
for every constant $a \geq 1$. To further simplify notation, introduce
event $\Mevent \defn \{ \theta^T \ProjK g \geq -a \ltwo{\theta}^2/2\}$
and by definition, we obtain
\begin{align}
\label{EqnSalmon}
\Exs \left(2\inprod{\theta}{\ProjK g} + a \ltwo{\theta}^2\right)^+
\Ind(\Bevent)
&= \Exs \left(2\inprod{\theta}{\ProjK g} + a \ltwo{\theta}^2\right)
\Ind(\Bevent \cap \Mevent) \notag \\
& \leq a \ltwo{\theta}^2 + 2\Exs [\inprod{\theta}{\ProjK g}
  \Ind(\Bevent \cap \Mevent)].
\end{align}
The right hand side of inequality~\eqref{EqnSalmon} consists of two
terms.  The first term $a \ltwo{\theta}^2$ is a constant, so that we
only need to further bound the second term $2\Exs
\inprod{\theta}{\ProjK g} \Ind(\Bevent \cap \Mevent)$. We claim that
\begin{align} 
\label{EqnLBTerm2}
\Exs [\inprod{\theta}{\ProjK g} \Ind(\Bevent \cap \Mevent)] \leq \Exs
\inprod{\theta}{\ProjK g} + \ltwo{\theta}\Exs \ltwo{\ProjK g}
(6\sqrt{\Prob(\Mevent^c)} + \Prob(\Bevent^c)/2).
\end{align}
Taking inequality~\eqref{EqnLBTerm2} as given for the moment,
combining inequalities~\eqref{EqnWine}, \eqref{EqnSalmon}
and~\eqref{EqnLBTerm2} yields
\begin{align} 
\label{EqnDinner}
\Exs \DiffPlus \Ind(\Bevent) \leq \TERMONE \leq \frac{2 a
  \ltwo{\theta}^2 + 4 \Exs \inprod{\theta}{\ProjK g}}{\Exs
  \ltwo{\ProjKcone g}} + \ltwo{\theta} (24\sqrt{\Prob(\Mevent^c)} + 2
\Prob(\Bevent^c)).
\end{align}

%%%%%%%%%%%%%%%%%%%%%%%%%%%%%%%%%%%%%%%%%%%%%%%%%%%%%%%%%%%%%%%%%%%%%%%%%%%%%%%%%%%%%%

As a summary of the above two parts---namely inequalities~\eqref{EqnDiffTermOne} and \eqref{EqnDinner}, if we assume
inequality~\eqref{EqnLBTerm2}, we have
\begin{align} 
\label{EqnIntGamma}
\Gamma(\theta) \leq \frac{2 a \ltwo{\theta}^2 + 4\Exs
  \inprod{\theta}{\ProjK g}}{\Exs \ltwo{\ProjK g}} +
\ltwo{\theta}(24\sqrt{\Prob(\Mevent^c)} + 3\Prob(\Bevent^c)).
\end{align}
Based on expression~\eqref{EqnIntGamma}, the last step in proving
Lemma~\ref{LemLatte} is to control the probabilities
$\Prob(\Mevent^c)$ and $\Prob(\Bevent^c)$ respectively.  Using the fact that
$\inprod{\theta}{\ProjK g} = \inprod{\theta}{(g - \ProjKconeStar
    g)} \geq \inprod{\theta}{g}$ and the concentration of $\inprod{\theta}{g}$, we have 
\begin{align*}
\Prob(\Mevent^c) &= \Prob( \inprod{\theta}{\ProjK g} < -\frac{a}{2}
\ltwo{\theta}^2) \leq \Prob(\inprod{\theta}{g} < -\frac{a}{2}
\ltwo{\theta}^2) \leq \exp( - \frac{a^2\ltwo{\theta}^2}{8}),\\
\text{ and }~~~ \Prob(\Bevent^c) &= \Prob(\ltwo{\ProjK g} < 
\frac{1}{2} \Exs \ltwo{\ProjK g}) \leq \exp( - \frac{(\Exs
  \ltwo{\ProjKcone g})^2}{8}).
\end{align*}
where the second inequality follows directly from concentration result in Lemma~\ref{LemConcentration} \eqref{EqnLipschitz1}.
Substituting the above two inequalities into
expression~\eqref{EqnIntGamma} yields Lemma~\ref{LemLatte}.

So it is only left for us to show inequality~\eqref{EqnLBTerm2}.  To
see this, first notice that
\begin{align} 
\label{EqnFix}
\Exs [\inprod{\theta}{\ProjK g}\Ind(\Bevent \cap \Mevent)] = \Exs
\inprod{\theta}{\ProjK g} - \Exs \inprod{\theta}{\ProjK g}
\Ind(\Mevent^c \cup \Bevent^c).
\end{align}
The Cauchy-Schwarz inequality and triangle inequality allow us to
deduce
\begin{align*} 
- \Exs \inprod{\theta}{\ProjK g} \Ind(\Mevent^c \cup \Bevent^c) 
&= \inprod{\theta}{- \Exs [\ProjK g \Ind(\Mevent^c \cup \Bevent^c)]}  \\
&\leq \ltwo{\theta} \ltwo{\Exs [\ProjK g \Ind(\Mevent^c \cup \Bevent^c)]} \\
& \leq \ltwo{\theta} \Big\{ \ltwo{\Exs \ProjK g \Ind(\Mevent^c)} +
\ltwo{\Exs \ProjK g \Ind(\Bevent^c)} \Big \}.
\end{align*}
Jensen's inequality further guarantees that 
\begin{align}
\label{EqnFixCauchy}
  - \Exs \inprod{\theta}{\ProjK g} \Ind(\Mevent^c \cup \Bevent^c) 
  & \leq \ltwo{\theta} \Big\{ \underbrace{ \Exs [\ltwo{\ProjK g}
    \Ind(\Mevent^c)}_{\defn \TERMTWO}] + \underbrace{ \Exs
  [\ltwo{\ProjK g} \Ind(\Bevent^c)}_{\defn \TERMTHREE}] \Big \},
\end{align}

By definition, on event $\Bevent^c$, we have $\ltwo{\ProjK g} \leq \Exs
\ltwo{\ProjK g}/2$, and consequently
\begin{align}
\label{EqnPartiii}
\TERMTHREE \leq \frac{\Exs \ltwo{\ProjK g} \Prob(\Bevent^c)}{2}.
\end{align}
Turning to the quantity $\TERMTWO$, applying Cauchy-Schwartz inequality
yields
\begin{align*} 
\TERMTWO & \leq \sqrt{\Exs\ltwo{\ProjK g}^2 } \sqrt{\Exs \Ind(\Mevent^c)} =
\sqrt{(\Exs\ltwo{\ProjK g})^2 + \var(\ltwo{\ProjK g})}
\sqrt{\Prob(\Mevent^c)}.
\end{align*}
The variance term can be bounded as in inequality~\eqref{EqnVarianceBound}
which says that $\var(\ltwo{\ProjK g}) \leq 4.$

From inequality~\eqref{EqnConstLower}, for every non-trivial cone ($\Kcone \neq
\{0\}$), we are guaranteed that $\Exs\ltwo{\ProjK g} \geq
1/\sqrt{2\pi}$, and hence $\var(\ltwo{ \ProjK g}) \leq 8\pi(\Exs
\ltwo{\ProjK g})^2$.  Consequently, the quantity $\TERMTWO$ can be
further bounded as
\begin{align}
\label{EqnPartii}
\TERMTWO & \leq \sqrt{1+8\pi} \Exs \ltwo{\ProjK g}
\sqrt{\Prob(\Mevent^c)} \leq 6 \Exs \ltwo{\ProjK g}
\sqrt{\Prob(\Mevent^c)}.
\end{align}

Putting together inequalities~\eqref{EqnPartiii}, \eqref{EqnPartii}
and~\eqref{EqnFixCauchy} yields
\begin{align*}
- \Exs [\inprod{\theta}{\ProjK g} \Ind(\Mevent^c \cup (\Mevent \cap
  \Bevent^c))] \leq \ltwo{\theta} \Exs \ltwo{\ProjK g} (6
\sqrt{\Prob(\Mevent^c)} + \Prob(\Bevent^c)/2),
\end{align*}
which validates claim~\eqref{EqnLBTerm2} when combined with inequality
\eqref{EqnFix}.  We finish the proof of Lemma~\ref{LemLatte}.

%%%%%%%%%%%%%%%%%%%%%%%%%%%%%%%%%%%%%%%%%%%%%%%%%%%%%%%%%%%%%%%%%%%%%%%%%%%%%%%%%%%%%%%%

\subsection{Calculate the testing error}
\label{AppCalculate}

\noindent The following lemma allows us to relate $\ltwo{\ProjK g}$ to
its expectation:
\begin{lem} 
\label{LemCLT}
Given every closed convex cone $\Kcone$ such that $\Exs \ltwo{\ProjK
  g} \geq 128$, we have
\begin{align} 
\label{EqnHammerCor}
\Prob(\ltwo{\ProjK g} > \Exs \ltwo{\ProjK g}) \, > \, 7/16.
\end{align} 
\end{lem}
\noindent See Appendix~\ref{AppLemCLT}  for the proof of this claim.

For future reference, we note that it is relatively straightforward to
show that the random variable $\ltwo{\ProjK g}$ is distributed as a
mixture of $\chi$-distributions, and indeed, the Lemma~\ref{LemCLT} can be proved
via this route.  Raubertas et al.~\cite{raubertas1986hypothesis}
proved that the squared quantity $\ltwo{\ProjK g}^2$ is a mixture of
$\chi^2$ distributions, and a very similar argument yields the
analogous statement for $\ltwo{\ProjK g}$.

We are now ready to calculate the testing error for the GLRT given in
equation~\eqref{EqnGLRT}. Our goal is to lower bound the error
$\UNIERR(\psiglrt; \{0\}, \Kcone, \epsilon)$ uniformly over the chosen
threshold $\beta \in [0,\infty)$.  We divide the choice of $\beta$ into
three cases, depending on the relationship between $\beta$ and $\Exs \ltwo{\ProjK g}$, $\Exs \ltwo{\ProjK(\theta + g)}$.  Notice
this particular $\theta$ is chosen to be the one that satisfies inequality~\eqref{EqnMusgraves}.

\paragraph{Case 1}  First, consider a threshold
$\beta \in [0, \; \Exs \ltwo{\ProjK g}]$.  It then follows
immediately from inequality~\eqref{EqnHammerCor} that the type I error by its own satisfies 
\begin{align*}
  \text{type I error} = \mprob_{0}(\ltwo{\ProjK \yvec} \geq \beta) \geq 
  \Prob(\ltwo{\ProjK g} \geq \Exs \ltwo{\ProjK g}) \geq \frac{7}{16}.
\end{align*}

\paragraph{Case 2}  Otherwise, consider a threshold $\beta \in 
\big( \Exs \ltwo{\ProjK g}, \; \Exs \ltwo{\ProjK (\theta + g)} \big]$.
In this case, we again use inequality~\eqref{EqnHammerCor} to bound
the type I error, namely
\begin{align*}
  \text{type I error} & = \mprob_{0}( \ltwo{\ProjK \yvec} \geq \beta)
  \\
& = \Prob \Big[ \ltwo{\ProjK g } \geq \Exs \ltwo{\ProjK g} \Big] -
  \Prob \Big[ \ltwo{\ProjK g} \in [\Exs \ltwo{\ProjK g}, \beta) \Big]
    \\
& \geq \frac{7}{16} - \max_x \{ f_{\ltwo{\ProjK g}}(x)(\beta - \Exs \ltwo{\ProjK g})
\},
\end{align*}
where we use $f_{\ltwo{\ProjK g}}$ to denote the density function of
the random variable $\ltwo{\ProjK g}$ As discussed earlier, the random
variable $\ltwo{\ProjK g}$ is distributed as a mixture of
$\chi$-distributions; in particular, see Lemma~\ref{LemCLT} above and
the surrounding discussion for details.  As can be verified by direct
numerical calculation, any $\chi_k$ variable has a density that
bounded from above by $4/5$.  Using this fact, we have
\begin{align*}
\text{type I error} &\geq \frac{7}{16} - \frac{4}{5}(\beta - \Exs
\ltwo{\ProjK g}) \, \stackrel{(i)}{\geq}\, \frac{7}{16} -
\frac{4}{5}\Gamma(\theta) \, \stackrel{(ii)}{>} 3/8,
\end{align*}
where step (i) follows by the assumption that $\beta$ belongs to the
interval \mbox{$\big( \Exs \ltwo{\ProjK g}, \; \Exs \ltwo{\ProjK
    (\theta + g)} \big]$,} and step (ii) follows since $\Gamma(\theta)
  \leq 1/16$.

\paragraph{Case 3} Otherwise,  given a threshold
$\beta \in \big( \Exs \ltwo{\ProjK (g+\theta)}, \infty \big)$, we
define the scalar $x \defn \beta - \Exs \ltwo{\ProjK (g+\theta)}$.
From the concentration inequality given in
Lemma~\ref{LemConcentration}, we can deduce that
\begin{align*}
\text{type II error} &\geq \mprob_{\theta}(\ltwo{\ProjK \yvec} \leq
\beta) \\
& = 1 - \Prob \Big( \ltwo{\ProjK (\theta+ g)} - \Exs
\ltwo{\ProjK(\theta+ g)} > \beta - \Exs \ltwo{\ProjK(\theta + g)}
\Big)\\
& \geq 1 - \exp(-x^2/2).
\end{align*}
At the same time, 
\begin{align*}
\text{type I error} = \mprob_{0}( \ltwo{\ProjK \yvec} \geq \beta) \; & = \;
\mprob (\ltwo{\ProjK g} \geq \Exs \ltwo{\ProjK g}) - \mprob
(\ltwo{\ProjK g} \in [\Exs \ltwo{\ProjK g}, \beta)) \\
  & \geq \frac{7}{16} - \frac{4}{5}(\beta - \Exs\ltwo{\ProjK g}),
\end{align*}
where we again use inequality~\eqref{EqnHammerCor} and the boundedness
of the density of $\ltwo{\ProjK g}$.  Recalling that we have defined
\mbox{$x \defn \beta - \Exs \ltwo{\ProjK (g+\theta)}$} as well as
\mbox{$\Gamma(\theta) = \Exs \big( \ltwo{\ProjK(\theta +g)} - \ltwo{
    \ProjK g} \big)$,} we have
\begin{align*}
  \beta - \Exs\ltwo{\ProjK g} = x + \Gamma(\theta) \; \leq \; x + \frac{1}{16},
\end{align*}
where the last step uses the fact that $\Gamma(\theta) \leq 1/16$.
Consequently, the type I error is lower bounded as
\begin{align*}
\text{type I error} & \geq \frac{7}{16} - \frac{4}{5}(x + 1/16) \; =
\; \frac{31}{80} - \frac{4}{5} x.
\end{align*}
Combining the two types of error, we find that the testing error is
lower bounded as
\begin{align*}
\inf_{x > 0} \Big\{(\frac{31}{80} - \frac{4}{5} x)^+ + 1 - \exp( -
x^2/2) \Big\} = 1 - \exp(-\frac{31^2}{2\times64^2}) \geq 0.11.
\end{align*}

Putting pieces together, the GLRT cannot succeed with error smaller than $0.11$ no matter how the cut-off $\beta$ is chosen.

%%%%%%%%%%%%%%%%%%%%%%%%%%%%%%%%%%%%%%%%%%%%%%%%%%%%%%%%%%%%%%%%%%%%%%%%%%%%%

\subsection{Proof of inequality~\eqref{EqnMusgraves}}
\label{AppLemMusgraves}

Now let us turn to the proof of inequality~\eqref{EqnMusgraves}.
First notice that if the radius satisfies $\epsilon^2 \leq \lowconst
\NEWEPSCRITSQ$, then there exists some $\theta \in \Hyp_1$ with
$\ltwo{\theta} = \epsilon$ that satisfies
\begin{align}
\label{EqnKacey}
\ltwo{\theta}^2 \leq \lowconst \Exs \ltwo{\ProjK g} \text{ and }
\inprod{\theta}{\Exs \ProjK g} \leq \sqrt{\lowconst} \Exs \ltwo{\ProjK
  g}.
\end{align}
Setting $a = 4 /\sqrt{\lowconst}\geq 1$ in inequality~\eqref{EqnLatte}
yields
\begin{align*}
\Gamma(\theta) \leq \frac{8 \ltwo{\theta}^2/\sqrt{\lowconst} + 4
  \inprod{\theta}{\Exs \ProjK g}}{\Exs \ltwo{\ProjK g}} +
b\ltwo{\theta}
\end{align*}
where $b \defn 3\exp( - \frac{(\Exs\ltwo{\ProjK g})^2}{8}) + 24 \exp(
- \frac{\ltwo{\theta}^2}{\lowconst})$.  Now we only need to bound the
two terms in the upper bound separately.  First, note that
inequality~\eqref{EqnKacey} yields
\begin{align}
\frac{8 \ltwo{\theta}^2/\sqrt{\lowconst} + 4 \inprod{\theta}{\Exs
    \ProjK g}}{\Exs \ltwo{\ProjK g}} \leq 12 \sqrt{\lowconst}.
\end{align}
On the other hand, again by applying inequality~\eqref{EqnKacey}, it is
straightforward to verify the following two facts that
\begin{align*} 
\ltwo{\theta} \exp(-\frac{(\Exs \ltwo{\ProjK g})^2}{8}) &\leq
\sqrt{\lowconst \Exs \ltwo{\ProjK g}}\exp(-\frac{(\Exs \ltwo{\ProjK
    g})^2}{8}) \\
    &\leq \sqrt{\lowconst}\max_{x \in (0,\infty)} \sqrt{x}
\exp(-\frac{x^2}{8}) =\sqrt{\lowconst}
\left(\frac{2}{e}\right)^{1/4},\\
\text{ and }~~~ \ltwo{\theta} \exp(- \frac{\ltwo{\theta}^2}{\lowconst}) &\leq \sup_{x
  \in (0,\infty)} x \exp(-\frac{x^2}{\lowconst}) =
\sqrt{\frac{\lowconst}{2e}}.
\end{align*}
Combining the above two inequalities ensures an upper bound for
product $b \ltwo{\theta}$ and directly leads to upper bound of
quantity $\Gamma(\theta)$, namely
\begin{align*}
\Gamma(\theta) \leq 12\sqrt{\lowconst} + 3\sqrt{\lowconst}
\left(\frac{2}{e}\right)^{1/4} + 24\sqrt{\frac{\lowconst}{2e}},
\end{align*}
With the choice of $\lowconst$, we established inequality~\eqref{EqnMusgraves}.

%%%%%%%%%%%%%%%%%%%%%%%%%%%%%%%%%%%%%%%%%%%%%%%%%%%%%%%%%%%%%%%%%%%%%%%%%%%%%

\subsection{Proof of Lemma~\ref{LemCLT}} 
\label{AppLemCLT}

In order to prove this result, we first define random variable $F
\defn \ltwo{\ProjK g}^2 - m$, where $m \defn \Exs \ltwo{\ProjK g}^2$ and
$\tilde{\sigma}^2 \defn \var(F)$. We make use of the Theorem 2.1 in
Goldstein et al.~\cite{goldstein2014gaussian} which shows that the
distribution of $F$ and Gaussian distribution $Z \sim \NORMAL(0,
\tilde{\sigma}^2)$ are very close, more specifically, the Theorem says
\begin{align}  
\label{EqnHammer}
  \tvnorm{F-Z} \, \leq \, \frac{16}{\tilde{\sigma}^2} \sqrt{m}
        \, \leq \, \frac{8}{\Exs \ltwo{\ProjK g}}.
\end{align}
In the last inequality, we use the facts that $\tilde{\sigma}^2 \geq
2m$ and $\sqrt{\Exs \ltwo{\ProjK g}^2} \geq \Exs \ltwo{\ProjK g}$.

It is known that
the quantity $\ltwo{\ProjK g}^2$ is distributed as a mixture of
$\chi^2$ distributions(see e.g., \cite{raubertas1986hypothesis,goldstein2014gaussian})---in particular, we can write
\begin{align*}
\ltwo{\ProjK g}^2 \stackrel{\text{ law} }{=} \sum_{i=1}^{V_\Kcone} X_i
= W_\Kcone + V_\Kcone, \qquad{ W_\Kcone = \sum_{i=1}^{V_\Kcone}(X_i -
  1)},
\end{align*}
where each $\{X_i\}_{i\geq 1}$ is an i.i.d. sequence $\chi_1^2$
variables, independent of $V_\Kcone$. Applying the decomposition of
variance yields
\begin{align*}
\tilde{\sigma}^2 = \var(V_\Kcone) + 2 \Exs \ltwo{\ProjK g}^2 \geq 2m.
\end{align*}

We can write the probability $\Prob(\ltwo{\ProjK g}  > \Exs \ltwo{\ProjK g})$ as
\begin{align*} 
\Prob( \ltwo{\ProjK g} > \Exs \ltwo{\ProjK g}) = \Prob( \ltwo{\ProjK
  g}^2 - \Exs \ltwo{\ProjK g}^2 >(\Exs \ltwo{\ProjK g})^2 - \Exs
\ltwo{\ProjKcone g}^2) \geq \Prob(F > 0).
\end{align*}
So if $\Exs \ltwo{\ProjK g} \geq 128$, then
inequality~\eqref{EqnHammer} ensures that $d_{TV}(F,N) \leq 1/16$, and
hence
\begin{align*}
\mprob(F > 0) & \geq \Prob(Z > 0) - \tvnorm{F-Z} \geq \frac{7}{16}.
\end{align*}
We finish the proof of Lemma~\ref{LemCLT}.

%%%%%%%%%%%%%%%%%%%%%%%%%%%%%%%%%%%%%%%%%%%%%%%%%%%%%%%%%%%%%%%%%%%%%%%%%%%%%%%%%%%%%%%%%%%%

\section{Auxiliary proofs for Theorem 2}
\label{AppThmLBGen}

In this appendix, we collect the proofs of various lemmas used in the proof
of Theorem~\ref{ThmLBGen}.

%%%%%%%%%%%%%%%%%%%%%%%%%%%%%%%%%%%%%%%%%%%%%%%%%%%%%%%%%%%%%%%%%%%%%%%%%%%%%%%

\subsection{Proof of Lemma~\ref{LemChisquareBound}}
\label{AppLemChisquareBound}

For every probability measure $\qprob$ supported on $\Kcone \cap \NewBall^c(1)$, let
vector $\theta$ be distributed accordingly to measure $\epsilon
\qprob$ then it is supported on $\Kcone\cap \NewBall^c(\epsilon)$.
Consider a mixture of distributions,
\begin{align}
  \Prob_1(y) = \Exs_{\theta}~ (2\pi)^{-\usedim/2} \exp(-\frac{
          \ltwo{y-\theta}^2}{2}).
\end{align}
Let us first control the $\chi^2$ distance between distributions
$\Prob_1$ and $\Prob_0 \defn \NORMAL(0,\EYE{\usedim})$. Direct calculations yield
\begin{align*}
\chi^2(\Prob_1,\Prob_0) + 1
= \Exs_{\Prob_0} \left(\frac{\Prob_1}{\Prob_0}\right)^2
&= \Exs_{\Prob_0} \left(\Exs_{\theta} \exp\{ -\frac{
  \ltwo{y-\theta}^2}{2} + \frac{ \ltwo{y}^2}{2} \}
\right)^2 \\
&= \Exs_{\Prob_0} \left(\Exs_{\theta}\exp\{ \inprod{y}{\theta} - \frac{
  \ltwo{\theta}^2}{2}\} \right)^2.
\end{align*}
Suppose random vector $\theta'$ is an independent copy of random
vector $\theta$, then
\begin{align}
\label{EqnChiSquare}
\chi^2(\Prob_1,\Prob_0) + 1 & = \Exs_{\Prob_0} \Exs_{\theta, \theta'}
\exp\{\inprod{y}{\theta+\theta'}-\frac{
  \ltwo{\theta}^2+ \ltwo{\theta'}^2}{2}\} \notag \\
& = \Exs_{\theta, \theta'}\exp\{
\frac{\ltwo{\theta+\theta'}^2}{2} - \frac{\ltwo{\theta}^2+
  \ltwo{\theta'}^2}{2}\} \notag \\
& = \Exs_{\theta, \theta'}\exp(\inprod{\theta}{\theta'}) \notag \\
& = \Exs \exp(\epsilon^2\inprod{\eta}{\eta'}),
\end{align}
where the second step uses the fact the moment generating function of
multivariate normal distribution.  As we know, the testing error is
always bounded below by $1 - \tvnorm{\mprob_1, \mprob_0}$, so by the
relation between the $\chi^2$ distance and TV distance, we have:
\begin{align*}
\text{testing error} \, \geq \, 1 - \frac{1}{2}\sqrt{\Exs \exp \left(
\epsilon^2 \inprod{\eta}{\eta'} \right)-1},
\end{align*}
which completes our proof.

%%%%%%%%%%%%%%%%%%%%%%%%%%%%%%%%%%%%%%%%%%%%%%%%%%%%%%%%%%%%%%%%%%%%%%%%%%%%%%%%%%%%%%%%%%%%

\subsection{Proof of Lemma~\ref{LemInfoKey}}
\label{AppLemInfoKey}

Let us first provide a formal statement of 
Lemma~\ref{LemInfoKey} and then prove it.

\begin{lem}

\label{LemInfoKey}
Letting $\eta$ and $\eta'$ denote an i.i.d pair of random variables
drawn from the distribution $\qprob$ defined in
equation~\eqref{EqnPrior}, we have
\begin{align}
\label{EqnUBspicy}
  \Exs_{\eta,\eta'} \exp(\epsilon^2 \inprod{\eta}{\eta'}) & \leq
  \frac{1}{a^2} \exp\left(\frac{5\epsilon^2\ltwo{\Exs \ProjK g}^2
  }{(\Exs \ltwo{\ProjK g})^2} + \frac{40 \epsilon^4 \Exs(\ltwo{\ProjK
      g}^2)}{(\Exs\ltwo{\ProjK g})^4} \right),
\end{align}
where \mbox{$a \defn \Prob(\ltwo{\ProjK g} \geq \frac{1}{2} \Exs
  \ltwo{\ProjK g})$} and $\epsilon > 0$ satisfies the inequality
$\epsilon^2 \leq (\Exs \|\Pi_K g\|_2)^2/32$. 
\end{lem}

% \noindent 

To prove this result, we use Borell's lemma~\cite{borell1975brunn} which states
that for a standard Gaussian vector $Z \sim \NORMAL(0, \EYE{\usedim})$
and a function $f: \mathbb{R}^d \rightarrow \mathbb{R}$ which is $L$-Lipschitz, we
have  
\begin{align}
\label{EqnBorell}
\Exs \exp(a f(Z)) & \leq \exp(a \Exs f(Z) + a^2 L^2/2) 
\end{align}
for every $a \geq 0$. 

Let $g, g'$ be i.i.d standard normal vectors in $\mathbb{R}^d$. Let
\begin{equation*}
  \Aevent(g) \defn \{\ltwo{\ProjK g} > 
\frac{1}{2} \Exs \ltwo{\ProjK g}\} \text{ and } \Aevent(g') \defn \{\ltwo{\ProjK g'} > 
\frac{1}{2} \Exs \ltwo{\ProjK g'}\} 
\end{equation*}
By definition of the probability measure $\qprob$ in expression
\eqref{EqnPrior}, we have 
\begin{align*}
\Exs_{\eta,\eta'} \exp(\epsilon^2 \inprod{\eta}{\eta'})
&= \Exs_{g,g'} \Bigg[ \exp \left(\frac{4\epsilon^2 \inprod{\ProjK g}{\ProjK
    g'}}{\Exs \ltwo{\ProjK g} \Exs \ltwo{\ProjK g'}} 
    \right) \Bigm \vert  \Aevent(g)\cap \Aevent(g') \Bigg]\\
&= \frac{1}{\Prob(\Aevent(g)\cap \Aevent(g'))}
    \Exs_{g,g'} \exp \left( \frac{4\epsilon^2\inprod{\ProjK g}{\ProjK
    g'}}{\Exs \ltwo{\ProjK g} \Exs \ltwo{\ProjK g'}} \right) \Ind(\Aevent(g)\cap \Aevent(g')).
\end{align*}
Using the independence of $g,g'$ and non-negativity of the exponential
function, we have  
\begin{align}
\label{EqnAllegro}
  \Exs_{\eta,\eta'} \exp(\epsilon^2 \inprod{\eta}{\eta'})
  \leq 
  \frac{1}{\Prob(\Aevent(g))^2}
  \underbrace{\Exs_{g,g'} \exp \left( \frac{4\epsilon^2\inprod{\ProjK g}{\ProjK
    g'}}{\Exs \ltwo{\ProjK g} \Exs \ltwo{\ProjK g'}} \right)}_{\defn \TERMONE}.
\end{align}
To simplify the notation, we write $\lambda \defn
4\epsilon^2/(\Exs \ltwo{\ProjK g})^2$ so that  % then by assumption $\lambda^2 < 1/8.$ We thus have 
\begin{align}
\label{EqnAdagio}
  \TERMONE &= \Exs_{g,g'} \exp \left( \lambda \inprod{\ProjK g}{\ProjK g'} \right).
\end{align}
Now for every fixed value of $g$, the function $h \mapsto
\inprod{\ProjK g}{\ProjK h}$ is Lipschitz with Lipschitz constant
equal to $\ltwo{\ProjK g}$. This is because 
\begin{align*}
  |\inprod{\ProjK g}{\ProjK h} - \inprod{\ProjK g}{\ProjK h'}|
   \leq  \ltwo{\ProjK g} \ltwo{\ProjK h - \ProjK h'}\leq  \ltwo{\ProjK g}\ltwo{h-h'},
\end{align*}
where we used Cauchy-Schwartz inequality and the non-expansive
property of convex projection. As a consequence of
inequality~\eqref{EqnBorell} and Cauchy-Schwartz inequality, the term 
$\TERMONE$ can be upper bounded as  
\begin{align}
\label{EqnAdagioCont}
\notag  \TERMONE &\leq 
  \Exs_{g} \exp\left(\lambda\inprod{\ProjK g}{\Exs\ProjK g'} + \frac{\lambda^2 \ltwo{\ProjK g}^2}{2}\right) \\
  &\leq 
  \underbrace{
  \sqrt{\Exs_{g} \exp\left(2\lambda\inprod{\ProjK g}{\Exs\ProjK g'} \right)}}_{\defn \TERMTWO} \, 
  \underbrace{
  \sqrt{\Exs_{g}\exp \left(\lambda^2 \ltwo{\ProjK g}^2\right)}}_{\defn \TERMTHREE}.
\end{align}
We now control $T_2,T_3$ separately. For $T_2$, note again that $h
\mapsto \inprod{\ProjK h}{\Exs \ProjK g'}$ is a Lipschitz function
with Lipschitz constant equal to $\ltwo{\Exs \ProjK
  g'}$. Inequality~\eqref{EqnBorell} implies therefore that 
\begin{align}
\label{EqnMinuet}
  \TERMTWO \leq \sqrt{\exp\left(2\lambda\inprod{\Exs \ProjK
  g}{\Exs\ProjK g'} + 2\lambda^2\ltwo{\Exs\ProjK g'}^2\right)}. 
\end{align}
To control quantity $T_3$, we use a result from \cite[Sublemma
E.3]{amelunxen2014living} on the moment generating function of
$\|\ProjK g\|^2$ which gives  
\begin{align}
\label{Rondo}
  \TERMTHREE \leq 
  \sqrt{ \exp \left(\lambda^2 \Exs (\ltwo{\ProjK g}^2) + 
  \frac{2\lambda^4 \Exs (\ltwo{\ProjK g}^2)}{1 - 4\lambda^2} \right)},
  \qquad{\mbox{whenever } \lambda < 1/4}.
\end{align}
Because of the assumption that $\epsilon^2 \leq (\Exs \|\Pi_K
g\|_2)^2/32$, we have $\lambda \leq 1/8 < 1/4$. Therefore putting all
the pieces together as above, we obtain
\begin{align*}
  \Exs_{\eta,\eta'} \exp(\epsilon^2 \inprod{\eta}{\eta'})
  &\leq  
  \frac{1}{\Prob(\Aevent(g))^2} \exp\left( (\lambda + \lambda^2)
    \ltwo{\Exs \ProjK g}^2 + \frac{\lambda^2\Exs(\ltwo{\ProjK
    g}^2)}{2} + \frac{\lambda^4 \Exs (\ltwo{\ProjK g}^2)}{1 -
    4\lambda^2} \right)\\ 
  &\leq  
  \frac{1}{\Prob(\Aevent(g))^2} \exp\left( 1.25\lambda \ltwo{\Exs \ProjK g}^2 + 
  2.5 \lambda^2 \Exs(\ltwo{\ProjK g}^2) \right)\\
&=
  \frac{1}{\Prob(\Aevent(g))^2}
  \exp\left(\frac{5\epsilon^2\ltwo{\Exs \ProjK g}^2 }{(\Exs (\ltwo{\ProjK g}^2)} + 
  \frac{40 \epsilon^4 \Exs(\ltwo{\ProjK g}^2)}{(\Exs\ltwo{\ProjK
  g})^4}) \right). 
\end{align*}
This completes the proof of inequality~\eqref{EqnUBspicy}.

%%%%%%%%%%%%%%%%%%%%%%%%%%%%%%%%%%%%%%%%%%%%%%%%%%%%%%%%%%%%%%%%%%%%%%%%%%%%%%%%%%

\section{Auxiliary proofs for Proposition 2 and the monotone cone}
\label{Applemmas}

In this appendix, we collect various results related to the monotone
cone, and the proof of Proposition~\ref{PropMon}.

%%%%%%%%%%%%%%%%%%%%%%%%%%%%%%%%%%%%%%%%%%%%%%%%%%%%%%%%%%%%%%%%%%%%%%%%%%%

\subsection{Proof of Lemma~\ref{LemMonInnProCal}}
\label{AppLemMonInnProCal}

So as to simplify notation, we define $\xi = \ProjK g$, with $j^{th}$
coordinate denoted as $\xi_j$.  Moreover, for a given vector $g \in
\real^d$ and integers $1 \leq u < v \leq \usedim$, we define the $u$
to $v$ average as
\begin{align*}
  \bar{g}_{uv} \defn \frac{1}{v - u + 1} \sum_{j = u}^{v} g_j.
\end{align*}
To demonstrate an upper bound for the inner product $\inf
\limits_{\eta \in \Kcone \cap \Sphere{\usedim}}
\inprod{\eta}{\Exs\ProjK g}$, it turns out that it is enough to take
$\eta = \frac{1}{\sqrt{2}}(-1,1, 0, \ldots, 0) \in \Kcone \cap
\Sphere{\usedim}$ and uses the fact that
\begin{align}
\label{EqnLem1Main}
 \inf \limits_{\eta \in \Kcone \cap \Sphere{\usedim}}
 \inprod{\eta}{\Exs\ProjK g} \leq \frac{1}{\sqrt{2}}\Exs(\xi_2 -
 \xi_1).
\end{align}
So it is only left for us to analyze $\Exs(\xi_2 - \xi_1)$ which
actually has an explicit form based on the explicit representation of
projection to the monotone cone (see Robertson et
al.~\cite{robertson1988order}, Chapter 1) where
\begin{align}
\label{EqnProjToMon-expression}
\xi_i = \lambda_i - \bar{\lambda},
\qquad \lambda_i = \max_{u \leq j} \min_{v \geq j} \bar{g}_{uv}.
\end{align}
This is true because projecting to cone $\Kcone = \Mon \cap
\LinSpacePerp$ can be written into two steps $\ProjK g = \projLperp
(\Pi_{\Mon} g)$ and projecting to subspace $\LinSpacePerp$ only shifts
the vector to be mean zero.

We claim that the difference satisfies
\begin{align}
\label{EqnTmp1Lem1}
\xi_2 - \xi_1 \leq \max_{v \geq 2} |\bar{g}_{2 v}| + \max_{v \geq 1}
|\bar{g}_{1 v}|.
\end{align}
To see this, as a consequence of expression
\eqref{EqnProjToMon-expression}, we have
\begin{align*}
\xi_2 - \xi_1 &= \max \{\min_{v \geq 2} \bar{g}_{1v},~ \min_{v \geq 2}
\bar{g}_{2v}\} - \min_{v \geq 1} \bar{g}_{1v}.
\end{align*}
The right hand side above only takes value in set $\{\min_{v \geq 2}
\bar{g}_{1v} - g_1,~0, ~\min_{v \geq 2} \bar{g}_{2v} - \min_{v\geq
  1}\bar{g}_{1v}\}$ where the last two values agree with
bound~\eqref{EqnTmp1Lem1} obviously while the first value can be
written as
\begin{align*}
  \min_{v \geq 2} \bar{g}_{1v} - g_1 = \min_{v \geq 2}
  \left(\frac{1}{v}\sum_{i=2}^v g_{i} - (1-\frac{1}{v}) g_1\right) =
  \min_{v \geq 2} (1- \frac{1}{v}) (\bar{g}_{2v} - g_1) \leq
  |\bar{g}_{2 v}| + |g_1|,
\end{align*}
which also agrees with inequality~\eqref{EqnTmp1Lem1}.

Next let us prove that for every $j = 1, 2$, we have
\begin{align}
  \label{EqnTmp2Lem1}
  \Exs \max_{v\geq j}|\bar{g}_{j v}| < 20 \sqrt{2},
\end{align}
and combine this fact with expressions~\eqref{EqnTmp1Lem1}
and~\eqref{EqnLem1Main} gives us $\inf \limits_{\eta \in \Kcone \cap
  \Sphere{\usedim}} \inprod{\eta}{\Exs\ProjK g} \leq 40$ which
validates the conclusion in Lemma~\ref{LemMonInnProCal}.

It is only left for us to verify inequality~\eqref{EqnTmp2Lem1}.
First as we can partition the interval $[j,\usedim]$ into $k$ smaller
intervals where each smaller interval is of length $2^m$ except the
last one, then
\begin{align}
  \label{EqnTmp3Lem1}
  \Exs \max_{j \leq v \leq \usedim} |\bar{g}_{j v}| &= \Exs
  \max_{1\leq m \leq k} \, \max_{v \in I_m} |\bar{g}_{jv}| \leq
  \sum_{m=1}^k \Exs \max_{v \in I_k}|\bar{g}_{jv}|,
\end{align}
where $I_m = [2^m+j-2, 2^{m+1}+j-3], ~1\leq m < k$, the number of
intervals $k$ and length of $I_k$ are chosen to make those intervals
sum up to $\usedim.$

Given index $2^m+j-2 \leq v \leq 2^{m+1}+j-3$, random variables
$\bar{g}_{jv}$ are Gaussian distributed with mean zero and variance
$1/(v-j+1)$.  Suppose we have Gaussian random variable $X_v$ with mean
zero and variance $\sigma_m^2 = 1/(2^m-1)$ and the covariance
satisfies $\cov(X_v, X_{v'}) = \cov(\bar{g}_{jv}, \bar{g}_{jv'})$.  
Since $\sigma_m^2 \geq 1/(v-j+1),$ the variable $\max_{v \in I_m} |\bar{g}_{jv}|$ is
stochastically dominated by the maximum $\max_{2^m \leq v \leq
  2^{m+1}-1} |X_v|$, and therefore
\begin{align*}
\sum_{m=1}^k \Exs \max_{v \in I_m} |\bar{g}_{jv}| \leq \sum_{m=1}^k
\Exs \max_{2^m \leq v \leq 2^{m+1}-1} |X_v|.
\end{align*}
Applying the fact that for $t\geq 2$ number of Gaussian random
variable $\epsilon_i \sim \NORMAL(0,\sigma^2)$, we have $\Exs \max_{1\leq i \leq t}
|\epsilon_i| \leq 4 \sigma \sqrt{2\log t} $ which gives
\begin{align}  \label{EqnTmp4Lem1}
\sum_{m=1}^k \Exs \max_{v \in I_m} |\bar{g}_{jv}| \leq \sum_{m=1}^k
4\sigma_m \sqrt{2\log(2^m)} = 4\sqrt{2\log 2} \left(\sum_{m=1}^k
\sqrt{\frac{m}{2^m-1}}\right).
\end{align}
The last step is to control the sum $\sum_{m=1}^k
\sqrt{\frac{m}{2^m-1}}$. There are many ways to show that it is upper
bounded by some constant.  One crude way is use the fact that
$\frac{\sqrt{m}}{2^{m}-1} \leq 2^{m/4}$ whenever $m \geq 5,$ therefore
we have
\begin{align*}
\sum_{m=1}^k \sqrt{\frac{m}{2^m-1}} \,=\, \sum_{m=1}^4
\sqrt{\frac{m}{2^{m}-1}} + \sum_{m=5}^k \sqrt{\frac{m}{2^{m}-1}} & <
\sum_{m=1}^4 \sqrt{\frac{m}{2^{m}-1}} + \sum_{m=5}^k
\frac{1}{2^{m/4}}\\
& < \sum_{m=1}^4 \sqrt{\frac{m}{2^{m}-1}} + \frac{2^{-5/4}}{1 -
  2^{-1/4}} < 6,
\end{align*}
which validates inequality~\eqref{EqnTmp2Lem1} when combined with
inequalities \eqref{EqnTmp3Lem1} and \eqref{EqnTmp4Lem1}.  This
completes the proof of Lemma~\ref{LemMonInnProCal}.

%%%%%%%%%%%%%%%%%%%%%%%%%%%%%%%%%%%%%%%%%%%%%%%%%%%%%%%%%%%%%%%%%%%%%%%%%%%%%%%%%%

\subsection{Proof of Lemma~\ref{LemSupport}}
\label{sec:LemSupport}

The proof of Lemma~\ref{LemSupport} involves two parts.  First, we
define the matrices $\OldAmat, \OldLmat$. Then we prove that the
distribution of $\eta$ has the right support where we make use of
Lemma~\ref{LemBeta-to-Eta}.

As stated, matrix $\OldAmat$ is a lower triangular matrix
satisfying~\eqref{EqnOldAmat}. Let us now specify the matrix
$\OldLmat$.  Recall that we denote $\delta \defn r^{-2}$ and $r \defn
1/3.$ To define matrix $\OldLmat$, let us first define a partition of
$[\usedim]$ into $m$ consecutive intervals $\big \{I_{1}, \ldots,
I_{m} \big \}$ with $m$ specified in expression~\eqref{EqnM} and the
length of each interval $|I_i| = \ell_i$ where $\ell_i$ is defined as
\begin{align}\label{EqnPartition}
  \ell_i \defn \lfloor \frac{\delta-1}{\delta^i} (\usedim + 
   \log_\delta \usedim + 3) \rfloor,
  \qquad
  1\leq i\leq m-1,
\end{align}
and $\ell_{m} \defn \usedim - \sum_{i=1}^{m-1} \ell_i.$

Following directly from the definition~\eqref{EqnPartition}, each length $\ell_i \geq 1$ and $\ell_i$ is a decreasing sequence with regard to $i$.
Also $\ell_i$ satisfies the following
\begin{align}
 \label{EqnConseductive}
  \ell_1 = \lfloor \frac{\delta-1}{\delta} (\usedim + \log_\delta \usedim + 3) \rfloor < \usedim 
  \qquad \text{ and }~~ 
  \ell_i \geq \delta \ell_{i+1}, \text{ for } 1\leq i\leq m-1,
\end{align}
where the first inequality holds since as $\sqrt{\log (e\usedim)} \geq 14$, we have $(\delta-1)(\log_\delta \usedim + 3) \leq \usedim$ and the last inequality follows from the fact that 
$\lfloor ab \rfloor \geq a \lfloor b \rfloor$ for positive integer $a$ and 
$b\geq 0$ (because $a\lfloor b \rfloor$ is an integer that is smaller than $ab$).

We are now ready to define the $\usedim \times m$ matrix $\OldLmat$. We take 
\begin{align}
\label{EqnOldLmat}
  \OldLmat(i,j) = 
  \begin{cases}
    \frac{1}{\sqrt{\ell_j}} & i\in I_j,\\
    0 & \text{otherwise.}
  \end{cases}
\end{align}
It is easy to check that matrix $\OldLmat$ satisfies $\OldLmat^T \OldLmat = \Ind_m$ which validates inequality~\eqref{EqnFF-ind}.

% So it is only left for us to verify that under the choice of
% matrices $\OldLmat$ and $\OldAmat$, both $\eta$ and $\eta -
% \bar{\eta} \ONES$ are supported on $\Mon\cap\LinSpace^T\cap
% \NewBall^c(1)$ to complete the proof of Lemma~\ref{LemSupport}.  To
% show these facts, we will make use of Lemma~\ref{LemBeta-to-Eta}.

First we show that both $\eta = \OldLmat \OldAmat b$ and $\eta - \bar{\eta} \ONES$ belong to $\Mon.$
The $i$-th coordinate of $\eta$ can be written as 
\begin{align*}
\eta_{i} = \frac{1}{\sqrt{\ell_j}} \sum_{t=1}^{j} r^{j-t}b_t,
\qquad \forall~i \in I_{j}.
\end{align*}
Therefore we can denote $u_j$ as the value of $\eta_i$ for $i \in I_j$.
To establish monotonicity, we only need to compare the value in the consecutive blocks. 
Direct calculation of the consecutive ratio yields 
\begin{align*}
   \frac{u_{j+1}}{u_j}
   =
   \frac{r (\sum_{t=1}^{j} r^{j-t} b_t) + b_{j+1}}{\sqrt{\ell_{j+1}}}
   \frac{\sqrt{\ell_j}}{\sum_{t=1}^j r^{j-t}b_t}
   \geq
   r \sqrt{\frac{\ell_j}{\ell_{j+1}}} \geq 1,
\end{align*}
where we used the non-negativity of coordinates of vector $b$ and the last inequality follows from inequality~\eqref{EqnConseductive} and $\delta = r^{-2}.$
The monotonicity of $\eta - \bar{\eta}\ONES$ thus inherits directly from the monotonicity of $\eta.$

To complete the proof of Lemma~\ref{LemSupport}, we only need to prove
lower bounds on $\ltwo{\eta}$ and $\ltwo{\eta - \bar{\eta}}$. For
these, we shall use inequality~\eqref{EqnNorm} of
Lemma~\ref{LemBeta-to-Eta}.

\paragraph{Proof of the bound $\ltwo{\eta} \geq 1$:}
Recall that $r = 1/3$ and as a direct consequence of
inequality~\eqref{EqnNorm} in Lemma~\ref{LemBeta-to-Eta}, we have
\begin{align} 
\label{EqnNormEta}
\inprod{\eta}{\eta} = \ltwo{\OldAmat b}^2 ~\geq~ \frac{9}{4} -
\frac{63}{32 s} > 1.96,
 \end{align}
where the last step follows form the fact that $s = \lfloor \sqrt{m}
\rfloor \geq 7$.  Therefore, the norm condition holds so $\eta$ is
supported on $\Mon\cap\LinSpace^T\cap\NewBall^c(1)$.

\paragraph*{Proof of the bound $\ltwo{\eta - \bar{\eta}\ONES} \geq 1$:} 
The norm $\ltwo{\eta - \bar{\eta}\ONES}^2$ has the following
decomposition where
\begin{align*}
  \ltwo{\eta - \bar{\eta}\ONES}^2 = \ltwo{\eta}^2 -
  \usedim(\bar{\eta})^2.
\end{align*}
We claim that $\usedim(\bar{\eta})^2 \leq 0.2$.  If we take this for
now, combining with inequality~\eqref{EqnNormEta} which says
$\ltwo{\eta}^2$ is greater than $1.96$, we can deduce that $\ltwo{\eta
  - \bar{\eta}\ONES}^2 \geq 1.$ So it suffices to verify the claim
$\usedim(\bar{\eta})^2 \leq 0.2$.  Recall that $\eta = \OldLmat
\OldAmat b$.  Direct calculation yields
\begin{align*}
  \usedim \bar{\eta} = \inprod{\ONES}{\eta} = \ONES^T \cdot \OldLmat
  \OldAmat b = \sum_{k=1}^m b_k \underbrace{ \sum_{i=k}^m
    \sqrt{\ell_i}r^{i-k}}_{\defn a_k}.
\end{align*}
Plugging into the definitions of $r$ and $\ell_i$ guarantees that
\begin{align*}
a_k \leq \sum_{i=k}^m \sqrt{ \frac{(\delta-1)(\usedim + \log_\delta
    \usedim + 3)}{\delta^i}} \frac{1}{\delta^{(i-k)/2}} &=
\sqrt{(\delta-1)(\usedim + \log_\delta \usedim + 3)\delta^{k}}
\sum_{i=k}^m \delta^{-i} \\
&\leq \sqrt{\frac{(\usedim + \log_\delta
    \usedim + 3)}{(\delta-1)\delta^{k-2}}},
\end{align*}
where the last step uses the summability of a geometric
sequence---namely $\sum_{i=k}^m \delta^{-i} \leq
\delta^{-k+1}/(\delta-1).$ Now for every vector $b$, our goal is to
control $\sum a_kb_k.$ Recall that every vector $b$ has $s$ non-zero
entries which equal to $1/\sqrt{s}$ where $s = \lfloor \sqrt{m}
\rfloor.$ Since $a_k$ decreases with $k$, this inner product $\sum
a_kb_k$ is largest when the first $s$ coordinates of $b$ are non-zero,
therefore
\begin{align*}
  \usedim \bar{\eta} \leq\sum_{k=1}^s a_k \frac{1}{\sqrt{s}} \leq
  \frac{1}{\sqrt{s}} \sqrt{\frac{\delta^2(\usedim + \log_\delta
      \usedim + 3)}{\delta-1}} \sum_{k=1}^s\frac{1}{\delta^{k/2}} \leq
  \frac{1}{\sqrt{s}} \sqrt{\frac{\delta^2(\usedim + \log_\delta
      \usedim + 3)}{\delta-1}} \frac{1}{\sqrt{\delta} - 1},
\end{align*}
and thus we have
\begin{align*}
  d (\bar{\eta})^2 \leq \frac{1}{\sqrt{m}-1} \frac{(\usedim +
    \log_\delta \usedim + 3)}{\usedim}
  \frac{\delta^2}{(\delta-1)(\sqrt{\delta} - 1)^2} \leq
  \frac{81(\usedim + \log_\delta \usedim + 3)}{32 \usedim(\sqrt{m} -
    1)}< 0.2,
\end{align*}
where the last step uses $\sqrt{m} \geq 8$. Therefore, the norm
condition also holds so $\eta - \bar{\eta} \ONES$ is supported on
$\Mon\cap\LinSpace^T\cap\NewBall^c(1)$.

Thus, we have completed the proof of Lemma~\ref{LemSupport}.

%%%%%%%%%%%%%%%%%%%%%%%%%%%%%%%%%%%%%%%%%%%%%%%%%%%%%%%%%%%%%%%%%%%%%%%%%%%%%%%

\subsection{Proof of Lemma~\ref{LemBeta-to-Eta}}
\label{AppLemBeta-to-Eta}

By definition of the matrix $\OldAmat$, we have
\begin{align*}
  \inprod{\OldAmat b}{\OldAmat b'} = \sum_{t=1}^m (\OldAmat b)_t
  (\OldAmat b')_t
  &= \sum_{t=1}^m (b_t + rb_{t-1} + \cdots + r^{t-1}b_1)(
  b'_t + rb'_{t-1} + \cdots + r^{t-1}b'_1) \\
  &= \sum_{t=1}^m \sum_{u=1}^t \sum_{v=1}^t r^{2t-u-v} b_u b'_v.
\end{align*}
Switching the order of summation yields
\begin{align}
 \label{EqnGb}
  \notag \inprod{\OldAmat b}{\OldAmat b'} &= \sum_{u=1}^m \sum_{v=1}^m
  b_u b'_v \sum_{t=\max\{u,v\}}^m r^{2t-u-v} \\
  \notag &= \sum_{u=1}^m \sum_{v=1}^m \frac{b_u b'_v}{r^{u+v}}
  \frac{r^{2\max\{u,v\}} - r^{2m+2}}{1-r^2} \\
  &= \underbrace{\frac{1}{1-r^2}\sum_{u=1}^m \sum_{v=1}^m b_u b'_v
    r^{|u-v|}}_{\defn \Delta_1 } -
  \underbrace{\frac{1}{1-r^2}\sum_{u=1}^m \sum_{v=1}^m b_u b'_v
    r^{2m+2-u-v}}_{\defn \Delta_2}.
\end{align}
We bound the two terms $\Delta_1$ and $\Delta_2$ separately.

Recall the fact that $b, b'$ belong to $\KposSet$, so there are exactly $s = \lfloor \sqrt{m} \rfloor$ non-zero entry in both $b$ and $b'$ and these entries equal to $1/\sqrt{s}$. The summation defining
$\Delta_1$ is not affected by the permutation of coordinates, so that
we can assume without loss of generality that the indices of non-zero
entries in $b$ are indexed by $\{1, \ldots,s \}$, and that the indices
of non-zero entries in $b'$ are indexed by $\{k, k+1,\ldots, k+s-1
\}$ for some $1\leq k \leq m+1-s.$

We split our proof into two cases depending on whether $k \leq s$ or $k> s.$

\paragraph{Case 1 ($k \leq s$):} The summation $\Delta_1$ can be written as
\begin{align*}
  s(1-r^2)\Delta_1 = s\sum_{u=1}^m \sum_{v=1}^m b_u b'_v r^{|u-v|}
  = \sum_{u=1}^s \sum_{v=k}^{k+s-1}  r^{|u-v|}.
\end{align*}
Direct calculation yields 
\begin{align*}
  s(1-r^2)\Delta_1 &= \sum_{u=1}^{k-1} \sum_{v=k}^{k+s-1}  r^{v-u} + \sum_{u=k}^s \sum_{v=k}^{u} r^{u-v} + \sum_{u=k}^s \sum_{v=u+1}^{k+s-1} r^{v-u} \\
  &= \frac{(1-r^s)(r-r^k)}{(1-r)^2} + \frac{s-k+1}{1-r}
  - \frac{r}{(1-r)^2}(1-r^{s-k+1}) + \frac{r(s-k+1)}{1-r} - 
  \frac{r^k-r^{s+1}}{(1-r)^2}\\
  &= \frac{1+r}{1-r}(s-k+1) + \frac{r^k(r^s+r^{s+2}-2)}{(1-r)^2}.
\end{align*}
Notice the following two facts that 
\begin{align*}
  \inprod{b}{b'} = \frac{s-k+1}{s}
  \qquad
  \text{ and }~
  \frac{-2r}{(1-r)^2} \leq 
  \frac{r^k(r^s+r^{s+2}-2)}{(1-r)^2} < 0,
\end{align*}
so that 
\begin{align}
\label{EqnGoodBound}
  \frac{1}{(1-r)^2}\inprod{b}{b'} + \frac{-2r}{s(1-r^2)(1-r)^2} \leq
  \Delta_1 \leq \frac{1}{(1-r)^2}\inprod{b}{b'}.
\end{align}

\paragraph{Case 2 ($k > s$):} The summation $\Delta_1$ satisfies
the bounds
\begin{align*}
  s(1-r^2)\Delta_1 &= s\sum_{u=1}^m \sum_{v=1}^m b_u b'_v r^{|u-v|}
  = \sum_{u=1}^s \sum_{v=k}^{k+s-1} r^{v-u}
  =  \frac{r^{k-s}(1-r^s)^2}{(1-r)^2}.
\end{align*}
Since $k-s \geq 1$, we have $\inprod{b}{b'} = 0$ and consequently
\begin{align}
  \label{EqnGoodBoundII}
\Delta_1 \leq \frac{1}{(1-r)^2}\inprod{b}{b'} +
\frac{r}{s(1-r^2)(1-r)^2}.
\end{align}

Combining inequalities \eqref{EqnGb}, \eqref{EqnGoodBound} and
\eqref{EqnGoodBoundII}, we can deduce that
\begin{align*}
\inprod{\OldAmat b}{\OldAmat b'} ~\leq~ \Delta_1 ~\leq~
\frac{1}{(1-r)^2}\inprod{b}{b'} + \frac{r}{s(1-r^2)(1-r)^2},
\end{align*}
which validates inequality~\eqref{EqnGInprod}.

On the other hand, when $b=b'$, the summation $\Delta_2$ is the
largest when the non-zero entries of $b$ lie on coordinates $m-s+1,\ldots,m$.
Thus we have
\begin{align}
\label{EqnGb-partII}
s (1-r^2) \Delta_2 \leq \sum_{u=m-s+1}^m \sum_{v=m-s+1}^m r^{2m+2-u-v}
= \frac{r^2(1-r^s)^2}{(1-r)^2} < \frac{r^2}{(1-r)^2}.
\end{align}
Combining decomposition~\eqref{EqnGb} with the
inequalities~\eqref{EqnGoodBound}, we can deduce that
\begin{align*}
\inprod{\OldAmat b}{\OldAmat b} \leq \frac{1}{(1-r)^2} - \frac{2r}{s
  (1-r^2)(1-r)^2} - \frac{r^2}{s(1-r^2)(1-r)^2},
\end{align*}
where we use the fact that $\inprod{b}{b} = 1$. This completes the
proof of inequality~\eqref{EqnNorm}.

\end{document}